\documentclass[11pt]{article}

\usepackage{amsmath, amsthm, amssymb, amsfonts, listings, hyperref, multicol, bm, shuffle, xcolor, enumerate, mathscinet, dsfont, tikz, dsfont, tabu}

\usepackage{tikz}
\usetikzlibrary{decorations.pathreplacing}
\usepackage{tikz}

\usepackage[margin=1.2in]{geometry}
\usepackage{fancyhdr}
\theoremstyle{definition}
\newtheorem{theorem}{Theorem}[section]

\newtheorem{conjecture}[theorem]{Conjecture}
\newtheorem{lemma}[theorem]{Lemma}
\newtheorem{remark}[theorem]{Remark}
\newtheorem{proposition}[theorem]{Proposition}
\newtheorem{definition}[theorem]{Definition}
\newtheorem{example}[theorem]{Example}

\usetikzlibrary{decorations.pathreplacing,shapes}
\usepackage{tikz}

\DeclareMathOperator{\Irr}{Irr}                                                                      
\DeclareMathOperator{\ch}{\text{supp}}                                                        
\DeclareMathOperator{\Char}{Char}                                                             
\DeclareMathOperator{\seq}{Seq}                                                                
\newcommand{\phcyc}[2]{\ensuremath{\p^{\Lambda_{#1}}_{#2}}}              
\newcommand{\phcycall}[1]{\ensuremath{\p^{\Lambda_{#1}}}}                    
\newcommand{\ep}[1]{\varepsilon_{#1}}                                                       
\newcommand{\epch}[1]{\varepsilon^\vee_{#1}}                                           
\newcommand{\w}{\widetilde}                                                                       
\newcommand{\Li}{L}                                                                                    
\newcommand{\Ti}{T}                                                                                    
\newcommand{\Tii}[2]{\ensuremath{\Ti_{#1;#2}}}                                          
\newcommand{\Tiip}[1]{\ensuremath{\Ti_{#1}}}                                             
\newcommand{\Sign}{S}                                                                               
\newcommand{\Sii}[2]{\ensuremath{\Sign_{#1;#2}}}                                      
\DeclareMathOperator{\Repr}{Rep}                                                              
\newcommand{\rep}[1]{\Repr \ensuremath{^{#1}}}                                        
\newcommand{\repL}[1]{\rep{{\Lambda_{#1}}} }
\newcommand{\Bkr}{B^{1,1}}
\newcommand{\Rcal}[1]{\mathcal{R}(#1)}                                                     
\newcommand{\Rcalj}[2]{\mathcal{R}_{#1}(#2)}                                            
\DeclareMathOperator{\pro}{pr}                                                                    
\newcommand{\pr}[1]{\pro \ensuremath{_{\Lambda_{#1}}}}                          
\DeclareMathOperator{\infltxt}{infl}                                                                
\newcommand{\infl}[1]{\infltxt \ensuremath{_{\Lambda_{#1}}}}                     
\newcommand{\e}[1]{\ensuremath{e_{#1}}}                                                   
\newcommand{\etil}[1]{\ensuremath{\w{e}_{#1}}}                                          
\newcommand{\ech}[1]{\ensuremath{e^\vee_{#1}}}                                      
\newcommand{\etilch}[1]{\ensuremath{\w{e}^\vee_{#1}}}                              
\newcommand{\ftil}[1]{\ensuremath{\w{f}_{#1}}}                                            
\newcommand{\ftilch}[1]{\ensuremath{\w{f}^\vee_{#1}}}                               
\newcommand{\und}[1]{\underline{\mathbf{#1}}}                                                   
\newcommand{\jump}[1]{\text{jump}_{#1}}                                                    
\newcommand{\wti}[1]{\text{wt}_{#1}}                                                           
\newcommand{\Sy}[1]{\ensuremath{\mathcal{S}_{#1}}}                               
\newcommand{\cycloI}[1]{\ensuremath{\mathcal{I}^{#1}}}                            
\newcommand{\maps}{\colon}                                                                       
\def\Shuffle{\,\raise 1pt\hbox{$\scriptscriptstyle\cup{\mskip                           
               -4mu}\cup$}\,}
\newcommand{\supp}[1]{\text{supp}(\ensuremath{#1})}                                 
\newcommand{\gammaplus}[2]{\ensuremath{\gamma_{#1;#2}}}                
\newcommand{\gammaminus}[2]{\ensuremath{\gamma^-_{#1;#2}}}              
\newcommand{\Lii}[1]{\ensuremath{\Li(#1)}}                                                   
\newcommand{\crystalmap}{\mathcal{T}}                                                       
\newcommand{\cryphi}[1]{\p_{#1}}                                                                  

\newcommand{\adjac}[3]{\eta_{#1}^{#2}(#3)}                                                          
\newcommand{\trivseq}[2]{\tau_{#1;#2}}                                                       
\newcommand{\signseq}[2]{\tau^-_{#1;#2}}                                                      
\newcommand{\Nextterm}[2]{\p(#1;#2)}                                                              
\newcommand{\Nexttermplus}[2]{\p^{+}(#1;#2)}                                              
\newcommand{\Nexttermminus}[2]{\p^{-}(#1;#2)}                                              
\newcommand{\phcyctriv}[3]{\p_{#1}(#2;#3)}                                                    
\newcommand{\phcyctrivminus}[3]{\p_{#1}^-(#2;#3)}                                           
\newcommand{\phcyctrivplus}[3]{\p_{#1}^+(#2;#3)}                                           
\newcommand{\pextplus}[1]{\text{ext}^+_{#1}}                                                
\newcommand{\pextminus}[1]{\text{ext}^-_{#1}}                                                
\newcommand{\exttailby}[1]{\pi(#1) \star }                                                        
\newcommand{\extheadby}[1]{\star \pi(#1) }                                                    
\newcommand{\Lcal}[1]{\mathcal{L}(#1)}                                                         
\newcommand{\ClassA}{\mathcal{A}}                                                              
\newcommand{\ClassB}{\mathcal{B}}                                                              
\newcommand{\ClassD}{\mathcal{D}}                                                              
\newcommand{\redj}{\textcolor{red}{i}}                                                             
\newcommand{\blueh}{\textcolor{blue}{h}}                                                        

\newcommand{\bi}[1]{b_{\Lambda_{#1}}}

\newcommand{\Bp}{\mathcal{B}}                      		

\newcommand{\F}{\mathbb{F}}                            	 
\newcommand{\Z}{\mathbb{Z}}                           	 
\newcommand{\wh}{\widehat}                              	 
\newcommand{\Fr}{\mathfrak}                            		 
\newcommand{\MB}{\mathbb}                            		 
\newcommand{\p}{\varphi}                                 		 
\DeclareMathOperator{\Hom}{Hom}                    	 
\DeclareMathOperator{\ind}{Ind}                          	 
\DeclareMathOperator{\coind}{coInd}                  	 
\DeclareMathOperator{\res}{Res}                        	 
\DeclareMathOperator{\soc}{soc}                        	 
\DeclareMathOperator{\cosoc}{cosoc}                	 
\DeclareMathOperator{\Mod}{\dash \text{mod}}          
\DeclareMathOperator{\wt}{wt}                           		 
\DeclareMathOperator{\HOM}{HOM}                 	 	 
\DeclareMathOperator{\dash}{-}                        		 
\DeclareMathOperator{\UnitModule}{\mathds{1}}         
\newcommand{\zero}{\bm{0}}                                       

\newcommand{\omitt}[1]{}

\usepackage{etex} 

\tikzstyle{new}=[fill =white, circle, inner sep=.5pt]
\tikzstyle{old}=[fill =blue!20, circle, inner sep=.5pt]

\newcommand{\Part}[1]{
 \foreach \xx [count=\ss from 1] in {#1}{
 	{\ifnum\ss=1
		\draw (0,\ss-1)--(\xx,\ss-1); 
		\fi}
   \draw (0,\ss) to (\xx,\ss);
   \foreach \y in {0, ..., \xx} {\draw (\y,\ss)--(\y,\ss-1);}
 }}

 \newcommand{\fPart}[2]{
 \foreach \xx [count=\ss from 1] in {#1}{
  \filldraw[#2] (0,\ss) -- (\xx,\ss)--(\xx,\ss-1)--(0,\ss-1);
 }
 \foreach \xx [count=\ss from 1] in {#1}{
 	{\ifnum\ss=1
		\draw (0,\ss-1)--(\xx,\ss-1); 
		\fi}
   \draw (0,\ss) to (\xx,\ss);
   \foreach \y in {0, ..., \xx} {\draw (\y,\ss)--(\y,\ss-1);}
 }}

\def\UNIT{.2} 

\newcounter{r}
\newcounter{c}

\newcommand\youngDiagram[2]{
        \setcounter{r}{0}
        \foreach \row in {#1} {
                \setcounter{c}{0}
                \foreach \b in \row {
                        \node at (\value{c}*#2 + .5*#2, \value{r}*#2+.5*#2)
{\b};
                        \addtocounter{c}{1}

                }
                \draw[step = {#2}] (0,\value{r}*#2) grid
(\value{c}*#2,\value{r}*#2+#2);
                \addtocounter{r}{-1}
        }
}

\def\Lattice{
	\coordinate (0) at (0,0);
	\coordinate (11) at (0,1.8);
	\coordinate (21) at (-1.2,3.8);\coordinate (22) at (1.2,3.9);
	\coordinate (31) at (-1.2,6.1);\coordinate (32) at (1.2,6.1);
	\foreach \xx in {1, ..., 4}{\coordinate (4\xx) at (-7.1+2.9*\xx,8.5);}

	\draw (0)--(11) (11)--(21) (11)--(22);
	\draw (21)--(31) (22)--(32);
	\draw (31)--(41)  (31)--(42)  (32)--(43)  (32)--(44);
\begin{scope}[every node/.style={fill=white}]
	\node at (0) {$\emptyset$};
	\node at (11) {};
	\node at (21) {};
	\node at (31) {};
	\node at (41) {};
	\node at (22) {};
	\node at (32) {};
	\node at (42) {};
	\node at (43) {};
	\node at (44) {};
\end{scope}
\foreach \xx in {-3,0,3} {\node at (\xx , 9.5) {$\vdots$}; }
}

\def\LatticeMV{
	\coordinate (0) at (0,0);
	\coordinate (11) at (0,1.8);
	\coordinate (21) at (-1.2,3.8);\coordinate (22) at (1.2,3.9);
	\coordinate (31) at (-1.2,6.1);\coordinate (32) at (1.2,6.1);
	\foreach \xx in {1, ..., 4}{\coordinate (4\xx) at (-7.1+2.9*\xx,8.5);}

	\draw (0)--(11) (11)--(21) (11)--(22);
	\draw (21)--(31) (22)--(32);
\begin{scope}[every node/.style={fill=white}]
	\node at (0) {$\emptyset$};
	\node at (11) {};
	\node at (21) {};
	\node at (31) {};
	\node at (22) {};
	\node at (32) {};
\end{scope}
\foreach \xx in {-3,0,3} {\node at (\xx , 7.5) {$\vdots$}; }
}

\begin{document}
\title{Categorifying the tensor product of the Kirillov-Reshetikhin
 crystal $B^{1,1} $ and a fundamental crystal}


\author{Henry Kvinge  and Monica Vazirani}


\date{\today}

\maketitle

\begin{abstract}
We use Khovanov-Lauda-Rouquier (KLR) algebras to categorify a crystal
isomorphism between a fundamental crystal and the tensor product of a
Kirillov-Reshetikhin
crystal and another fundamental crystal, all in affine type.
 The nodes
of the Kirillov-Reshetikhin crystal correspond
to a family of ``trivial" modules.
The nodes of the fundamental crystal
 correspond to simple modules of the corresponding
cyclotomic KLR algebra.
The crystal operators correspond to socle of restriction
and behave compatibly with the rule for tensor product
of crystal graphs.
\end{abstract}

\tableofcontents
\section{Introduction }
\label{sec-intro}

Kang-Kashiwara \cite{KK12} and Webster \cite{Web}
show the cyclotomic
Khovanov-Lauda-Rouquier (KLR)
algebra $R^\Lambda$
categorifies the highest weight representation $V(\Lambda)$
in arbitrary symmetrizable type.
(KLR algebras are also known in the literature as quiver Hecke
algebras.) By a slight abuse of language, we will
say the {\em combinatorial} version of this statement is that
$R^\Lambda$ categorifies the crystal $B(\Lambda)$, where
simple modules correspond to nodes, and functors that take socle
of restriction correspond to arrows, i.e.~the Kashiwara crystal
operators \cite{LV11}. 
Webster \cite{Web} and Losev-Webster \cite{LW} 
categorify the tensor product of highest weight modules,
and hence the tensor product of highest weight crystals.
However, one can consider a tensor product of crystals
\begin{gather}
\label{eq-tensor}
\Bp \otimes B(\Lambda) \simeq B(\Lambda')
\end{gather}
where $\Lambda, \Lambda' \in P^+$ are of level $k$ and
$\Bp$ is a perfect crystal of level $k$.
In this paper, we (combinatorially) categorify the crystal isomorphism 
\eqref{eq-tensor} in the case
$\Lambda = \Lambda_i$ is a fundamental weight and 
$\Bp = \Bkr$ is a Kirillov-Reshetikhin crystal. 
In other words, our main theorems give a purely module-theoretic
construction of this crystal isomorphism.
(One must modify the form of the crystal isomorphism
in the case $\Bkr$ is not perfect or when $\Lambda_i$ is
not of level $1$. See Section \ref{sec-general}.)
Each node of $\Bkr$ corresponds to an infinite
 family of ``trivial" modules,
but note this does {\em not} give a categorification of $\Bp$.
These ``trivial" modules $\Tii{p}{k}$ are the KLR analogues
of the nodes in highest weight crystals studied in \cite{V07}
and are completely described in Section \ref{Description-of-family}.

We note that this gives a construction of simple modules,
starting from the $\Tii{p}{k}$.
In type $A$, this is somewhat intermediate between the crystal operator
construction and the Specht module construction.  See 
\cite{V15} for details. 
\omitt{ 
  Combinatorially, the former corresponds to building a
($\ell$-restricted) partition one (good) box at a time.  Our
construction builds a partition one row at a time, or dually one column
at a time.  The Specht module construction (at least for $\F_\ell
\Sy{n}$ or the Hecke algebra of type $A$) starts from the whole
partition,  building a simple as a subquotient of an induced trivial
module from a parabolic subalgebra.
} 
 This paper also describes how socle of restriction
interacts with the construction.
One can also recover this for finite type
whose Dynkin diagram is a subdiagram of that
of type $X_\ell$ studied here.
For a construction of simple modules
related to the crystal $B(\infty)$
for finite type KLR algebras see \cite{BKOP}.

This paper 
generalizes the theorems and constructions  from \cite{V15}
for type $A$ affine, which were in turn originally
proved
for the affine Hecke algebra of type $A$
at an $(\ell+1)$st root of unity \cite{V03, Vnotes}.


\medskip

The authors welcome input on adding references if the ones
included in this paper are incomplete.

Acknowledgments:
We wish to thank Peter Tingley for interesting discussions and
for pointing out we were using KR crystals and not just level
$1$ perfect crystals. 
The second author would like to thank David Hill
for useful discussions and some initial jump computations.

The first author was partially supported by
NSA grant H98230-12-1-0232, ICERM, and the Simons Foundation.

\section{Background and notation}

\subsection{Cartan datum} \label{Cartan-datum-section}
We first review the Cartan datum associated with types $A^{(1)}_{\ell}$, $C^{(1)}_{\ell}$, $A^{(2)}_{2\ell}$, $A^{(2)\dagger}_{2\ell}$, $D^{(2)}_{\ell +1}$, $D^{(1)}_{\ell}$, $B^{(1)}_{\ell}$, and $A^{(2)}_{2\ell-1}$. Fix an integer $\ell \geq 2$. For each type $X_\ell$ listed above, $I = \{0,1, \dots, \ell\}$ will denote the indexing set.
Let $[a_{ij}]_{i,j \in I}$ denote the associated Cartan matrix. 
We direct the reader to \cite{Kac85} for the explicit matrices.
Following \cite{Kac85} we let $\Fr{h}$ be a Cartan subalgebra, $\prod = \{\alpha_0, \dots, \alpha_{\ell}\}$ its system of simple roots, $\prod^{\vee} = \{h_0, \dots, h_{\ell}\}$ its simple coroots, and $Q$ and $Q^\vee$ the root and coroot lattices respectively. Then set
\begin{equation}
Q^+ = \bigoplus_{i \in I} \mathbb{Z}_{\geq 0} \alpha_i.
\end{equation}
For an element $\nu \in Q^+$, we define its {\emph{height}}, $|\nu|$, to be the sum of the coefficients, i.e. if $\nu = \sum_{i \in I} \nu_i \alpha_i$ then
\begin{equation}
|\nu| = \sum_{i \in I} \nu_i.
\end{equation}
We also have a symmetric bilinear form
\begin{equation}
(\; , \; ) : \Fr{h}^* \times \Fr{h}^* \rightarrow \mathbb{C},
\end{equation}
satisfying the property that 
\begin{equation}
a_{ij} = \langle h_i, \alpha_j \rangle =  
\frac{2 (\alpha_i, \alpha_j) }{(\alpha_i, \alpha_i)}
\end{equation}
where $\langle \; , \; \rangle: \Fr{h} \times \Fr{h}^* \rightarrow \mathbb{C}$ is the canonical pairing. Using this pairing we define the fundamental weights $\{\Lambda_i \; | \; i \in I\}$ via
\begin{equation}
\langle h_i, \Lambda_j \rangle = \delta_{ij}.
\end{equation}
The weight lattice is $ \bigoplus_{i \in I} \MB{Z} \Lambda_i $ and the integral dominant weights are
\begin{equation}
P^+ = \bigoplus_{i \in I} \mathbb{Z}_{\geq 0} \Lambda_i.
\end{equation}
For all types considered in this paper, the associated Lie algebra $\Fr{g}$ has 1-dimensional center generated by the canonical central element $c = c_0h_0 + \dots + c_{\ell}h_{\ell}$ where $c_i \in \MB{Z}_{\geq 0}$. The level of a weight
$\Lambda \in P^+$ is then defined to be $\langle c,\Lambda \rangle$. 

For each type $X_\ell$ we draw below the associated Dynkin diagram and list the level of the fundamental weights. 

\begin{itemize}

\item $A^{(1)}_1, \;\; I = \{0,1\}, \;\; \text{Level 1 weights:} \; \Lambda_0, \Lambda_1,$
\begin{center}
\begin{tikzpicture}
\draw (0,0) circle (.1cm);
\draw (2,0) circle (.1cm);
\node at (0,-.75) {$0$};
\node at (2,-.75) {$1$};
\draw (0.1,.05) -- (1.9,.05);
\draw (0.1,-.05) -- (1.9,-.05);
\draw (0.2,0) -- (0.3,.2);
\draw (0.2,0) -- (0.3,-.2);
\draw (1.8,0) -- (1.7,.2);
\draw (1.8,0) -- (1.7,-.2);
\end{tikzpicture}
\end{center}

\item $A^{(1)}_{\ell}, \;\; \ell \ge 2,
\;\; \text{Level 1 weights:} \; \Lambda_i, \; i \in I,$
\begin{center}
\begin{tikzpicture} [scale=.7]  
\draw (0,0) circle (.15cm);
\node at (0,-.75) {$1$};
\draw (1.5,0) circle (.15cm);
\node at (1.5,-.75) {$2$};
\draw (3,0) circle (.15cm);
\node at (3,-.75) {$3$};
\draw (6,0) circle (.15cm);
\node at (6,-.75) {$\ell-1$};
\draw (7.5,0) circle (.15cm);
\node at (7.5,-.75) {$\ell$};
\draw (3,2) circle (.15cm);
\node at (3.5,2.5) {$0$};
\draw (.13,.071) -- (2.85,1.95);
\draw (.15,0) -- (1.35,0);
\draw (1.65,0) -- (2.85,0);
\draw (3.15,0) -- (3.65,0);
\draw (5.35,0) -- (5.85,0);
\draw (6.15,0) -- (7.35,0);
\draw (7.35,0.071) -- (3.15,1.95);
\node at (4.1,0) {$\cdot$};
\node at (4.5,0) {$\cdot$};
\node at (4.9,0) {$\cdot$};
\end{tikzpicture}
\end{center}

\item $C^{(1)}_{\ell}, \;\; \ell \geq 2,\;\; \text{Level 1 weights:} \; \Lambda_i, i \in I,$
\begin{center}
\begin{tikzpicture} [scale=.7]  
\draw (0,0) circle (.15cm);
\node at (0,-.75) {$0$};
\draw (1.5,0) circle (.15cm);
\node at (1.5,-.75) {$1$};
\draw (3,0) circle (.15cm);
\node at (3,-.75) {$2$};
\draw (6,0) circle (.15cm);
\node at (6,-.75) {$\ell -1$};
\draw (7.5,0) circle (.15cm);
\node at (7.5,-.75) {$\ell$};
\draw (.1,.1) -- (1.4,.1);
\draw (.1,-.1) -- (1.4,-.1);
\draw (1.65,0) -- (2.85,0);
\draw (3.15,0) -- (3.65,0);
\draw (5.35,0) -- (5.85,0);
\draw (6.1,.1) -- (7.4,.1);
\draw (6.1,-.1) -- (7.4,-.1);
\draw (.5,.3) -- (.7,0);
\draw (.5,-.3) -- (.7,0);
\draw (7,.3) -- (6.8,0);
\draw (7,-.3) -- (6.8,0);
\node at (4.1,0) {$\cdot$};
\node at (4.5,0) {$\cdot$};
\node at (4.9,0) {$\cdot$};
\end{tikzpicture}
\end{center}

\item $A^{(2)}_{2\ell}, \;\; \ell \geq 2, \;\; \text{Level 1 weight:} \; \Lambda_{0}, \; \text{Level 2 weights:} \; \Lambda_1, \dots, \Lambda_\ell$
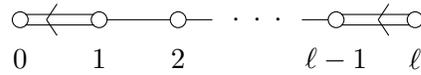
\begin{center}
\begin{tikzpicture} [scale=.7]  
\draw (0,0) circle (.15cm);
\node at (0,-.75) {$0$};
\draw (1.5,0) circle (.15cm);
\node at (1.5,-.75) {$1$};
\draw (3,0) circle (.15cm);
\node at (3,-.75) {$2$};
\draw (6,0) circle (.15cm);
\node at (6,-.75) {$\ell-1$};
\draw (7.5,0) circle (.15cm);
\node at (7.5,-.75) {$\ell$};
\draw (.1,.1) -- (1.4,.1);
\draw (.1,-.1) -- (1.4,-.1);
\draw (1.65,0) -- (2.85,0);
\draw (3.15,0) -- (3.65,0);
\draw (5.35,0) -- (5.85,0);
\draw (6.1,.1) -- (7.4,.1);
\draw (6.1,-.1) -- (7.4,-.1);
\draw (.7,.3) -- (.5,0);
\draw (.7,-.3) -- (.5,0);
\draw (7,.3) -- (6.8,0);
\draw (7,-.3) -- (6.8,0);
\node at (4.1,0) {$\cdot$};
\node at (4.5,0) {$\cdot$};
\node at (4.9,0) {$\cdot$};
\end{tikzpicture}
\end{center}

\item $A^{(2)\dagger}_{2\ell}, \;\; \ell \geq 2, \;\; \text{Level 1 weight:} \; \Lambda_{\ell}, \; \text{Level 2 weights:} \; \Lambda_0, \dots, \Lambda_{\ell-1}$
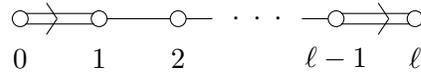
\begin{center}
\begin{tikzpicture} [scale=.7]  
\draw (0,0) circle (.15cm);
\node at (0,-.75) {$0$};
\draw (1.5,0) circle (.15cm);
\node at (1.5,-.75) {$1$};
\draw (3,0) circle (.15cm);
\node at (3,-.75) {$2$};
\draw (6,0) circle (.15cm);
\node at (6,-.75) {$\ell-1$};
\draw (7.5,0) circle (.15cm);
\node at (7.5,-.75) {$\ell$};
\draw (.1,.1) -- (1.4,.1);
\draw (.1,-.1) -- (1.4,-.1);
\draw (1.65,0) -- (2.85,0);
\draw (3.15,0) -- (3.65,0);
\draw (5.35,0) -- (5.85,0);
\draw (6.1,.1) -- (7.4,.1);
\draw (6.1,-.1) -- (7.4,-.1);
\draw (.5,.3) -- (.7,0);
\draw (.5,-.3) -- (.7,0);
\draw (6.8,.3) -- (7,0);
\draw (6.8,-.3) -- (7,0);
\node at (4.1,0) {$\cdot$};
\node at (4.5,0) {$\cdot$};
\node at (4.9,0) {$\cdot$};
\end{tikzpicture}
\end{center}

\item $D^{(2)}_{\ell+1}, \;\; \ell \geq 2, \;\; \text{Level 1 weights:} \; \Lambda_{0}, \; \Lambda_{\ell}, \;\; \text{Level 2 weights:} \; \Lambda_1, \dots, \Lambda_{\ell-1}$
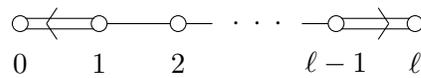
\begin{center}
\begin{tikzpicture} [scale=.7]  
\draw (0,0) circle (.15cm);
\node at (0,-.75) {$0$};
\draw (1.5,0) circle (.15cm);
\node at (1.5,-.75) {$1$};
\draw (3,0) circle (.15cm);
\node at (3,-.75) {$2$};
\draw (6,0) circle (.15cm);
\node at (6,-.75) {$\ell-1$};
\draw (7.5,0) circle (.15cm);
\node at (7.5,-.75) {$\ell$};
\draw (.1,.1) -- (1.4,.1);
\draw (.1,-.1) -- (1.4,-.1);
\draw (1.65,0) -- (2.85,0);
\draw (3.15,0) -- (3.65,0);
\draw (5.35,0) -- (5.85,0);
\draw (6.1,.1) -- (7.4,.1);
\draw (6.1,-.1) -- (7.4,-.1);
\draw (.7,.3) -- (.5,0);
\draw (.7,-.3) -- (.5,0);
\draw (6.8,.3) -- (7,0);
\draw (6.8,-.3) -- (7,0);
\node at (4.1,0) {$\cdot$};
\node at (4.5,0) {$\cdot$};
\node at (4.9,0) {$\cdot$};
\end{tikzpicture}
\end{center}

\item $D^{(1)}_{\ell}, \;\; \ell \geq 4, \;\; \text{Level 1 weights:} \; \Lambda_{0}, \; \Lambda_1, \; \Lambda_{\ell-1}, \; \Lambda_{\ell}, \;\; \text{Level 2 weights:} \; \Lambda_2, \dots, \Lambda_{\ell-2}$
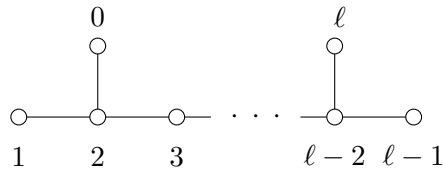
\begin{center}
\begin{tikzpicture} [scale=.7]  
\draw (0,0) circle (.15cm);
\node at (0,-.75) {$1$};
\draw (1.5,0) circle (.15cm);
\node at (1.5,-.75) {$2$};
\draw (3,0) circle (.15cm);
\node at (3,-.75) {$3$};
\draw (6,0) circle (.15cm);
\node at (6,-.75) {$\ell-2$};
\draw (7.5,0) circle (.15cm);
\node at (7.5,-.75) {$\ell-1$};
\draw (1.5,1.35) circle (.15cm);
\node at (1.5,1.9) {$0$};
\draw (6,1.35) circle (.15cm);
\node at (6.1,1.9) {$\ell$};
\draw (1.5,.15) -- (1.5,1.2);
\draw (6,.15) -- (6,1.2);
\draw (.15,0) -- (1.35,0);
\draw (1.65,0) -- (2.85,0);
\draw (3.15,0) -- (3.65,0);
\draw (5.35,0) -- (5.85,0);
\draw (6.15,0) -- (7.35,0);
\node at (4.1,0) {$\cdot$};
\node at (4.5,0) {$\cdot$};
\node at (4.9,0) {$\cdot$};
\end{tikzpicture}
\end{center}

\item $B^{(1)}_{\ell}, \;\; \ell \geq 3, \;\; \text{Level 1 weights:} \; \Lambda_{0}, \; \Lambda_{1}, \; \Lambda_{\ell}, \;\; \text{Level 2 weights:} \; \Lambda_2, \dots, \Lambda_{\ell-1}$
\begin{center}
\begin{tikzpicture} [scale=.7]  
\draw (0,0) circle (.15cm);
\node at (0,-.75) {$1$};
\draw (1.5,0) circle (.15cm);
\node at (1.5,-.75) {$2$};
\draw (3,0) circle (.15cm);
\node at (3,-.75) {$3$};
\draw (6,0) circle (.15cm);
\node at (6,-.75) {$\ell-1$};
\draw (7.5,0) circle (.15cm);
\node at (7.5,-.75) {$\ell$};
\draw (1.5,1.3) circle (.15cm);
\node at (.9,1.4) {$0$};
\draw (1.5,.15) -- (1.5,1.15);
\draw (.15,0) -- (1.35,0);
\draw (1.65,0) -- (2.85,0);
\draw (3.15,0) -- (3.65,0);
\draw (5.35,0) -- (5.85,0);
\draw (6.1,.1) -- (7.4,.1);
\draw (6.1,-.1) -- (7.4,-.1);
\draw (6.8,.3) -- (7,0);
\draw (6.8,-.3) -- (7,0);
\node at (4.1,0) {$\cdot$};
\node at (4.5,0) {$\cdot$};
\node at (4.9,0) {$\cdot$};
\end{tikzpicture}
\end{center}

\item $A^{(2)}_{2\ell-1}, \;\; \ell \geq 3, \;\; \text{Level 1 weights:} \; \Lambda_0, \; \Lambda_1, \;\; \text{Level 2 weights:} \; \Lambda_2, \dots, \Lambda_\ell$
\begin{center}
\begin{tikzpicture} [scale=.7]  
\draw (0,0) circle (.15cm);
\node at (0,-.75) {$1$};
\draw (1.5,0) circle (.15cm);
\node at (1.5,-.75) {$2$};
\draw (3,0) circle (.15cm);
\node at (3,-.75) {$3$};
\draw (6,0) circle (.15cm);
\node at (6,-.75) {$\ell-1$};
\draw (7.5,0) circle (.15cm);
\node at (7.5,-.75) {$\ell$};
\draw (1.5,1.3) circle (.15cm);
\node at (.9,1.4) {$0$};
\draw (1.5,.15) -- (1.5,1.15);
\draw (.15,0) -- (1.35,0);
\draw (1.65,0) -- (2.85,0);
\draw (3.15,0) -- (3.65,0);
\draw (5.35,0) -- (5.85,0);
\draw (6.1,.1) -- (7.4,.1);
\draw (6.1,-.1) -- (7.4,-.1);
\draw (7,.3) -- (6.8,0);
\draw (7,-.3) -- (6.8,0);
\node at (4.1,0) {$\cdot$};
\node at (4.5,0) {$\cdot$};
\node at (4.9,0) {$\cdot$};
\end{tikzpicture}
\end{center}

\end{itemize}

In most of the following theorems, 
we omit type $A^{(1)}_1$. All theorems and constructions in this paper hold in this type but require special arguments. These are presented in \cite{V15}. 

%
\subsection{Review of crystals} \label{Crystal-Review}
%

We recall the tensor category of crystals following
Kashiwara \cite{Kas95}, see also \cite{Kas90b,Kas91,KS97}.

A {\emph{crystal}} is a set $B$ together with maps
\begin{itemize}
  \item $\wt \maps B \longrightarrow P$,
  \item $\ep{i}, \p_i : B \longrightarrow \MB{Z} \sqcup \{ -\infty\}$ \quad for $i \in I$,
  \item $\etil{i}, \ftil{i} \maps B \longrightarrow B \sqcup \{0\}$ \quad for $i \in I$,
\end{itemize}
such that
\begin{enumerate}[C1.]
 \item $\p_{i}(b) =\ep{i}(b)+ \langle h_i, \wt(b) \rangle$ \quad for any $i$.
 \item If $b \in B$ satisfies $\etil{i} b \neq 0$, then
  \begin{align}
 &  \ep{i}(\etil{i}b) = \ep{i}(b)-1, & \p_{i}(\etil{i}b) = \p_{i}(b) +1, & & \wt(\etil{i}b) = \wt(b)+\alpha_i.
  \end{align}
 \item If $b \in B$ satisfies $\ftil{i} b \neq 0$, then
  \begin{align}
 &  \ep{i}(\ftil{i}b) = \ep{i}(b)+1,
 & \p_{i}(\ftil{i}b) = \p_{i}(b)-1,
 & &\wt(\ftil{i}b) = \wt(b)-\alpha_i.
  \end{align}
 \item For $b_1$, $b_2 \in B$, $b_2=\ftil{i}b_1$ if and only if
$
\etil{i}b_2
= b_1 
$. 
\item If $\p_{i}(b) = -\infty$, then
$\etil{i}b=\ftil{i}b=0$.
\end{enumerate} \medskip

For $b \in B$ we also define
\begin{equation}
\ep{}(b) = \sum_{i \in I} \ep{i}(b)\Lambda_i
\end{equation}
and
\begin{equation}
\p(b) = \sum_{i \in I} \p_i(b)\Lambda_i.
\end{equation}

If $B_1$ and $B_2$ are two crystals, then a {\emph{morphism}} $\psi \maps B_1 \rightarrow B_2$ of crystals is a map $$\psi \maps B_1 \sqcup \{0\} \rightarrow B_2 \sqcup \{0\}$$ satisfying the following properties:
\begin{enumerate}[M1.]
  \item $\psi(0) = 0$.
  \item If $\psi(b) \neq 0$ for $b \in B_1$, then
\begin{align}
  & \wt(\psi(b)) = \wt(b),
  & \ep{i}(\psi(b)) = \ep{i}(b),
  & &\p_{i}(\psi(b)) = \p_{i}(b).
\end{align}
 \item For $b \in B_1$ such that $\psi(b) \neq 0$ and $\psi(\etil{i}b) \neq 0$, we have $\psi(\etil{i}b) = \etil{i}(\psi(b))$.
 \item For $b \in B_1$ such that $\psi(b) \neq 0$ and $\psi(\ftil{i}b) \neq 0$, we have $\psi(\ftil{i}b) = \ftil{i}(\psi(b))$.
\end{enumerate}
A morphism $\psi$ of crystals is called {\emph{strict}} if
\begin{equation}
  \psi \etil{i} = \etil{i}\psi, \qquad \quad \psi \ftil{i} = \ftil{i}\psi,
\end{equation}
and an {\emph{embedding}} if $\psi$ is injective. \medskip

Given two crystals $B_1$ and $B_2$ their tensor product
$B_1 \otimes B_2$ (using the reverse Kashiwara convention)
 has underlying set $\{b_1 \otimes b_2 \; | \; b_1 \in B_1 \;
\text{and} \; b_2 \in B_2 \}$ where we identify $b_1 \otimes 0 = 0
\otimes b_2 = 0$. The crystal structure is given as follows:
\begin{align} \label{tensor-product-crystals}
  & \wt(b_1 \otimes b_2) = \wt(b_1) + \wt (b_2), \\
  & \ep{i}(b_1 \otimes b_2) = \max\{ \ep{i}(b_2), \ep{i}(b_1)-\langle h_i, \wt(b_2) \rangle\},
\\ &
\cryphi{i}(b_1 \otimes b_2) = \max\{ \cryphi{i}(b_2) + \langle h_i,\wt(b_1) \rangle, \cryphi{i}(b_1)\}, 
\end{align}
\begin{align}
  &\etil{i}(b_1 \otimes b_2 ) = 
  \begin{cases}
    \etil{i}b_1 \otimes b_2 & \text{if $\ep{i}(b_1) > \cryphi{i}(b_2)$}\\
    b_1 \otimes \etil{i}b_2 & \text{if $\ep{i}(b_1) \leq \cryphi{i}(b_2)$},
  \end{cases}
\label{eq_ei_tensor}
\\
& \ftil{i}(b_1 \otimes b_2 ) =
  \begin{cases}
    \ftil{i}b_1 \otimes b_2 &  \text{if $\ep{i}(b_1) \geq \cryphi{i}(b_2)$}\\
    b_1 \otimes \ftil{i}b_2 & \text{if $\ep{i}(b_1) < \cryphi{i}(b_2)$}.
\label{eq_fi_tensor}
  \end{cases}
\end{align}

Given a crystal $B$, we can draw its  associated crystal graph
with nodes (or vertices) $B$ and $I$-colored arrows (directed edges)
as follows.
When $\etil{i}b = a$ (so $b = \ftil{i}a$) we draw an $i$-colored arrow $a \xrightarrow{i} b$. We also say $b$ has an incoming $i$-arrow and $a$ has an outgoing $i$-arrow.


\subsection{Perfect crystals and Kirillov-Reshetikhin crystals}
\label{perfect-crystal-section}


\subsubsection{Type $A$}
In type $A^{(1)}_{\ell}$, the highest weight crystal $B(\Lambda_i)$ has a model
(see Figure \ref{fig-fundamental}) with nodes the
$(\ell+1)$-{\em restricted}
partitions, i.e. $\lambda = (\lambda_1, \ldots, \lambda_t)$
such that $\lambda_{r} \in \Z_{\ge 0}$, 
$ 0 \le \lambda_{r} - \lambda_{r+1} < \ell+1$ for all $r$.
Let $B^{1,1}$ be the crystal graph 

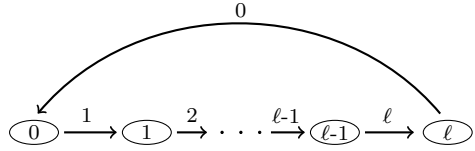
\begin{center}
\begin{tikzpicture} [scale=.5]
\draw (0,.25) ellipse (.65cm and .32cm);
\draw (3,.25) ellipse (.65cm and .32cm);
\draw (8,.25) ellipse (.65cm and .32cm);
\draw (11,.25) ellipse (.65cm and .32cm);
\draw[->,thick] (.8,.25) -- (2.2,.25);
\draw[->,thick] (8.8,.25) -- (10.1,.25);
\draw[->,thick] (3.8,.25) -- (4.6,.25);
\draw[->,thick] (6.3,.25) -- (7.2,.25);
\draw [fill] (5,0.25) circle [radius=0.03];
\draw [fill] (5.5,0.25) circle [radius=0.03];
\draw [fill] (6,0.25) circle [radius=0.03];
\node at (0,.25) {\scriptsize{0}};
\node at (3,.25) {\scriptsize{1}};
\node at (8,.25) {\scriptsize{$\ell$-1}};
\node at (11,.25) {\scriptsize{$\ell$}};
\node at (1.4,.7) {\scriptsize{1}};
\node at (4.2,.7) {\scriptsize{2}};
\node at (6.7,.7) {\scriptsize{$\ell$-1}};
\node at (9.4,.7) {\scriptsize{$\ell$}};
\node at (5.5,3.5) {\scriptsize{0}};
\draw[->,thick] (10.8,.75) to [out = 130, in = 50] (0.1,.75);
\end{tikzpicture}
\end{center}
$B^{1,1}$ is also drawn in \eqref{perfect-crystal-A} without the node labels. 

$B^{1,1}$ is an example of a \emph{perfect} crystal (see \cite{KKMMNN} for the definition
and important properties).
One key property this level $1$ perfect crystal has
is that tensoring it with a fundamental (or highest weight
level $1$) crystal
 yields an isomorphism to another
level $1$ highest weight crystal.
In particular,
for $i \in I$ there exists an isomorphism of crystals,
$$\crystalmap:B(\Lambda_{i})
\xrightarrow{\simeq}
 B^{1,1} \otimes B(\Lambda_{i-1}).$$
The isomorphism is pictured in Figure
\ref{crystal-iso} for $i = 0$ and $\ell = 2$.
Note the underlying graph of $B(\Lambda_i)$ is identical to
that of $B(\Lambda_0)$, but the colors of the arrows
are obtained from those of $B(\Lambda_0)$ by adding
$i \bmod (\ell\!+\!1)$.

\begin{figure}[h]
\begin{center}

\begin{tikzpicture} [scale=.1]
\node at (0,0) {\begin{tikzpicture}[xscale=3*\UNIT, yscale=-4.5*\UNIT]
\begin{scope}
	\LatticeMV 
\end{scope} 
\begin{scope}[every node/.style={fill=white}]
	\node (a) at (0) {$\emptyset$};
	\node (b) at (11) {\begin{tikzpicture}\youngDiagram{{\scriptsize{0}}}{.5};\end{tikzpicture}};
	\node (c) at (21) {\begin{tikzpicture}\youngDiagram{{\scriptsize{0},\scriptsize{\color{red}{1}}}}{.5};\end{tikzpicture}};
	\node (d) at (31) {\begin{tikzpicture}\youngDiagram{{\scriptsize{0},\scriptsize{\color{red}{1}}},{\scriptsize{\color{blue}{2}}}}{.5};\end{tikzpicture}};
	\node (g) at (22) {\begin{tikzpicture}\youngDiagram{{\scriptsize{0}},{\scriptsize{\color{blue}{2}}}}{.5};\end{tikzpicture}};
	\node (i) at (32) {\begin{tikzpicture}\youngDiagram{{\scriptsize{0}},{\scriptsize{\color{blue}{2}}},{\scriptsize{\color{red}{1}}}}{.5};\end{tikzpicture}};
\end{scope}
\begin{scope}[black, thick,->]
\draw (a)--(b) node[midway,left] {0};
\draw[red] (b)--(c) node[midway,left] {\color{red}{1}};
\draw[blue] (b)--(g) node[midway,left] {\color{blue}{2}};
\draw[red] (g)--(i) node[midway,left] {\color{red}{1}};
\draw[blue] (c)--(d) node[midway,left] {\color{blue}{2}};
\end{scope}
\end{tikzpicture}};

\node at (80,0) {\begin{tikzpicture}[xscale=3*\UNIT, yscale=-4.5*\UNIT]
\begin{scope}
	\LatticeMV 
\end{scope} 
\begin{scope}[every node/.style={fill=white}]
	\node (a) at (0) {$\emptyset$};
	\node (b) at (11) {\begin{tikzpicture}\youngDiagram{{\scriptsize{\color{blue}{2}}}}{.5};\end{tikzpicture}};
	\node (c) at (21) {\begin{tikzpicture}\youngDiagram{{\scriptsize{{\color{blue}{2}}},\scriptsize{\color{black}{0}}}}{.5};\end{tikzpicture}};
	\node (d) at (31) {\begin{tikzpicture}\youngDiagram{{\scriptsize{{\color{blue}{2}}},\scriptsize{\color{black}{0}}},{\scriptsize{\color{red}{1}}}}{.5};\end{tikzpicture}};
	\node (g) at (22) {\begin{tikzpicture}\youngDiagram{{\scriptsize{{\color{blue}{2}}}},{\scriptsize{\color{red}{1}}}}{.5};\end{tikzpicture}};
	\node (i) at (32) {\begin{tikzpicture}\youngDiagram{{\scriptsize{{\color{blue}{2}}}},{\scriptsize{\color{red}{1}}},{\scriptsize{\color{black}{0}}}}{.5};\end{tikzpicture}};
\end{scope}
\begin{scope}[black, thick,->]
\draw[blue] (a)--(b) node[midway,left] {{\color{blue}{2}}};
\draw (b)--(c) node[midway,left] {\color{black}{0}};
\draw[red] (b)--(g) node[midway,left] {\color{red}{1}};
\draw (g)--(i) node[midway,left] {\color{black}{0}};
\draw[red] (c)--(d) node[midway,left] {\color{red}{1}};
\end{scope}
\end{tikzpicture}};

\end{tikzpicture}

\caption{ \label{fig-fundamental}
$B(\Lambda_0)$ and $ B(\Lambda_2)$
for $\ell = 2$.
}
\end{center}
\end{figure}
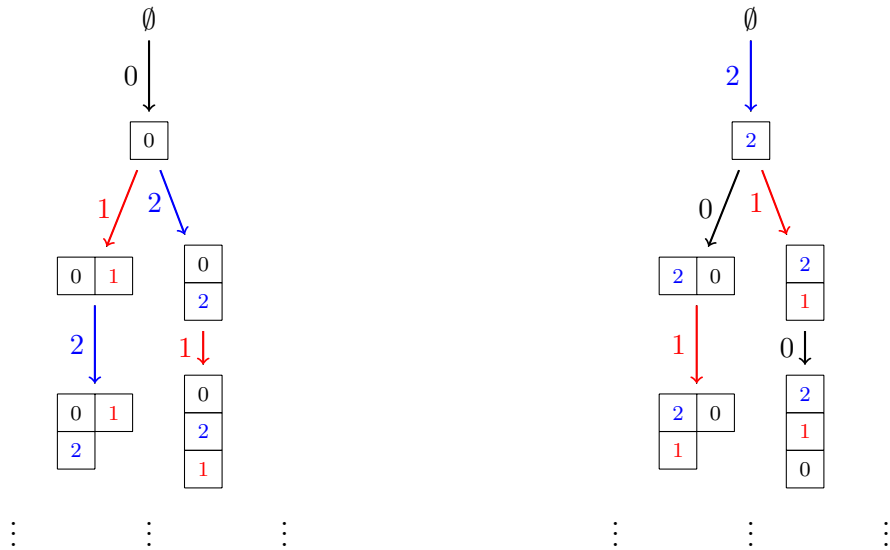

Combinatorially, $\crystalmap(\lambda) = $
\begin{tikzpicture}[baseline=-2pt] 
	\node at (0,0) {\;\;$k$\;\;}; 
        \draw (0,0) ellipse (.38cm and .2cm);
\end{tikzpicture}$ \; \otimes\; \mu$
where
$k\equiv \lambda_1 +i-1 \bmod (\ell+1)$
and $\mu = (\lambda_2,
\dots, \lambda_t)$ if $\lambda =(\lambda_1, \lambda_2, \cdots,
\lambda_t)$.
So we obtain $\mu$ from $\lambda$ by removing its top row.
In Figure \ref{crystal-iso}, we draw
\begin{equation}
\begin{tikzpicture}
			\node at (0,0) {$\crystalmap(\lambda) = $};
			\node (z) at (1,.33) {\scriptsize{$k$}};
                       	\draw (z) ellipse (.35cm and .15cm);
			\node at (1,0) {$\otimes$};
			\node at (1,-.3) {$\mu$};
\end{tikzpicture}
\end{equation}
so the visual of the top row removal stands out.
Note $\lambda_1 - \mu_1 < \ell+1$ means that
 $\crystalmap$ has a well-defined inverse.

When drawing $B(\Lambda_i)$, we label each box of a partition with 
$k \in I$, such that the main diagonal gets label $i$, 
and labels increase by $1 \bmod (\ell + 1)$ as one increases
diagonals (moving right).
In this manner, the last box in the top row of $\lambda$
is labeled $k$ when
 $\crystalmap(\lambda) = $
\begin{tikzpicture}[baseline=-2pt] 
	\node at (0,0) {\;\;$k$\;\;}; 
        \draw (0,0) ellipse (.38cm and .2cm);
\end{tikzpicture}$ \; \otimes\; \mu$.
Note further that
if we have a $k$-arrow $\gamma \xrightarrow{k} \lambda$
then the box $\lambda/\gamma$ is labeled $k$ (though not necessarily
conversely). This box is often called a ``good" $k$-box.
In fact, once one knows the structure of $B^{1,1}$ and
the tensor product rule for crystals, one can obtain
the rule for which $k$-box $\etil{k}$ removes by
iterating $\crystalmap$.

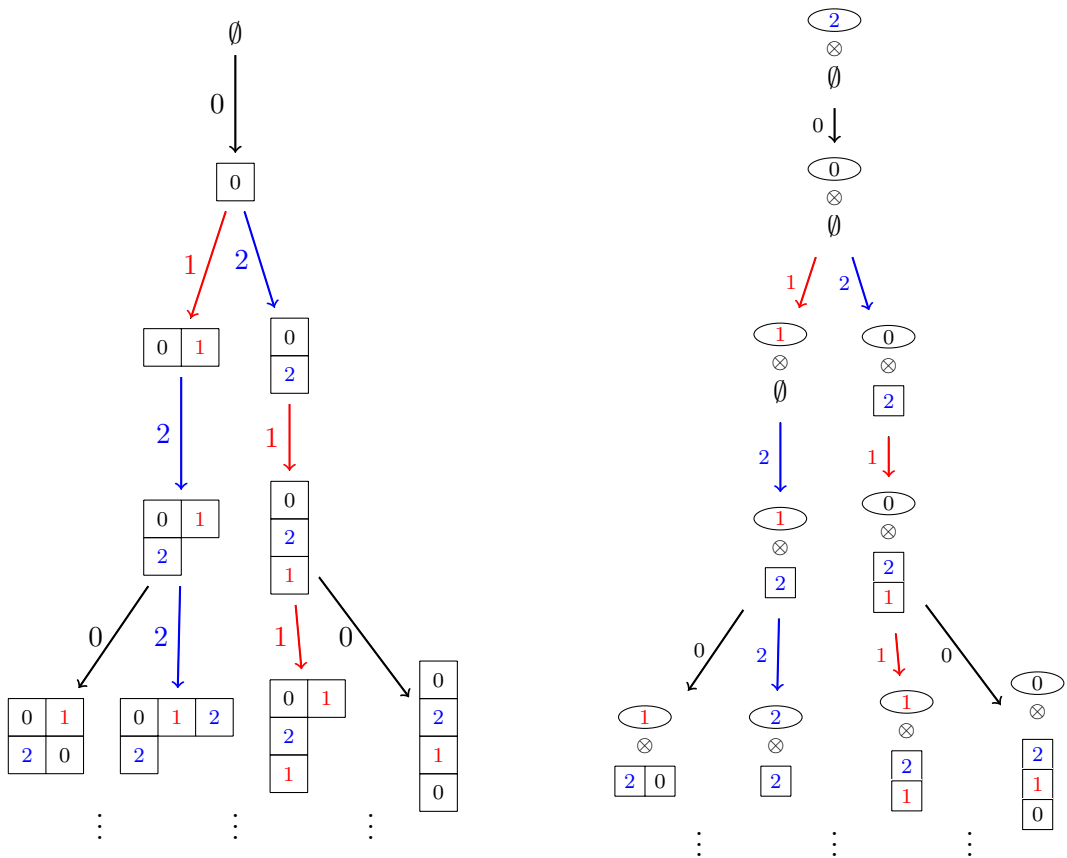
\begin{figure}[h]
\begin{center}
\begin{tikzpicture} [scale=.1]
\node at (0,0) {\begin{tikzpicture}[xscale=3*\UNIT, yscale=-5.5*\UNIT]
\begin{scope}
	\Lattice 
\end{scope} 
\begin{scope}[every node/.style={fill=white}]
	\node (a) at (0) {$\emptyset$};
	\node (b) at (11) {\begin{tikzpicture}\youngDiagram{{\scriptsize{0}}}{.5};\end{tikzpicture}};
	\node (c) at (21) {\begin{tikzpicture}\youngDiagram{{\scriptsize{0},\scriptsize{\color{red}{1}}}}{.5};\end{tikzpicture}};
	\node (d) at (31) {\begin{tikzpicture}\youngDiagram{{\scriptsize{0},\scriptsize{\color{red}{1}}},{\scriptsize{\color{blue}{2}}}}{.5};\end{tikzpicture}};
	\node (e) at (41) {\begin{tikzpicture}\youngDiagram{{\scriptsize{0},\scriptsize{\color{red}{1}}},{\scriptsize{\color{blue}{2}},\scriptsize{0}}}{.5};\end{tikzpicture}};
	\node (g) at (22) {\begin{tikzpicture}\youngDiagram{{\scriptsize{0}},{\scriptsize{\color{blue}{2}}}}{.5};\end{tikzpicture}};
	\node (i) at (32) {\begin{tikzpicture}\youngDiagram{{\scriptsize{0}},{\scriptsize{\color{blue}{2}}},{\scriptsize{\color{red}{1}}}}{.5};\end{tikzpicture}};
	\node (k) at (42) {\begin{tikzpicture}\youngDiagram{{\scriptsize{0},\scriptsize{\color{red}{1}},\scriptsize{\color{blue}{2}}},{\scriptsize{\color{blue}{2}}}}{.5};\end{tikzpicture}};
	\node (l) at (43) {\begin{tikzpicture}\youngDiagram{{\scriptsize{0},\scriptsize{\color{red}{1}}},{\scriptsize{\color{blue}{2}}},{\scriptsize{\color{red}{1}}}}{.5};\end{tikzpicture}};
	\node (m) at (44) {\begin{tikzpicture}\youngDiagram{{\scriptsize{0}},{\scriptsize{\color{blue}{2}}},{\scriptsize{\color{red}{1}}},{\scriptsize{0}}}{.5};\end{tikzpicture}};
\end{scope}
\begin{scope}[black, thick,->]
\draw (a)--(b) node[midway,left] {0};
\draw[red] (b)--(c) node[midway,left] {\color{red}{1}};
\draw (d)--(e) node[midway,left] {0};
\draw[blue] (b)--(g) node[midway,left] {\color{blue}{2}};
\draw[red] (g)--(i) node[midway,left] {\color{red}{1}};
\draw[blue] (c)--(d) node[midway,left] {\color{blue}{2}};
\draw[blue] (d)--(k) node[midway,left] {\color{blue}{2}};
 \draw[red] (i)--(l) node[midway,left] {\color{red}{1}};
\draw (i)--(m) node[midway,left] {0};
\end{scope}
\end{tikzpicture}};

\node at (80,0) {\begin{tikzpicture}[xscale=3*\UNIT, yscale=-5.5*\UNIT]
\begin{scope}
	\Lattice 
\end{scope} 
\begin{scope}[every node/.style={fill=white}]
	\node (a) at (0) {\begin{tikzpicture} 
	                           \node at (0,0.1) {$\emptyset$};
	                           \node at (0,.47) {\scriptsize{$\otimes$}}; 
	                           \node (z) at (0,.85) {\scriptsize{\color{blue}{2}}};
	                           \draw (z) ellipse (.35cm and .15cm);
	                           \end{tikzpicture}};
	\node (b) at (11) {\begin{tikzpicture}
	                           \node at (0,0.1) {$\emptyset$};
	                           \node at (0,.47) {\scriptsize{$\otimes$}};
	                           \node (z) at (0,.85) {\scriptsize{0}};
	                           \draw (z) ellipse (.35cm and .15cm);
	                           \end{tikzpicture}};
	\node (c) at (21) {\begin{tikzpicture}
				 \node at (0,0.1) {$\emptyset$};				  
				  \node at (0,.47) {\scriptsize{$\otimes$}};
				 \node (z) at (0,.85) {\scriptsize{\color{red}{1}}};
	                           \draw (z) ellipse (.35cm and .15cm);
	                           \end{tikzpicture}};
	\node (d) at (31) {\begin{tikzpicture}
				  \node at (0,0) {\begin{tikzpicture}\youngDiagram{{\scriptsize{\color{blue}{2}}}}{.4}\end{tikzpicture}};
				  \node at (0,.47) {\scriptsize{$\otimes$}};
				  \node (z) at (0,.85) {\scriptsize{\color{red}{1}}};
	                           \draw (z) ellipse (.35cm and .15cm);
	                           \end{tikzpicture}};
	\node (e) at (41) {\begin{tikzpicture}
	                            \node at (0,0) {\begin{tikzpicture}\youngDiagram{{\scriptsize{\color{blue}{2}},\scriptsize{0}}}{.4};\end{tikzpicture}};
	                            \node at (0,.47) {\scriptsize{$\otimes$}};
	                           \node (z) at (0,.85) {\scriptsize{\color{red}{1}}};
	                           \draw (z) ellipse (.35cm and .15cm);
	                           \end{tikzpicture}};
	\node (g) at (22) {\begin{tikzpicture}
	                          \node at (0,0) {\begin{tikzpicture}\youngDiagram{{\scriptsize{\color{blue}{2}}}}{.4};\end{tikzpicture}};
				\node at (0,.47) {\scriptsize{$\otimes$}};
				\node (z) at (0,.85) {\scriptsize{0}};
	                         \draw (z) ellipse (.35cm and .15cm);
	                          \end{tikzpicture}};
	\node (i) at (32) {\begin{tikzpicture}
				\node at (0,-.2) {\begin{tikzpicture}\youngDiagram{{\scriptsize{\color{blue}{2}}},{\scriptsize{\color{red}{1}}}}{.4};\end{tikzpicture}};
				\node at (0,.47) {\scriptsize{$\otimes$}};
				\node (z) at (0,.85) {\scriptsize{0}};
	                          \draw (z) ellipse (.35cm and .15cm);
	                          \end{tikzpicture}};
	\node (k) at (42) {\begin{tikzpicture}
				\node at (0,0) {\begin{tikzpicture}\youngDiagram{{\scriptsize{\color{blue}{2}}}}{.4};\end{tikzpicture}};
				 \node at (0,.47) {\scriptsize{$\otimes$}};
				\node (z) at (0,.85) {\scriptsize{\color{blue}{2}}};
	                          \draw (z) ellipse (.35cm and .15cm);
	                           \end{tikzpicture}};
	\node (l) at (43) {\begin{tikzpicture}
				\node at (0,-.2) {\begin{tikzpicture}\youngDiagram{{\scriptsize{\color{blue}{2}}},{\scriptsize{\color{red}{1}}}}{.4};\end{tikzpicture}};
				\node at (0,.47) {\scriptsize{$\otimes$}};
				\node (z) at (0,.85) {\scriptsize{\color{red}{1}}};
	                           \draw (z) ellipse (.35cm and .15cm);
	                           \end{tikzpicture}};
	\node (m) at (44) {\begin{tikzpicture}
				\node at (0,.6) {};
				\node at (0,-.5) {\begin{tikzpicture}\youngDiagram{{\scriptsize{\color{blue}{2}}},{\scriptsize{\color{red}{1}}},{\scriptsize{0}}}{.4};\end{tikzpicture}};
				\node at (0,.47) {\scriptsize{$\otimes$}};
				\node (z) at (0,.85) {\scriptsize{0}};
	                           \draw (z) ellipse (.35cm and .15cm);
	                           \end{tikzpicture}};
\end{scope}
\begin{scope}[black, thick,->]
\draw (a)--(b) node[midway,left] {\scriptsize{0}};
 \draw[red] (b)--(c) node[midway,left] {\scriptsize{\color{red}{1}}};
\draw (d)--(e) node[midway,left] {\scriptsize{0}};
\draw[blue] (b)--(g) node[midway,left] {\scriptsize{\color{blue}{2}}};
\draw[red] (g)--(i) node[midway,left] {\scriptsize{\color{red}{1}}};
\draw[blue] (c)--(d) node[midway,left] {\scriptsize{\color{blue}{2}}};
\draw[blue] (d)--(k) node[midway,left] {\scriptsize{\color{blue}{2}}}; 
\draw[red] (i)--(l) node[midway,left] {\scriptsize{\color{red}{1}}};
\draw (i)--(m) node[midway,left] {\scriptsize{0}};
\end{scope}
\end{tikzpicture}};

\end{tikzpicture}
\caption{ \label{crystal-iso}
The isomorphism $B(\Lambda_0) \simeq B^{1,1} \otimes B(\Lambda_2)$
for $\ell = 2$.
}
\end{center}
\end{figure}


\begin{remark} \label{B-ell-1-remark}
There is another model of $B(\Lambda_i)$
for which nodes are indexed by
$(\ell+1)$-regular partitions.
A partition
$\lambda $ is $(\ell+1)$-{\em regular} iff its transpose
$\lambda^T$ is $(\ell+1)$-restricted. 
(However one does not obtain this model by 
merely transposing the partition indexing each node.)
There is a similar isomorphism
$\crystalmap':B(\Lambda_{i})
\xrightarrow{\simeq}
 B^{\ell,1} \otimes B(\Lambda_{i+1})$
corresponding to 
column removal.

For type $A^{(1)}_\ell$ the diagram for $B^{\ell,1}$ is 

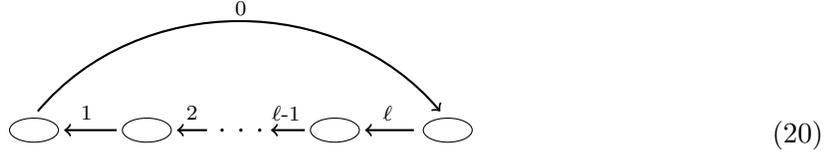
\begin{center}
\begin{equation} \label{perfect-crystal-A-reverse}
\begin{tikzpicture} [scale=.5]
\draw (0,.25) ellipse (.65cm and .32cm);
\draw (3,.25) ellipse (.65cm and .32cm);
\draw (8,.25) ellipse (.65cm and .32cm);
\draw (11,.25) ellipse (.65cm and .32cm);
\draw[<-,thick] (.8,.25) -- (2.2,.25);
\draw[<-,thick] (8.8,.25) -- (10.1,.25);
\draw[<-,thick] (3.8,.25) -- (4.6,.25);
\draw[<-,thick] (6.3,.25) -- (7.2,.25);
\draw [fill] (5,0.25) circle [radius=0.03];
\draw [fill] (5.5,0.25) circle [radius=0.03];
\draw [fill] (6,0.25) circle [radius=0.03];
\node at (0,.25) {\scriptsize{}};
\node at (3,.25) {\scriptsize{}};
\node at (8,.25) {\scriptsize{}};
\node at (11,.25) {\scriptsize{}};
\node at (1.4,.7) {\scriptsize{1}};
\node at (4.2,.7) {\scriptsize{2}};
\node at (6.7,.7) {\scriptsize{$\ell$-1}};
\node at (9.4,.7) {\scriptsize{$\ell$}};
\node at (5.5,3.5) {\scriptsize{0}};
\draw[<-,thick] (10.8,.75) to [out = 130, in = 50] (0.1,.75);
\end{tikzpicture}
\end{equation}
\end{center}
This can be obtained from $B^{1,1}$ in type $A^{(1)}_\ell$ by reversing the direction of all arrows.

\end{remark}

See Section \ref{sec-LV} for another model of $B(\Lambda_i)$.

\subsubsection{General type}
\label{sec-general}
The perfect crystal $B^{1,1}$ in \eqref{perfect-crystal-A} is also an
example of a Kirillov-Reshetikhin (KR) crystal. For a quantized affine
algebra $U'_q(\Fr{g})$, the KR crystals $B^{r,s}$ correspond to a
special family of finite dimensional modules $W^{r,s}$ indexed by a
positive integer $s$ and a Dynkin node $r$ from the classical
subalgebra $\Fr{g}_0$ of $\Fr{g}$ \cite{KR90}, \cite{OS08}. In this
paper we work with the crystals $B^{1,1}$ which have particularly
simple graphs.
In all of the types we consider,
with the exception of $C^{(1)}_\ell$, the crystal $B^{1,1}$ is perfect
of level 1 \cite{FOS10}. When $B^{1,1}$ is perfect and $\Lambda_i$ is a
level 1 fundamental weight for $i \in I$, $B^{1,1}$ has a unique node
$b_i$ such that $\ep{}(b_i) = \Lambda_i$ and $\p(b_i) =
\Lambda_{\sigma(i)}$ for some $\sigma(i) \in I$. There then exists a
crystal isomorphism \cite{KKMMNN}
\begin{equation}
\label{eq-perfect}
\crystalmap: B(\Lambda_{\sigma(i)}) \xrightarrow{\sim} B^{1,1} \otimes B(\Lambda_i).
\end{equation}
In this paper we will also consider the tensor product $B^{1,1} \otimes
B(\Lambda_i)$ when $\Lambda_i$ is a level 2 weight and 
consider type
$C^{(1)}_\ell$ where $B^{1,1}$ is not perfect. In those cases, $B^{1,1}
\otimes B(\Lambda_i)$ decomposes as
\begin{equation}
\label{eq-nonperfect}
B^{1,1} \otimes B(\Lambda_i) \cong \bigoplus_{j \in I} k_j
B(s_j\Lambda_j), \quad \quad k_j \in \MB{Z}_{\geq 0}, s_j \in \{1,2\}
 \end{equation}
for appropriate $k_j, s_j$.
See \cite{KKMMNN92}, \cite{KM98} for the case where $\Lambda_i$ is not level 1, and \cite{SL13} for type $C^{(1)}_\ell$, $i = 0$ which can be generalized to $i \in I$ by similar methods.
(When the level of $\Lambda_i$ is $1$, all $s_j=1$; but when
the level of $\Lambda_i$ is $2$ and the level of $\Lambda_j$
is only $1$, we can have $s_j=2$.)

\begin{example} \label{ex-typeC}

In type $C_\ell^{(1)}$, one can check 
\begin{equation}
B^{1,1} \otimes B(\Lambda_1) \simeq B(\Lambda_0) \oplus B(\Lambda_2).
\end{equation}
Recall $B^{1,1}$ is {\em not} a perfect crystal in type $C_\ell^{(1)}$.

If we call the crystal isomorphism $\crystalmap$,
label the nodes of $B^{1,1}$ as
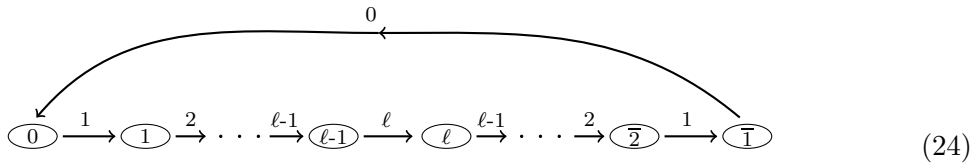
\begin{equation} \label{KR-crystal-C}
\begin{tikzpicture} [scale=.5]
\draw (0,.25) ellipse (.65cm and .32cm);
\draw (3,.25) ellipse (.65cm and .32cm);
\draw (8,.25) ellipse (.65cm and .32cm);
\draw (11,.25) ellipse (.65cm and .32cm);
\draw (16,.25) ellipse (.65cm and .32cm);
\draw (19,.25) ellipse (.65cm and .32cm);
\draw[->,thick] (.8,.25) -- (2.2,.25);
\draw[->,thick] (8.8,.25) -- (10.1,.25);
\draw[->,thick] (3.8,.25) -- (4.6,.25);
\draw[->,thick] (6.3,.25) -- (7.2,.25);
\draw[->,thick] (11.8,.25) -- (12.6,.25);
\draw[->,thick] (14.4,.25) -- (15.2,.25);
\draw[->,thick] (16.8,.25) -- (18.2,.25);
\draw [fill] (5,0.25) circle [radius=0.03];
\draw [fill] (5.5,0.25) circle [radius=0.03];
\draw [fill] (6,0.25) circle [radius=0.03];
\draw [fill] (13,0.25) circle [radius=0.03];
\draw [fill] (13.5,0.25) circle [radius=0.03];
\draw [fill] (14,0.25) circle [radius=0.03];
\node at (0,.25) {\scriptsize{}};
\node at (3,.25) {\scriptsize{}};
\node at (8,.25) {\scriptsize{}};
\node at (11,.25) {\scriptsize{}};
\node at (1.4,.7) {\scriptsize{1}};
\node at (4.2,.7) {\scriptsize{2}};
\node at (6.7,.7) {\scriptsize{$\ell$-1}};
\node at (9.4,.7) {\scriptsize{$\ell$}};
\node at (16,.25) {\scriptsize{}};
\node at (19,.25) {\scriptsize{}};
\node at (12.2,.7) {\scriptsize{$\ell$-1}};
\node at (14.8,.7) {\scriptsize{2}};
\node at (17.4,.7) {\scriptsize{1}};
\node at (9,3.5) {\scriptsize{0}};
\node at (9.4,3) {\scriptsize{}};
\draw[->,thick] (9.2,3) to [out = 180, in = 50] (0.1,.75);
\draw[->,thick] (18.8,.75) to [out = 140, in = 0] (9.2,3);
\node at (0,.25) {\scriptsize{0}};
\node at (3,.25) {\scriptsize{1}};
\node at (8,.25) {\scriptsize{$\ell$-1}};
\node at (11,.25) {\scriptsize{$\ell$}};
\node at (16,.25) {\scriptsize{$\overline 2$}};
\node at (19,.25) {\scriptsize{$\overline 1$}};
\end{tikzpicture}
\end{equation}

and the highest weight node of $B(\Lambda_i)$ as $\bi{i}$ then
\begin{gather}
 \crystalmap \left( \right.
\begin{tikzpicture}[baseline=-2pt]
        \node at (0,0) {\;\;{\scriptsize{$1$}}\;\;};
        \draw (0,0) ellipse (.38cm and .2cm);
\end{tikzpicture} \left. \; \otimes\; \bi{1} \right) = \bi{2}
\\
 \crystalmap \left( \right.
\begin{tikzpicture}[baseline=-2pt]
        \node at (0,0) {\;\;{\scriptsize{$\overline 1$}}\;\;};
        \draw (0,0) ellipse (.38cm and .2cm);
\end{tikzpicture} \left. \; \otimes\; \bi{1} \right) = \bi{0}
\end{gather}
as $\ep{}(\begin{tikzpicture}[baseline=-2pt]
        \node at (0,0) {\;\;{\scriptsize{$1$}}\;\;};
        \draw (0,0) ellipse (.38cm and .2cm);
\end{tikzpicture}) = \ep{}(\begin{tikzpicture}[baseline=-2pt]
        \node at (0,0) {\;\;{\scriptsize{$\overline 1$}}\;\;};
        \draw (0,0) ellipse (.38cm and .2cm);
\end{tikzpicture}) = \Lambda_1$ but
$\p(\begin{tikzpicture}[baseline=-2pt]
        \node at (0,0) {\;\;{\scriptsize{$1$}}\;\;};
        \draw (0,0) ellipse (.38cm and .2cm);
\end{tikzpicture}) = \Lambda_2, \p(\begin{tikzpicture}[baseline=-2pt]
        \node at (0,0) {\;\;{\scriptsize{$\overline 1$}}\;\;};
        \draw (0,0) ellipse (.38cm and .2cm);
\end{tikzpicture}) = \Lambda_0.$

How this difference manifests in  the main theorems 
of the paper is also addressed in Remark \ref{rem-typeC}
and Example \ref{ex-modC}.

Note, one can also check
$B^{1,1} \otimes B(\Lambda_0) \simeq B(\Lambda_1).$

In type $A_{2\ell}^{(2)}$ one can check that for similar reasons
as above that
\begin{equation}
\label{eq-A2-2}
B^{1,1} \otimes B(\Lambda_2) \simeq B(\Lambda_1) \oplus B(\Lambda_3),
\end{equation}
if $\ell \ge 3$.
Recall that $\Lambda_1, \Lambda_2, \Lambda_3$ are all of level $2$,
whereas $\Lambda_0$ is level $1$.
One can also check
\begin{equation}
\label{eq-A2-1}
B^{1,1} \otimes B(\Lambda_1) \simeq B(2\Lambda_0) \oplus B(\Lambda_2).
\end{equation}

\end{example}


Below we draw the crystal graphs for the KR crystals $B^{1,1}$ for each type $X_\ell$ we are considering. The nodes in $B^{1,1}$ are labelled with 
\begin{tikzpicture}[baseline=-2pt] 
        \draw (0,0) ellipse (.25cm and .13cm);
\end{tikzpicture} 
to distinguish them from the Dynkin nodes 
\begin{tikzpicture}[baseline=-2pt] 
        \draw (0,0) circle (.13cm);
\end{tikzpicture} .
The conditions $\ell \ge 2,3,4$ hold below in each respective
type as they do for the Dynkin diagrams listed in Section
\ref{Cartan-datum-section}.

\begin{itemize}

\item $A^{(1)}_{\ell}$ 
\begin{center}
\begin{equation} \label{perfect-crystal-A}
\begin{tikzpicture} [scale=.5]
\draw (0,.25) ellipse (.65cm and .32cm);
\draw (3,.25) ellipse (.65cm and .32cm);
\draw (8,.25) ellipse (.65cm and .32cm);
\draw (11,.25) ellipse (.65cm and .32cm);
\draw[->,thick] (.8,.25) -- (2.2,.25);
\draw[->,thick] (8.8,.25) -- (10.1,.25);
\draw[->,thick] (3.8,.25) -- (4.6,.25);
\draw[->,thick] (6.3,.25) -- (7.2,.25);
\draw [fill] (5,0.25) circle [radius=0.03];
\draw [fill] (5.5,0.25) circle [radius=0.03];
\draw [fill] (6,0.25) circle [radius=0.03];
\node at (0,.25) {\scriptsize{}};
\node at (3,.25) {\scriptsize{}};
\node at (8,.25) {\scriptsize{}};
\node at (11,.25) {\scriptsize{}};
\node at (1.4,.7) {\scriptsize{1}};
\node at (4.2,.7) {\scriptsize{2}};
\node at (6.7,.7) {\scriptsize{$\ell$-1}};
\node at (9.4,.7) {\scriptsize{$\ell$}};
\node at (5.5,3.5) {\scriptsize{0}};
\draw[->,thick] (10.8,.75) to [out = 130, in = 50] (0.1,.75);
\end{tikzpicture}
\end{equation}
\end{center}

\item $C^{(1)}_{\ell}$ 
\begin{center}
\begin{equation} \label{perfect-crystal-C}
\begin{tikzpicture} [scale=.5]
\draw (0,.25) ellipse (.65cm and .32cm);
\draw (3,.25) ellipse (.65cm and .32cm);
\draw (8,.25) ellipse (.65cm and .32cm);
\draw (11,.25) ellipse (.65cm and .32cm);
\draw (16,.25) ellipse (.65cm and .32cm);
\draw (19,.25) ellipse (.65cm and .32cm);
\draw[->,thick] (.8,.25) -- (2.2,.25);
\draw[->,thick] (8.8,.25) -- (10.1,.25);
\draw[->,thick] (3.8,.25) -- (4.6,.25);
\draw[->,thick] (6.3,.25) -- (7.2,.25);
\draw[->,thick] (11.8,.25) -- (12.6,.25);
\draw[->,thick] (14.4,.25) -- (15.2,.25);
\draw[->,thick] (16.8,.25) -- (18.2,.25);
\draw [fill] (5,0.25) circle [radius=0.03];
\draw [fill] (5.5,0.25) circle [radius=0.03];
\draw [fill] (6,0.25) circle [radius=0.03];
\draw [fill] (13,0.25) circle [radius=0.03];
\draw [fill] (13.5,0.25) circle [radius=0.03];
\draw [fill] (14,0.25) circle [radius=0.03];
\node at (0,.25) {\scriptsize{}};
\node at (3,.25) {\scriptsize{}};
\node at (8,.25) {\scriptsize{}};
\node at (11,.25) {\scriptsize{}};
\node at (1.4,.7) {\scriptsize{1}};
\node at (4.2,.7) {\scriptsize{2}};
\node at (6.7,.7) {\scriptsize{$\ell$-1}};
\node at (9.4,.7) {\scriptsize{$\ell$}};
\node at (16,.25) {\scriptsize{}};
\node at (19,.25) {\scriptsize{}};
\node at (12.2,.7) {\scriptsize{$\ell$-1}};
\node at (14.8,.7) {\scriptsize{2}};
\node at (17.4,.7) {\scriptsize{1}};
\node at (9,3.5) {\scriptsize{0}};
\node at (9.4,3) {\scriptsize{}};
\draw[->,thick] (9.2,3) to [out = 180, in = 50] (0.1,.75);
\draw[->,thick] (18.8,.75) to [out = 140, in = 0] (9.2,3);
\end{tikzpicture}
\end{equation}
\end{center}

\item $A^{(2)}_{2\ell}$
\begin{center}
\begin{equation} \label{perfect-crystal-A22l}
\begin{tikzpicture} [scale=.5]
\draw (0,.25) ellipse (.65cm and .32cm);
\draw (3,.25) ellipse (.65cm and .32cm);
\draw (8,.25) ellipse (.65cm and .32cm);
\draw (11,.25) ellipse (.65cm and .32cm);
\draw (9.4,3) ellipse (.65cm and .32cm);
\draw (16,.25) ellipse (.65cm and .32cm);
\draw (19,.25) ellipse (.65cm and .32cm);
\draw[->,thick] (.8,.25) -- (2.2,.25);
\draw[->,thick] (8.8,.25) -- (10.1,.25);
\draw[->,thick] (3.8,.25) -- (4.6,.25);
\draw[->,thick] (6.3,.25) -- (7.2,.25);
\draw[->,thick] (11.8,.25) -- (12.6,.25);
\draw[->,thick] (14.4,.25) -- (15.2,.25);
\draw[->,thick] (16.8,.25) -- (18.2,.25);
\draw [fill] (5,0.25) circle [radius=0.03];
\draw [fill] (5.5,0.25) circle [radius=0.03];
\draw [fill] (6,0.25) circle [radius=0.03];
\draw [fill] (13,0.25) circle [radius=0.03];
\draw [fill] (13.5,0.25) circle [radius=0.03];
\draw [fill] (14,0.25) circle [radius=0.03];
\node at (0,.25) {\scriptsize{}};
\node at (3,.25) {\scriptsize{}};
\node at (8,.25) {\scriptsize{}};
\node at (11,.25) {\scriptsize{}};
\node at (1.4,.7) {\scriptsize{1}};
\node at (4.2,.7) {\scriptsize{2}};
\node at (6.7,.7) {\scriptsize{$\ell$-1}};
\node at (9.4,.7) {\scriptsize{$\ell$}};
\node at (5.5,3.5) {\scriptsize{0}};
\node at (16,.25) {\scriptsize{}};
\node at (19,.25) {\scriptsize{}};
\node at (12.2,.7) {\scriptsize{$\ell$-1}};
\node at (14.8,.7) {\scriptsize{2}};
\node at (17.4,.7) {\scriptsize{1}};
\node at (13,3.5) {\scriptsize{0}};
\node at (9.4,3) {\scriptsize{}};
\draw[->,thick] (8.6,3) to [out = 180, in = 50] (0.1,.75);
\draw[->,thick] (18.8,.75) to [out = 140, in = 0] (10.2,3);
\end{tikzpicture}
\end{equation}
\end{center}

\item $A^{(2)\dagger}_{2\ell}$

\begin{center}
\begin{equation} \label{perfect-crystal-A2ldagger}
\begin{tikzpicture} [scale=.5]
\draw (0,.25) ellipse (.65cm and .32cm);
\draw (3,.25) ellipse (.65cm and .32cm);
\draw (8,.25) ellipse (.65cm and .32cm);
\draw (11,.25) ellipse (.65cm and .32cm);
\draw (14,.25) ellipse (.65cm and .32cm);
\draw (19,.25) ellipse (.65cm and .32cm);
\draw (22,.25) ellipse (.65cm and .32cm);
\draw[->,thick] (.8,.25) -- (2.2,.25);
\draw[->,thick] (8.8,.25) -- (10.1,.25);
\draw[->,thick] (3.8,.25) -- (4.6,.25);
\draw[->,thick] (6.3,.25) -- (7.2,.25);
\draw[->,thick] (11.8,.25) -- (13.2,.25);
\draw[->,thick] (14.8,.25) -- (15.6,.25);
\draw[->,thick] (17.4,.25) -- (18.2,.25);
\draw[->,thick] (20,.25) -- (21,.25);
\draw [fill] (5,0.25) circle [radius=0.03];
\draw [fill] (5.5,0.25) circle [radius=0.03];
\draw [fill] (6,0.25) circle [radius=0.03];
\draw [fill] (16,0.25) circle [radius=0.03];
\draw [fill] (16.5,0.25) circle [radius=0.03];
\draw [fill] (17,0.25) circle [radius=0.03];
\node at (0,.25) {\scriptsize{}};
\node at (3,.25) {\scriptsize{}};
\node at (8,.25) {\scriptsize{}};
\node at (11,.25) {\scriptsize{}};
\node at (1.4,.7) {\scriptsize{1}};
\node at (4.2,.7) {\scriptsize{2}};
\node at (6.7,.7) {\scriptsize{$\ell$-1}};
\node at (9.4,.7) {\scriptsize{$\ell$}};
\node at (14,.25) {\scriptsize{}};
\node at (19,.25) {\scriptsize{}};
\node at (22,.25) {\scriptsize{}};
\node at (12.4,.7) {\scriptsize{$\ell$}};
\node at (15.2,.7) {\scriptsize{$\ell$-1}};
\node at (17.8,.7) {\scriptsize{2}};
\node at (20.5,.7) {\scriptsize{1}};
\node at (11,3.7) {\scriptsize{0}};
\draw[->,thick] (11,3.2) to [out = 180, in = 50] (0.1,.75);
\draw[->,thick] (21.5,.75) to [out = 130, in = 0] (11,3.2);
\end{tikzpicture}
\end{equation}
\end{center}

\item $D^{(2)}_{\ell+1}$

\begin{center}
\begin{equation} \label{perfect-crystal-D2l+1}
\begin{tikzpicture} [scale=.5]
\draw (0,.25) ellipse (.65cm and .32cm);
\draw (3,.25) ellipse (.65cm and .32cm);
\draw (8,.25) ellipse (.65cm and .32cm);
\draw (11,.25) ellipse (.65cm and .32cm);
\draw (11,3.2) ellipse (.65cm and .32cm);
\draw (14,.25) ellipse (.65cm and .32cm);
\draw (19,.25) ellipse (.65cm and .32cm);
\draw (22,.25) ellipse (.65cm and .32cm);
\draw[->,thick] (.8,.25) -- (2.2,.25);
\draw[->,thick] (8.8,.25) -- (10.1,.25);
\draw[->,thick] (3.8,.25) -- (4.6,.25);
\draw[->,thick] (6.3,.25) -- (7.2,.25);
\draw[->,thick] (11.8,.25) -- (13.2,.25);
\draw[->,thick] (14.8,.25) -- (15.6,.25);
\draw[->,thick] (17.4,.25) -- (18.2,.25);
\draw[->,thick] (20,.25) -- (21,.25);
\draw [fill] (5,0.25) circle [radius=0.03];
\draw [fill] (5.5,0.25) circle [radius=0.03];
\draw [fill] (6,0.25) circle [radius=0.03];
\draw [fill] (16,0.25) circle [radius=0.03];
\draw [fill] (16.5,0.25) circle [radius=0.03];
\draw [fill] (17,0.25) circle [radius=0.03];
\node at (0,.25) {\scriptsize{}};
\node at (3,.25) {\scriptsize{}};
\node at (8,.25) {\scriptsize{}};
\node at (11,.25) {\scriptsize{}};
\node at (1.4,.7) {\scriptsize{1}};
\node at (4.2,.7) {\scriptsize{2}};
\node at (6.7,.7) {\scriptsize{$\ell$-1}};
\node at (9.4,.7) {\scriptsize{$\ell$}};
\node at (14,.25) {\scriptsize{}};
\node at (19,.25) {\scriptsize{}};
\node at (22,.25) {\scriptsize{}};
\node at (12.4,.7) {\scriptsize{$\ell$}};
\node at (15.2,.7) {\scriptsize{$\ell$-1}};
\node at (17.8,.7) {\scriptsize{2}};
\node at (20.5,.7) {\scriptsize{1}};
\node at (11,3.2) {\scriptsize{}};
\node at (5.5,3.7) {\scriptsize{0}};
\node at (16.5,3.7) {\scriptsize{0}};
\draw[->,thick] (10.2,3.2) to [out = 180, in = 50] (0.1,.75);
\draw[->,thick] (21.5,.75) to [out = 130, in = 0] (11.8,3.2);
\end{tikzpicture}
\end{equation}
\end{center}

\item $D^{(1)}_{\ell}$

\begin{center}
\begin{equation} \label{perfect-crystal-Dl}
\begin{tikzpicture} [scale=.5]
\draw (0,.25) ellipse (.65cm and .32cm);
\draw (3,.25) ellipse (.65cm and .32cm);
\draw (8,.25) ellipse (.65cm and .32cm);
\draw (11,1.7) ellipse (.65cm and .32cm);
\draw (11,-1.2) ellipse (.65cm and .32cm);
\draw (14,.25) ellipse (.65cm and .32cm);
\draw (19,.25) ellipse (.65cm and .32cm);
\draw (22,.25) ellipse (.65cm and .32cm);
\draw[->,thick] (.8,.25) -- (2.2,.25);
\draw[->,thick] (8.8,.6) -- (10.1,1.3);
\draw[->,thick] (8.8,-.1) -- (10.1,-.8);
\draw[->,thick] (3.8,.25) -- (4.6,.25);
\draw[->,thick] (6.3,.25) -- (7.2,.25);
\draw[->,thick] (11.8,1.3) -- (13.2,.6);
\draw[->,thick] (11.8,-.8) -- (13.2,-.1);
\draw[->,thick] (14.8,.25) -- (15.6,.25);
\draw[->,thick] (17.4,.25) -- (18.2,.25);
\draw[->,thick] (20,.25) -- (21,.25);
\draw [fill] (5,0.25) circle [radius=0.03];
\draw [fill] (5.5,0.25) circle [radius=0.03];
\draw [fill] (6,0.25) circle [radius=0.03];
\draw [fill] (16,0.25) circle [radius=0.03];
\draw [fill] (16.5,0.25) circle [radius=0.03];
\draw [fill] (17,0.25) circle [radius=0.03];
\node at (0,.25) {\scriptsize{}};
\node at (3,.25) {\scriptsize{}};
\node at (8,.25) {\scriptsize{}};
\node at (11,-1.2) {\scriptsize{}};
\node at (1.4,.7) {\scriptsize{1}};
\node at (4.2,.7) {\scriptsize{2}};
\node at (6.7,.7) {\scriptsize{$\ell$-2}};
\node at (9.2,1.5) {\scriptsize{$\ell$-1}};
\node at (9.2,-1) {\scriptsize{$\ell$}};
\node at (14,.25) {\scriptsize{}};
\node at (19,.25) {\scriptsize{}};
\node at (12.6,1.5) {\scriptsize{$\ell$}};
\node at (12.6,-1) {\scriptsize{$\ell$-1}};
\node at (15.2,.7) {\scriptsize{$\ell$-2}};
\node at (17.8,.7) {\scriptsize{2}};
\node at (20.5,.7) {\scriptsize{1}};
\node at (11,3.9) {\scriptsize{0}};
\node at (8,-3.4) {\scriptsize{0}};
\draw[->,thick] (11,3.4) to [out = 180, in = 50] (3,.75);
\draw[->,thick] (21.5,.75) to [out = 130, in = 0] (11,3.4);
\draw[->,thick] (18.5,-.25) to [out = 230, in = 0] (8,-2.9);
\draw[->,thick] (8,-2.9) to [out = 180, in = 310] (.3,-.25);
\end{tikzpicture}
\end{equation}
\end{center}

\item $B^{(1)}_{\ell}$

\begin{center}
\begin{equation} \label{perfect-crystal-Bl}
\begin{tikzpicture} [scale=.5]
\draw (0,.25) ellipse (.65cm and .32cm);
\draw (3,.25) ellipse (.65cm and .32cm);
\draw (8,.25) ellipse (.65cm and .32cm);
\draw (11,.25) ellipse (.65cm and .32cm);
\draw (14,.25) ellipse (.65cm and .32cm);
\draw (19,.25) ellipse (.65cm and .32cm);
\draw (22,.25) ellipse (.65cm and .32cm);
\draw[->,thick] (.8,.25) -- (2.2,.25);
\draw[->,thick] (8.8,.25) -- (10.1,.25);
\draw[->,thick] (3.8,.25) -- (4.6,.25);
\draw[->,thick] (6.3,.25) -- (7.2,.25);
\draw[->,thick] (11.8,.25) -- (13.2,.25);
\draw[->,thick] (14.8,.25) -- (15.6,.25);
\draw[->,thick] (17.4,.25) -- (18.2,.25);
\draw[->,thick] (20,.25) -- (21,.25);
\draw [fill] (5,0.25) circle [radius=0.03];
\draw [fill] (5.5,0.25) circle [radius=0.03];
\draw [fill] (6,0.25) circle [radius=0.03];
\draw [fill] (16,0.25) circle [radius=0.03];
\draw [fill] (16.5,0.25) circle [radius=0.03];
\draw [fill] (17,0.25) circle [radius=0.03];
\node at (0,.25) {\scriptsize{}};
\node at (3,.25) {\scriptsize{}};
\node at (8,.25) {\scriptsize{}};
\node at (11,.25) {\scriptsize{}};
\node at (1.4,.7) {\scriptsize{1}};
\node at (4.2,.7) {\scriptsize{2}};
\node at (6.7,.7) {\scriptsize{$\ell$-1}};
\node at (9.4,.7) {\scriptsize{$\ell$}};
\node at (14,.25) {\scriptsize{}};
\node at (19,.25) {\scriptsize{}};
\node at (12.4,.7) {\scriptsize{$\ell$}};
\node at (15.2,.7) {\scriptsize{$\ell$-1}};
\node at (17.8,.7) {\scriptsize{2}};
\node at (20.5,.7) {\scriptsize{1}};
\node at (11,3.9) {\scriptsize{0}};
\node at (8,-3.4) {\scriptsize{0}};
\draw[->,thick] (11,3.4) to [out = 180, in = 50] (3,.75);
\draw[->,thick] (21.5,.75) to [out = 130, in = 0] (11,3.4);
\draw[->,thick] (18.5,-.25) to [out = 230, in = 0] (8,-2.9);
\draw[->,thick] (8,-2.9) to [out = 180, in = 310] (.1,-.25);
\end{tikzpicture}
\end{equation}
\end{center}

\item $A^{(2)}_{2\ell-1}$

\begin{center}
\begin{equation} \label{perfect-crystal-A2(2l-1)}
\begin{tikzpicture} [scale=.5]
\draw (0,.25) ellipse (.65cm and .32cm);
\draw (3,.25) ellipse (.65cm and .32cm);
\draw (8,.25) ellipse (.65cm and .32cm);
\draw (11,.25) ellipse (.65cm and .32cm);
\draw (16,.25) ellipse (.65cm and .32cm);
\draw (19,.25) ellipse (.65cm and .32cm);
\draw[->,thick] (.8,.25) -- (2.2,.25);
\draw[->,thick] (8.8,.25) -- (10.1,.25);
\draw[->,thick] (3.8,.25) -- (4.6,.25);
\draw[->,thick] (6.3,.25) -- (7.2,.25);
\draw[->,thick] (11.8,.25) -- (12.6,.25);
\draw[->,thick] (14.4,.25) -- (15.2,.25);
\draw[->,thick] (16.8,.25) -- (18.2,.25);
\draw [fill] (5,0.25) circle [radius=0.03];
\draw [fill] (5.5,0.25) circle [radius=0.03];
\draw [fill] (6,0.25) circle [radius=0.03];
\draw [fill] (13,0.25) circle [radius=0.03];
\draw [fill] (13.5,0.25) circle [radius=0.03];
\draw [fill] (14,0.25) circle [radius=0.03];
\node at (0,.25) {\scriptsize{}};
\node at (3,.25) {\scriptsize{}};
\node at (8,.25) {\scriptsize{}};
\node at (11,.25) {\scriptsize{}};
\node at (1.4,.7) {\scriptsize{1}};
\node at (4.2,.7) {\scriptsize{2}};
\node at (6.7,.7) {\scriptsize{$\ell$-1}};
\node at (9.4,.7) {\scriptsize{$\ell$}};
\node at (16,.25) {\scriptsize{}};
\node at (12.2,.7) {\scriptsize{$\ell$-1}};
\node at (14.8,.7) {\scriptsize{2}};
\node at (17.4,.7) {\scriptsize{1}};
\node at (9.4,3.5) {\scriptsize{0}};
\node at (6.4,-3.2) {\scriptsize{0}};
\draw[->,thick] (9.4,3) to [out = 180, in = 50] (3,.75);
\draw[->,thick] (18.8,.75) to [out = 140, in = 0] (9.4,3);
\draw[->,thick] (6.4,-2.5) to [out = 180, in = 310] (0,-.25);
\draw[->,thick] (15.8,-.25) to [out = 230, in = 0] (6.4,-2.5);
\end{tikzpicture}
\end{equation}
\end{center}

\end{itemize}

\subsubsection{Paths on $B^{1,1}$}
\label{sec-path-defn}

\begin{definition}
Let $X_\ell$ be one of the affine types listed
in Section \ref{Cartan-datum-section}
and let $I$ be the
indexing set of its Dynkin nodes. A type $X_\ell$ \emph{path} $p$ of
length $k$, is a function $p: \{0, 1, \dots, k-1\} \rightarrow I$ such
that there is a directed walk in the type $X_\ell$ crystal $B^{1,1}$
whose $i$th step corresponds to a $p(i)$-arrow.  \end{definition}

We specify the edge color data
but not node data of the walk.
There might then appear to be some danger of
ambiguity of two different walks corresponding to the same path,
as $B^{1,1}$ may have more than one arrow of a given color.
Indeed, the single step paths $p: \{0\} \rightarrow I$ in general do
not specify a unique walk in $B^{1,1}$. 
However, for paths with length $k$ with $k
\geq 2$, the definition above is sufficient to specify a unique walk. 

\begin{proposition} \label{unique-path-prop}
Let $p: \{0, 1, \dots, k-1\} \rightarrow I$ be a type $X_\ell$ path of length $k$. If $k > 1$, then there is a unique walk in $B^{1,1}$ that corresponds to $p$.
\end{proposition}

\begin{proof}
This proposition is based on three observations about the graph $B^{1,1}$: 
\begin{itemize}
\item A node $a$ never has two incoming (respectively outgoing) arrows of the same color (this of course is true of all crystals, even when $\p_{i}(a) > 1$ or $\ep{i}(a) > 1$).
\item There are at most two $i$-arrows in $B^{1,1}$ and hence at most two nodes $a$ and $b$ with an incoming (respectively outgoing) $i$-arrow. 
\item If $a$ and $b$ are nodes that each have an incoming (respectively outgoing) $i$-arrow, $a$ has an outgoing (respectively incoming) $j_1$-arrow and $b$ has an outgoing (respectively incoming) $j_2$-arrow, then $j_1 \neq j_2$.

\end{itemize}
By the first observation, a walk corresponding to $p$ in $B^{1,1}$ is uniquely determined once we have chosen its initial node. 

By the second observation, if $p(0) = i$ there are at most two choices ($a$ and $b$) for the second node of the walk. If $k > 1$ then there is a second $p(1)$-arrow leaving either $a$ or $b$. By the third observation, $a$ and $b$ cannot both have an outgoing $p(1)$-arrow. Assume without loss of generality that $a$ has an outgoing $p(1)$-arrow. Then $p$ uniquely corresponds to the walk in $B^{1,1}$ that starts by entering node $a$ via a $p(0)$-arrow and leaves $a$ by a $p(1)$-arrow.
\end{proof}

Because of the uniqueness implied by Proposition \ref{unique-path-prop} we will henceforth use the terms ``type $X_\ell$ path'' and ``walk in $B^{1,1}$'' interchangeably when $k > 1$.

\begin{example} \label{path-example}
Let $p: \{0,1,2\} \rightarrow I$ be the path of length 3 in type $D_{5}^{(1)}$, defined such that $p(0) = 0$, $p(1) = 2$,  $p(2) = 3$. We mark the walk corresponding to $p$ with a dotted line.
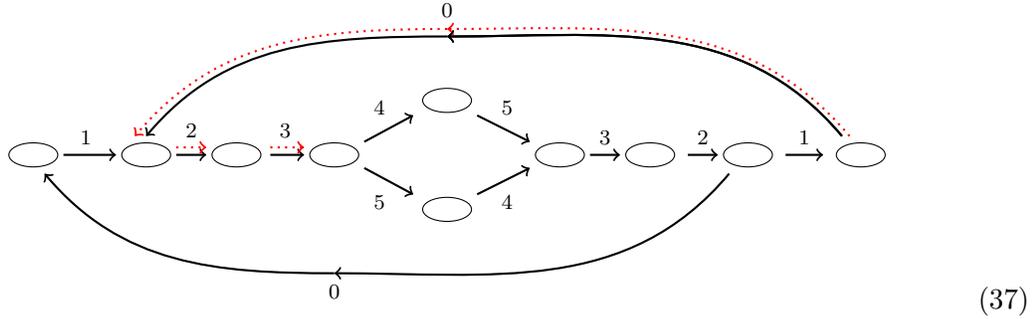
\begin{center}
\begin{equation} \label{walk-example} 
\begin{tikzpicture} [scale=.5]
\draw (0,.25) ellipse (.65cm and .32cm);
\draw (3,.25) ellipse (.65cm and .32cm);
\draw (8,.25) ellipse (.65cm and .32cm);
\draw (11,1.7) ellipse (.65cm and .32cm);
\draw (11,-1.2) ellipse (.65cm and .32cm);
\draw (14,.25) ellipse (.65cm and .32cm);
\draw (19,.25) ellipse (.65cm and .32cm);
\draw (22,.25) ellipse (.65cm and .32cm);
\draw[->,thick] (.8,.25) -- (2.2,.25);
\draw[->,thick] (8.8,.6) -- (10.1,1.3);
\draw[->,thick] (8.8,-.1) -- (10.1,-.8);
\draw[->,thick] (3.8,.25) -- (4.6,.25);
\draw[->,thick,dotted,red] (3.8,.45) -- (4.6,.45);
\draw[->,thick] (6.3,.25) -- (7.2,.25);
\draw[->,thick,dotted,red] (6.3,.45) -- (7.2,.45);
\draw[->,thick] (11.8,1.3) -- (13.2,.6);
\draw[->,thick] (11.8,-.8) -- (13.2,-.1);
\draw[->,thick] (14.8,.25) -- (15.6,.25);
\draw[->,thick] (17.4,.25) -- (18.2,.25);
\draw[->,thick] (20,.25) -- (21,.25);
\draw (5.4,.25) ellipse (.65cm and .32cm);
\draw (16.4,.25) ellipse (.65cm and .32cm);
\node at (0,.25) {\scriptsize{}};
\node at (3,.25) {\scriptsize{}};
\node at (8,.25) {\scriptsize{}};
\node at (11,-1.2) {\scriptsize{}};
\node at (1.4,.7) {\scriptsize{1}};
\node at (4.2,.9) {\scriptsize{2}};
\node at (6.7,.9) {\scriptsize{3}};
\node at (9.2,1.5) {\scriptsize{4}};
\node at (9.2,-1) {\scriptsize{5}};
\node at (14,.25) {\scriptsize{}};
\node at (19,.25) {\scriptsize{}};
\node at (12.6,1.5) {\scriptsize{5}};
\node at (12.6,-1) {\scriptsize{4}};
\node at (15.2,.7) {\scriptsize{3}};
\node at (17.8,.7) {\scriptsize{2}};
\node at (20.5,.7) {\scriptsize{1}};
\node at (11,4.1) {\scriptsize{0}};
\node at (8,-3.4) {\scriptsize{0}};
\draw[->,thick] (11,3.4) to [out = 180, in = 50] (3,.75);
\draw[->,thick] (21.5,.75) to [out = 130, in = 0] (11,3.4);
\draw[->,thick,dotted,red] (11,3.6) to [out = 180, in = 50] (2.7,.75);
\draw[->,thick,dotted,red] (21.7,.75) to [out = 130, in = 0] (11,3.6);
\draw[->,thick] (21.5,.75) to [out = 130, in = 0] (11,3.4);
\draw[->,thick] (18.5,-.25) to [out = 230, in = 0] (8,-2.9);
\draw[->,thick] (8,-2.9) to [out = 180, in = 310] (.3,-.25);
\end{tikzpicture}
\end{equation}
\end{center}
\end{example}

\begin{definition}
For a path $p: \{0,1, \dots, k-1\} \rightarrow I$ we call the arrow corresponding to $p(0)$ the \emph{tail} of $p$, and the arrow corresponding to $p(k-1)$ the \emph{head} of $p$. An extension to the tail of $p$ by a $j$-arrow is a path $p':\{0,1,\dots, k\} \rightarrow I$ such that $p'(t) = p(t-1)$ for $1 \leq t \leq k$ and $p'(0) = j$. An extension to the head of $p$ by a $j$-arrow, is a path $p'':\{0,1, \dots, k\} \rightarrow I$ such that $p''(t) = p(t)$ for $0 \leq t \leq k-1$ and $p''(k) = j$.
\end{definition}

Let $\pi(j)$ be the length 1 path $\pi(j): \{0\} \rightarrow I$,
$\pi(j)(0) = j$. For a path $p$, we denote the extension of its tail by
a $j$-arrow by $\exttailby{j} p$ and the extension of its head by a
$j$-arrow by $p \extheadby{j}$. We can think of extension as
concatenation of paths. If the tail (respectively head) of $p$ cannot
be extended by a $j$-arrow then we set
$\exttailby{j} p = \zero$
(respectively $p \extheadby{j} = \zero$).

The set of colors of arrows that can extend the tail of a path $p$ of length $k > 1$ is denoted $\pextminus{p}$,
\begin{equation}
\pextminus{p} := \{ \; j \; | \; \exttailby{j} p \neq \zero \; \}.
\end{equation}
The set of colors of arrows that can extend the head of $p$ is denoted $\pextplus{p}$,
\begin{equation}
\pextplus{p} := \{ \; j \; | \;p \extheadby{j} \neq \zero \; \}.
\end{equation}
When $\pextminus{p}$ (respectively $\pextplus{p}$) contains a single element $i$, we allow ourselves the convenience of writing $p(-1) = i$ (respectively $p(k) = i$).

\begin{example}
In example \ref{path-example} 
\begin{equation}
\pextplus{p} = \{4,5\} \quad\quad \pextminus{p} = \{1\}.
\end{equation}
so we write $p(-1) = 1$.
\end{example}

\begin{remark} \label{extension-of-length-1-path}
In general, when $p$ has length 1, then $\pextminus{p}$ and $\pextplus{p}$ are not well-defined. However, the symmetry of the $B^{1,1}$ graphs means that $\pextminus{p}\cup \pextplus{p}$ is well-defined even for length 1 paths.
\end{remark}

\section{Key definitions: class $\ClassA$, $\ClassB$, $\ClassD$
nodes and cyclotomic paths}

Let $i, j, j',a,a',b,b' \in I$. We classify the arrows of $B^{1,1}$ into three classes: $\ClassA$, $\ClassB$, and $\ClassD$. 

\begin{itemize}

\item An $i$-arrow belongs to class $\ClassB$ if it is adjacent to another $i$-arrow. That is, it either enters a node that has an outgoing $i$-arrow or it departs from a node with an incoming $i$-arrow. For instance, in the diagram below both $i$-arrows are of class $\ClassB$.

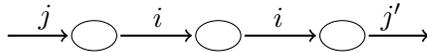
\begin{center}
\begin{tikzpicture}
\draw (0,0) ellipse (.3cm and .2cm);
\draw (-1.65,0) ellipse (.3cm and .2cm);
\draw (1.65,0) ellipse (.3cm and .2cm);
\draw[->,thick] (-1.3,0) -- (-.35,0);
\draw[->,thick] (.35,0) -- (1.3,0);
\draw[->,thick] (-2.8,0)--(-2,0);
\draw[->,thick] (2,0)--(2.8,0);
\node at (-.8,.25) {$i$};
\node at (.8,.25) {$i$};
\node at (-2.3,.25) {$j$};
\node at (2.3,.25) {$j'$};

\end{tikzpicture}
\end{center}

\item A pair of arrows $a$ and $a'$ are said to be a class $\ClassD$ pair if they are the first and second step respectively of a bifurcation (a bifurcation occurs at a node of degree 3). Below both $(a,a')$ and $(b,b')$ are class $\ClassD$ pairs.

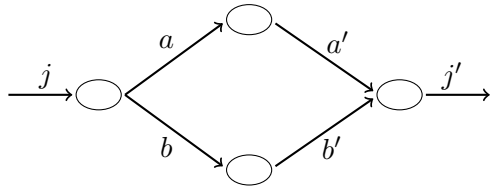
\begin{center}
\begin{tikzpicture}
\draw (0,0) ellipse (.3cm and .2cm);
\draw (4,0) ellipse (.3cm and .2cm);
\draw (2,1) ellipse (.3cm and .2cm);
\draw (2,-1) ellipse (.3cm and .2cm);
\draw[->,thick] (.35,0) -- (1.65,.95);
\draw[->,thick] (.35,0) -- (1.65,-.95);
\draw[->,thick] (2.35,.95) -- (3.65,.05);
\draw[->,thick] (2.35,-.95) -- (3.65,-.05);
\draw[->,thick] (-1.2,0) -- (-.35,0);
\draw[->,thick] (4.35,0) -- (5.2,0);
\node at (.9,.7) {$a$};
\node at (.9,-.7) {$b$};
\node at (3.18,.7) {$a'$};
\node at (3.1,-.7) {$b'$};
\node at (-.7,.25) {$j$};
\node at (4.7,.25) {$j'$};

\end{tikzpicture}
\end{center}

For all $B^{1,1}$ it will actually be the case that $a = b'$, $a' = b$, and $j = j'$,
whenever $(a,a')$ is a class $\ClassD$ pair.
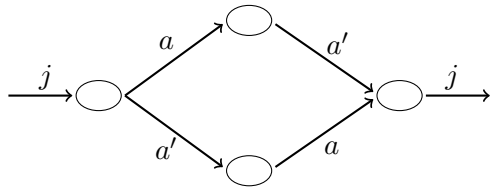
\begin{center}
\begin{tikzpicture}
\draw (0,0) ellipse (.3cm and .2cm);
\draw (4,0) ellipse (.3cm and .2cm);
\draw (2,1) ellipse (.3cm and .2cm);
\draw (2,-1) ellipse (.3cm and .2cm);
\draw[->,thick] (.35,0) -- (1.65,.95);
\draw[->,thick] (.35,0) -- (1.65,-.95);
\draw[->,thick] (2.35,.95) -- (3.65,.05);
\draw[->,thick] (2.35,-.95) -- (3.65,-.05);
\draw[->,thick] (-1.2,0) -- (-.35,0);
\draw[->,thick] (4.35,0) -- (5.2,0);
\node at (.9,.7) {$a$};
\node at (.9,-.7) {$a'$};
\node at (3.18,.7) {$a'$};
\node at (3.1,-.7) {$a$};
\node at (-.7,.25) {$j$};
\node at (4.7,.25) {$j$};

\end{tikzpicture}
\end{center}

\item All other arrows in $B^{1,1}$ are class $\ClassA$. For large $\ell$, the vast majority of arrows will belong to class $\ClassA$.

\end{itemize}

\begin{remark}
It is not hard to see from inspection that within a given type, our classification of arrows in $B^{1,1}$ descends to a classification of elements of $I$ in that type. In other words within a $B^{1,1}$ diagram, there are never two $i$-arrows that belong to different classes. So it makes sense to talk about $i \in I$ being of a particular class. Table \ref{i-class-table} classifies the elements of $I$ for all types $X_\ell$ considered in this paper. 
\end{remark}

\begin{table}
\begin{center}
{\tabulinesep=1.2mm
 \begin{tabu}{| l | c | c | c |}
 \hline
 Type & class $\ClassA$ & class $\ClassB$ & class $\ClassD$ pairs  \\
 \hline \hline
 $A^{(1)}_{\ell}$ & $0,1,\dots, \ell-1, \ell$ & & \\
 $C^{(1)}_{\ell}$ & $0,1,\dots, \ell-1, \ell$ & & \\
 $A^{(2)}_{2\ell }$ & $1,\dots, \ell-1, \ell$ & $0$ & \\
 $A^{(2)\dagger}_{2\ell}$ & $0,1,\dots, \ell-1$ & $\ell$ & \\
 $D^{(2)}_{\ell+1}$ & $1,\dots, \ell-1$ & $0,\ell$ & \\
 $D^{(1)}_{\ell}$ & $2,3,\dots, \ell-2$ & & $(0,1),(\ell-1,\ell)$ \\
 $B^{(1)}_{\ell}$ & $2,\dots,\ell-1$ & $\ell$ & $(0,1)$ \\
 $A^{(2)}_{2\ell-1}$ & $2,\dots,\ell-1,\ell$ & & $(0,1)$ \\
\hline
\end{tabu}}
\end{center}
\caption{\label{i-class-table} Classification of elements of $I$ for each type $X_\ell$}
\label{node-classification}
\end{table}

Below is one of our key definitions.
\begin{definition}
Let $p$ be a path of length $k \geq 1$ in $B^{1,1}$. $p$ is called a \emph{cyclotomic path of tail weight} $(\Lambda_{i_1},\Lambda_{i_2})$ if the following hold:
\begin{enumerate} 
\item $p(0) = i_2$ and $p(1) \neq i_2$. If $p$ has length 1 then
$ p \extheadby{i_2} = \zero$, i.e. the head of $p$
cannot be extended by an $i_2$-arrow.  \item $p(0)$ and $p(1)$ are not
a class $\ClassD$ pair. If $p$ has length 1 then there
is no $i \in I$ such that the head of $p$ can be extended by an
$i$-arrow and $(p(0),i)$ is a class $\ClassD$ pair.
\item $\pextminus{p} = \{i_1\}$, i.e. $p$ has a unique extension by an $i_1$-arrow to its tail.
\item $(\exttailby{i_1}(\exttailby{i_1} p)) = \zero$, i.e.
 $p$ cannot have its tail extended twice by $i_1$-arrows.
\end{enumerate}
\end{definition}

The following proposition follows from inspection of the $B^{1,1}$ 
crystal graphs.

\begin{proposition}\label{cyclotomic-path-lemma}
\begin{enumerate} 
\item Let $X_\ell$ be any type not equal to $D^{(2)}_3$, $D^{(1)}_4$ and $B^{(1)}_3$. Then for all $i \in I$ there exists some $j \in I$ such that there is a cyclotomic path $p$ of tail weight $(\Lambda_{j},\Lambda_{i})$. 
\item In type $D^{(2)}_3$ so long as $i \neq 1$, i.e. $i \in \{0,2\}$, there is $j \in I$ such that there exists a cyclotomic path $p$ of tail weight $(\Lambda_j,\Lambda_i)$.
\item In type $D^{(1)}_4$ and $B^{(1)}_3$ so long as $i \neq 2$, i.e. $i \in \{0,1,3\}$, there is some $j \in I$ such that there exists a cyclotomic path $p$ of tail weight $(\Lambda_j,\Lambda_i)$.
\end{enumerate}
When the above paths exists they can be of any length $k \in \MB{Z}_{\geq 1}$.

\end{proposition}

\begin{definition}
We refer to $1 \in I$ in type $D^{(2)}_3$ and $2 \in I$ in types $D^{(1)}_4$ and $B^{(1)}_3$ as \emph{forbidden} elements of $I$.
\end{definition}

\begin{remark}
Above we saw that the one-to-one correspondence between length $k$ paths in $B^{1,1}$ and length k directed walks in $B^{1,1}$ breaks down when $k = 1$. It is important to note that this does not occur for length 1 cyclotomic paths because of the extra information imposed by the unique tail extension.
\end{remark}


\section{The KLR algebra $R$ and some functors} \label{sec-KLR}
\subsection{Definition of the KLR algebra $R(\nu)$}
\label{basic-definitions}

Fix an indeterminant $q$.
In what follows we let $q_i = q^{\frac{(\alpha_i,\alpha_i)}{2}}$, 
\begin{equation}
[k]_i = q_i^{k-1} + q_i^{k-3} + \dots + q_i^{1-k} \quad \text{and} \quad [k]_i! = [k]_i [k-1]_i \dots [1]_i.
\end{equation}
For $\nu = \sum_{i \in I} \nu_i \alpha_i$ in $Q^+$ with $|\nu| = m$, we define $\seq(\nu)$ to be all sequences 
\begin{equation}
\und{i} = (i_{1}, i_{2}, \dots, i_{m})
\end{equation}
such that $i_k$ appears $\nu_k$ times. For $\und{i} \in \seq(\nu)$ and $\und{j} \in \seq(\mu)$, $\und{ij}$, will denote the concatenation of the two sequences unless otherwise specified. It follows that $\und{ij} \in \seq(\nu + \mu)$. We write 
\begin{equation}
i^n = (\underbrace{i,i, \dots ,i}_{n}).
\end{equation}
There is a left  action of the symmetric group, $\Sy{m}$, on $\seq(\nu)$ defined by,
\begin{equation} \label{sym-grp-action-seq}
s_{k}(\und{i}) = s_k\Big( i_{1}, i_{2}, \dots ,i_{k}, i_{k+1}, \dots ,i_m\Big) = (i_{1}, i_{2}, \dots, i_{k+1}, i_{k}, \dots, i_{m})
\end{equation}
where $s_k$ is the adjacent transposition in $\Sy{m}$ that interchanges $k$ and $k+1$. 

Our presentation of KLR algebras below follows that found in \cite{KL09} 
and \cite{KL11}. Using the more general definition with Rouquier's parameters $Q_{i,j}(u,v)$ will not change the results or proofs in this paper, as they concern crystal-theoretic phenomena. There is also a
useful diagrammatic presentation of KLR algebras which can be found in
\cite{KL09}, \cite{KL11}. By results of Brundan-Kleshchev \cite{BK09a},
\cite{BK09b}, there is an isomorphism between the cyclotomic KLR algebra, $R^{\Lambda}(\nu)$, of type $A^{(1)}_{\ell}$, and
$H^{\Lambda}_{\nu}$ where $H^{\Lambda}_\nu$ is a block of the
cyclotomic Hecke algebra $H^{\Lambda}_m$ as defined in
\cite{AK94, BM94, Che87}. Hence readers unfamiliar with KLR algebras
can translate all statements and proofs for the type $A^{(1)}_\ell$,
$\Lambda = \Lambda_i$
case in terms of Hecke algebras
throughout the paper (historically, this is the original setting in
which the theorems from this paper were proved). In fact for type
$A^{(1)}_\ell$ the reader can think of all results as being stated for
$\MB{F}_{\ell+1} \Sy{m}$ in the case that $\ell+1$ is prime.

 For $\nu \in Q^+$ with $|\nu| = m$, the {\emph{KLR
algebra}}
$R(\nu)$ associated with Cartan matrix $[a_{ij}]_{i,j \in I}$ is the associative, graded, unital $\MB{C}$-algebra generated
by \begin{equation}
1_{\und{i}} \;\; \text{for} \;\; \und{i} \in \seq(\nu),
\quad x_r \;\; \text{for} \;\; 1 \leq r \leq m,
\quad  \psi_r \;\; \text{for} \;\; 1 \leq r \leq m-1,
\end{equation}
subject to the following relations, where $\und{i}, \und{j} \in
\seq(\nu)$:
\begin{equation}
\label{eq-idem}
1_{\und{i}}1_{\und{j}} = \delta_{\und{i},\und{j}}1_{\und{i}}, \quad\quad x_r1_{\und{i}} = 1_{\und{i}}x_r, \quad\quad \psi_r 1_{\und{i}} = 1_{s_r(\und{i})}\psi_r,  \quad\quad x_rx_t = x_tx_r,
\end{equation}
\begin{align}
 \psi_r\psi_t &= \psi_t \psi_r \qquad \qquad
 \text{if} \; |r-t| > 1, \\
\label{square-relation}
 \psi_r\psi_r1_{\und{i}} &= 
   \begin{cases}
    0 & \text{if $i_r = i_{r+1}$,}\\
    1_{\und{i}} & \text{if} \; (\alpha_{i_r},\alpha_{i_{r+1}}) = 0,\\
    (x_r^{-a_{i_r i_{r+1}}}+ x_{r+1}^{-a_{i_{r+1} i_r}})1_{\und{i}} &
	\text{if $(\alpha_{i_r}, \alpha_{i_{r+1}}) \neq 0$
	and $i_r \neq i_{r+1}$},
  \end{cases} 
\\
(\psi_r \psi_{r+1} \psi_r - \psi_{r+1}\psi_r \psi_{r+1})1_{\und{i}}&=
  \label{braid-relation}
   \begin{cases}
\begin{displaystyle}
\sum^{-a_{i_r i_{r+1}} -1}_{t = 0}
x_r^tx_{r+2}^{-a_{i_r i_{r+1}}-1-t}1_{\und{i}}
\end{displaystyle}
& \text{if $i_r = i_{r+2}$ and $(\alpha_{i_r},\alpha_{i_{r+1}}) \neq 0$,}\\
     \quad 0 & \text{otherwise,}
  \end{cases} 
\\
 \label{dot-past-crossing}
(\psi_rx_t - x_{s_r(t)}\psi_r)1_{\und{i}} &= 
   \begin{cases}
    1_{\und{i}} & \quad \quad \text{if $t = r$ and $i_r = i_{r+1}$},\\
   -1_{\und{i}} & \quad \quad \text{if $t = r+1$ and $i_r = i_{r+1}$},\\
    0 & \quad \quad \text{otherwise.}
  \end{cases} 
\end{align}

The elements $1_{\und{i}}$ are idempotents in $R(\nu)$
by \eqref{eq-idem} and the identity element is
\begin{equation}
1_\nu = \sum_{\und{i} \in \seq(\nu)} 1_{\und{i}}.
\end{equation}
Thus, as a vector space $R(\nu)$ decomposes as, 
\begin{equation}
R(\nu) = \bigoplus_{\und{i},\und{j} \in \seq(\nu)} 1_{\und{i}} R(\nu) 1_{\und{j}}.
\end{equation}
The generators of $R(\nu)$ are graded via
\begin{equation}
\deg(1_{\und{i}}) = 0, \;\;\; \deg(x_r1_{\und{i}}) = (\alpha_{i_r},\alpha_{i_r}), \;\;\; \deg(\psi_r1_{\und{i}}) = -(\alpha_{i_r}, \alpha_{i_{r+1}}).
\end{equation}
We define
\begin{equation}
R = \bigoplus_{\nu \in Q^+}R(\nu).
\end{equation}
Notice that while $R(\nu)$ is unital, $R$ is not. 

For each $w \in \Sy{m}$ we fix once and for all a reduced expression
\begin{equation}
\wh{w} = s_{i_1}s_{i_2} \dots s_{i_t}.
\end{equation} 
Here the $s_k$ are Coxeter generators of $\Sy{m}$, and $t$ is the Coxeter length. Let $\psi_{\wh{w}} = \psi_{i_1} \psi_{i_2} \dots \psi_{i_t}$ correspond to the chosen reduced expression $\wh{w}$ of $w$. For $\und{i},\und{j} \in \seq(\nu)$, let $_{\und{j}}\Sy{\und{i}}$ be the permutations
in  $\Sy{m}$ that take $\und{i}$ to $\und{j}$.

\begin{theorem} \label{basis-theorem} \cite[Theorem 2.5]{KL09} 
As a $\mathbb{C}$-vector space $1_{\und{j}}R(\nu)1_{\und{i}}$ has basis,
\begin{equation}
\{\psi_{\wh{w}} x_1^{b_1} \dots x_{m}^{b_m} 1_{\und{i}} \; | \; w \in
\, _{\und{j}}\!{\Sy{}}_{\und{i}}, \; b_r \in \mathbb{Z}_{\geq 0}\}.
\end{equation}
\end{theorem}

It is known that all simple $R(\nu)$-modules are finite dimensional \cite{KL09}. For this reason, in this paper we only consider the category of finite-dimensional KLR-modules denoted by $R(\nu) \Mod$ and $R \Mod$.
We write $G_0(R)$ for the Grothendieck group.

We often refer to $1_{\und{i}}M$ as the $\und{i}$-weight space
of $M$ and any $0 \neq v \in 1_{\und{i}}M$ as a weight vector.
A weight basis is a basis consisting of weight vectors.

We define the graded character of an $R(\nu)$-module to be
\begin{equation}
\Char(M) = \sum_{\und{i} \in \seq(\nu)} \text{gdim}(1_{\und{i}}M) \cdot [\und{i}].
\end{equation}
Here $\text{gdim}(1_{\und{i}}M)$ is an element of
$\MB{Z}[q,q^{-1}]$, and hence $\Char(M)$ is an element of the free
$\MB{Z}[q,q^{-1}]$-module generated by all $[\und{i}]$ for $\und{i} \in
\seq(\nu)$. We will let $\ch(M)$ denote the multiset that is the
support of $\Char(M)|_{q =1}$ so that 
\begin{equation}
\Char(M)|_{q=1} = \sum_{[\und{i}] \in \ch(M)} [\und{i}].
\end{equation}
Our notational convention
is to write $\und{i} \in \seq(\nu)$ but write $[\und{i}] \in
\ch(M)$. Since characters are an important combinatorial tool, it is
worthwhile to set a special notation for them. 

\begin{remark} \label{character-determines-simple}
A simple module is determined by its character, hence $[M] \in G_0(R)$
the Grothendieck group, is also determined by its character \cite{KL09}, \cite{KL11}.
\end{remark}

Because $R$ is a graded algebra, we will only work with homomorphisms between $R$-modules that are either degree preserving or degree homogeneous. We denote the $\mathbb{C}$-vector space of degree preserving homomorphisms between $R(\nu)$-modules $M$ and $N$ by $\Hom(M,N)$. Since any homogeneous homomorphism can be interpreted as degree preserving by shifting the grading on our target or source module, we can write the $\mathbb{C}$-vector space of homogeneous homomorphisms between $M$ and $N$, $\HOM(M,N)$, by
\begin{equation}
\HOM(M,N) = \bigoplus_{k \in \mathbb{Z}} \Hom(M,N\{k\}).
\end{equation}
While the grading is important, it is shown in \cite{KL09} that there is a unique grading on a simple $R$-module up to overall grading shift. Since this paper concerns simple modules, we will rarely use or discuss the grading.
All isomorphisms between modules will be taken up to overall
grading shift.

\begin{remark} \label{x-nilpotent}
Because $x_r1_{\und{i}} \in R(\nu)$ is always positively graded for $1 \leq r \leq |\nu|$ and $\und{i} \in \seq(\nu)$, then on a finite dimensional $R(\nu)$-module, $M$, $x_r1_{\und{i}}$ will always act nilpotently.
\end{remark}



\subsection{Induction, co-induction, and restriction}

It was shown in \cite{KL09} and \cite{KL11} that for $\nu, \mu \in Q^+$ there is a non-unital embedding
\begin{equation}
R(\nu) \otimes R(\mu) \hookrightarrow R(\nu + \mu).
\end{equation}
This map sends the idempotent $1_{\und{i}} \otimes 1_{\und{j}}$ to $1_{\und{ij}}$. The identity $1_{\nu} \otimes 1_{\mu}$ of $R(\nu) \otimes R(\mu)$ has as its image
\begin{equation}
\sum_{\und{i} \in \seq(\nu)} \sum_{\und{j} \in \seq(\mu)} 1_{\und{ij}}.
\end{equation}
Using this embedding one can define induction and restriction functors, 
\begin{align}
\ind_{\nu, \mu}^{\nu + \mu}: (R(\nu) \otimes R(\mu)) \Mod & \rightarrow R(\nu + \mu) \Mod\\
M & \mapsto R(\nu + \mu) \otimes_{R(\nu) \otimes R(\mu)} M
\end{align}
and
\begin{equation}
\res_{\nu, \mu}^{\nu + \mu}: R(\nu + \mu) \Mod \rightarrow (R(\nu) \otimes R(\mu)) \Mod.
\end{equation}
In the future we will write $\ind_{\nu,\mu}^{\nu+\mu} = \ind$ and $\res_{\nu,\mu}^{\nu+\mu} = \res$ when the algebras are understood from the context. More generally we can extend this embedding to finite tensor products 
\begin{equation}
R(\nu^{(1)}) \otimes R(\nu^{(2)}) \otimes \dots \otimes R(\nu^{(k)}) \hookrightarrow R(\nu^{(1)} + \nu^{(2)} + \dots + \nu^{(k)}).
\end{equation}
We refer to the image of this embedding as a \emph{parabolic
subalgebra} and denote it by
$R(\und{\nu}) \subset R(\nu^{(1)} + \dots + \nu^{(k)})$.
We denote the image of the identity under this embedding
as $1_{\und{\nu}}$.
It follows from Theorem \ref{basis-theorem} that
$R(\nu^{(1)} + \nu^{(2)} + \dots + \nu^{(k)})1_{\und{\nu}}$
is a free right $R(\und{\nu})$-module,
 and $1_{\und{\nu}}R(\nu^{(1)} + \nu^{(2)} + \dots + \nu^{(k)})$
is a free left $R(\und{\nu})$-module.

 Let $m_i = |\nu^{(i)}|$ and set
\begin{equation}
P = (m_1, \dots, m_k) \quad \text{and} \quad \Sy{P} = \Sy{m_1} \times \Sy{m_1} \times \dots \times \Sy{m_k}.
\end{equation}
Let $\Sy{m_1 + \dots + m_k} / \Sy{P}$ be the collection of minimal
length left coset representatives of $\Sy{P}$ in $\Sy{m_1 + \dots +
m_k}$ and $\Sy{P}\backslash \Sy{m_1 + \dots + m_k}$ be the collection
of minimal length right coset representatives of
$\Sy{P}$ in $\Sy{m_1 + \dots + m_k}$.
We can construct a weight basis for an induced module as follows.
If $M$ is an $R(\und{\nu})$-module with
weight basis $U$ then
$\ind_{\und{\nu}}^{\nu^{(1)} + \dots + \nu^{(k)}} M$
has weight basis
\begin{equation} \label{ind-basis}
\{\psi_{\wh{w}} \otimes u  \; | \; w \in \Sy{m_1 + \dots + m_k} /
\Sy{P}, u \in U \}.  \end{equation}
Induction is left adjoint to restriction (a property known as Frobenius reciprocity),
\begin{equation}
\HOM_{R(\nu^{(1)} + \dots + \nu^{(k)})}(\ind_{\und{\nu}}^{\nu^{(1)} + \dots + \nu^{(k)}} M, N) \cong \HOM_{R(\und{\nu})}(M,\res_{\und{\nu}}^{\nu^{(1)} + \dots + \nu^{(k)}}N).
\end{equation}
The right adjoint to restriction is the co-induction functor $\coind \maps R(\und{\nu}) \Mod \rightarrow R(\nu^{(1)} + \dots + \nu^{(k)}) \Mod$,
\begin{equation}
\coind_{\und{\nu}}^{\nu^{(1)} + \dots + \nu^{(k)}} - := \HOM_{R(\und{\nu})}(R(\nu^{(1)} + \dots + \nu^{(k)}),-),
\end{equation} 
so there is an isomorphism of $\mathbb{C}$-vectors spaces
\begin{equation}
\HOM_{R(\nu^{(1)} + \dots + \nu^{(k)})}(N, \coind_{\und{\nu}}^{\nu^{(1)} + \dots + \nu^{(k)}}M) \cong \HOM_{R(\und{\nu})}(\res_{\und{\nu}}^{\nu^{(1)} + \dots + \nu^{(k)}}N,M).
\end{equation}
For KLR algebras there exists a useful connection between induced and co-induced modules.

\begin{proposition} \cite{LV11} \label{ind-and-coind}
Let $M$ be a finite dimensional $R(\mu)$-module and $N$ a finite dimensional $R(\nu)$-module. Then up to grading shift
\begin{equation}
\ind^{\mu + \nu}_{\mu,\nu} M \boxtimes N \cong \coind^{\mu + \nu}_{\nu,\mu}N \boxtimes M.
\end{equation}
\end{proposition}

This proposition immediately tells us that most properties that we can prove about induced modules can be transferred to co-induced modules with appropriate modifications.

\begin{remark} \label{factoring-through-convention}
Suppose we have
\begin{equation}
\ind M \boxtimes D \overset{f}{\twoheadrightarrow} A
\end{equation}
with $A$ a simple $R(\mu+\nu)$-module, $D$ an $R(\nu)$-module, and $M$ an $R(\mu)$-module with a filtration 
\begin{equation}
0 = M_0 \subseteq M_1 \subseteq \dots \subseteq M_t = M.
\end{equation}
Consider the restriction 
\begin{equation}
\ind M_i \boxtimes D \xrightarrow{f} A
\end{equation}
and suppose this is nonzero, hence surjective. Suppose further that $f(\ind M_{i-1} \boxtimes D) = 0$; then we get a nonzero map
\begin{equation}
\ind M_i/M_{i-1} \boxtimes D \xrightarrow{f} A.
\end{equation}
By abuse of notation we will say $f$ \emph{factors through} $\ind M_i/M_{i-1} \boxtimes D$ in this case (note we have extended the usual notation of factoring through $M/M_{i-1}$ to $M_i/M_{i-1}$).
\end{remark}

Given $\und{i} \in \seq(\nu)$ and $\und{j} \in \seq(\mu)$, a shuffle of
$\und{i}$ and $\und{j}$ is an element $\und{k}$ of $\seq(\nu + \mu)$
such that $\und{k}$ has $\und{i}$ as a subsequence and $\und{j}$ as the
complementary subsequence.
We denote by $\und{i} \shuffle
\und{j}$ the formal sum of all shuffles of $\und{i}$ and $\und{j}$.
The multi-set of all shuffles of $\und{i}$ and
$\und{j}$ are in bijection with the minimal length left coset
representatives of $\Sy{|\nu| + |\mu|}/\Sy{|\nu|} \times \Sy{|\mu|}$.
Using the definition of degree from KLR algebras, we can associate to
any shuffle a degree which we denote as
$\deg(\und{i},\und{j},\und{k})$.
Then the quantum shuffle of $\und{i}$ and $\und{j}$ is
\begin{equation}
\und{i} \Shuffle \und{j} =
\sum_{\sigma \in \Sy{|\nu| + |\mu|}/\Sy{|\nu|} \times \Sy{|\mu|}}
q^{\deg(\und{i},\und{j},\sigma(\und{ij}))} \sigma(\und{ij}),
\end{equation}
so that $\und{i} \shuffle \und{j} = (\und{i} \Shuffle \und{j})|_{q=1}$. Note that we will usually shuffle characters, hence we write $[\und{i}] \shuffle [\und{j}]$. For an $R(\mu)$-module $M$ and $R(\nu)$-module $N$ it was shown in \cite{KL09} that
\begin{equation}
\Char(\ind_{\mu,\nu}^{\mu + \nu} M \boxtimes N) = \Char(M) \Shuffle \Char(N).
\end{equation}
This identity is referred to as the {\bf Shuffle Lemma}.


\subsection{Simple modules of \protect$R(n \alpha_{i}) $}
\label{R(alpha)-modules}

For $\nu = n\alpha_i$, induction allows for a particularly easy description of all simple $R(n\alpha_i)$-modules. Let $\Lii{i}$ be the 1-dimensional $R(\alpha_i)$-module where $x_11_{i}$ acts as zero. Then up to overall grading shift the unique simple $R(n\alpha_i)$-module is 
\begin{equation}
\Lii{i^n} = \ind_{\alpha_i, \alpha_i,  \dots, \alpha_i}^{n\alpha_i} \Lii{i} \boxtimes \dots \boxtimes \Lii{i}.
\end{equation}
Some authors prefer to shift the degree so that the character is
\begin{equation}
\Char(\Lii{i^n}) = [n]_i! [i,i,\dots, i].
\end{equation}


\subsection{Crystal operators}

In the previous section we defined induction and restriction for KLR
algebras. Following the work of Grojnowski \cite{Gro99} where crystal
operators were developed as functors on the category of modules
over affine Hecke algebras of type $A$ (or \cite{Klesh} for
$\MB{F}_{\ell\!+\!1} \Sy{m}$),
the KLR analogues of crystal operators were
introduced in \cite{KL09}, and further developed in \cite{LV11},
\cite{KK12}. For each $i \in I$, if $M \in R(\nu)\Mod$
and  $\nu -\alpha_i \in Q^+$,
 define the functor $\Delta_i:
R(\nu)\Mod \rightarrow R(\nu - \alpha_i) \otimes R(\alpha_i) \Mod$ as
the restriction
\begin{equation}
\Delta_i M := \res^{\nu}_{\nu-\alpha_i,\alpha_i}M
\end{equation}
Note that this is equivalent to multiplying $M$ by $1_{\nu-\alpha_i}
\otimes 1_{\alpha_i}$. It is also sometimes useful to think of this
functor as killing all weight spaces corresponding to elements of
$\seq(\nu)$ that do not end in $i$. 
If $\nu -\alpha_i \not\in Q^+$ then $\Delta_i M = \zero$.
 We similarly define
\begin{equation}
\Delta_{i^n} M := \res_{\nu-n\alpha_i,n\alpha_i}^\nu M.
\end{equation}
Next define the functor $\e{i}: R(\nu) \Mod \rightarrow R(\nu-\alpha_i) \Mod$ as the restriction,
\begin{equation}
\e{i}M := \res^{R(\nu-\alpha_i)\otimes R(\alpha_i)}_{R(\nu-\alpha_i)} \circ \Delta_i M,
\end{equation}
hence it is an exact functor.
When $M$ is simple we can further refine this functor by setting
\begin{equation}
\etil{i}M := \soc \e{i}M.
\end{equation}
We measure how many times we can apply $\etil{i}$ to $M$ by
\begin{equation}
\ep{i}(M) := \max\{ n \geq 0 \; | \; (\etil{i})^n M \neq \zero \;\}.
\end{equation}
Let $\ftil{i}:R(\nu) \Mod \rightarrow R(\nu+\alpha_i) \Mod$ be defined by
\begin{equation}
\ftil{i}M := \cosoc \ind M \boxtimes L(i)
\end{equation}
(still assuming $M$ is simple). Some of the most important facts about $\e{i}, \etil, \ftil{i}$ 
stated in \cite{KL09} are given in the following proposition

\begin{proposition} \label{crystal-op-facts} \mbox{}
Let $i \in I$, $\nu \in Q^+$, $n \in \Z_{>0}$.
\begin{enumerate}
\item
Let $M \in R(\nu)\Mod$. Then
\begin{equation*}
\Char(\Delta_{i^n}M) = \sum_{\und{j} \in \seq(\nu - n\alpha_i)} \text{gdim}(1_{\und{j}i^n}M) \cdot \und{j}i^n,
\end{equation*}

\item \label{ftil-and-ep} Let $N \in R(\nu) \Mod$ be irreducible and $M = \ind_{\nu,n\alpha_i}^{\nu + n\alpha_i}N \boxtimes \Lii{i^n}$. Let $\ep{} = \ep{i}(N).$ Then
\begin{enumerate}
\item \label{pull-off-i's} $\Delta_{i^{\ep{} +n}}M \cong (\etil{i})^\ep{}N \boxtimes \Lii{i^{\ep{}+n}}$.
\item \label{cosoc-is-simple} $\cosoc M$ is irreducible, and $\cosoc M \cong (\ftil{i})^n N$, $\Delta_{i^{\ep{}+n}}(\ftil{i})^n N \cong (\etil{i})^{\ep{}}N \boxtimes \Lii{i^{\ep{}+n}}$, and $\ep{i}((\ftil{i})^n N) = \ep{} + n$.
\item $(\ftil{i})^n N$ occurs with multiplicity one as a composition factor of $M$.
\item \label{all-comp-have-less-i} All other composition factors $K$ of $M$ have $\ep{i}(K) < \ep{}+n$.
\end{enumerate}

\item Let
$\und{\mu} = ({\mu_1} \alpha_i, \dots, {\mu_r} \alpha_i)$
with $\sum^r_{k=1} \mu_k = n$.
\begin{enumerate}

\omitt{\item The module $\Lii{i^n}$ over the algebra $R(n\alpha_i)$ is the only graded irreducible module, up to isomorphism.
} 

\item All composition factors of $\res_{\und{\mu}}^{n\alpha_i}\Lii{i^n}$ are isomorphic to $\Lii{i^{\mu_1}} \boxtimes \dots \boxtimes \Lii{i^{\mu_r}}$, and $\soc(\res_{\und{\mu}}^{n \alpha_i} \Lii{i^n})$ is irreducible.
\item $\etil{i} \Lii{i^n} \cong \Lii{i^{n-1}}.$
\end{enumerate}

\item \label{etil-is-only-special-compfactor} Let $M \in R(\nu) \Mod$ be irreducible with $\ep{i}(M) > 0$. Then $\etil{i}M = \soc(\e{i}M)$ is irreducible and $\ep{i}(\etil{i}M) = \ep{i}(M) - 1$. Furthermore if $K$ is a composition factor of $\e{i}M$ and $K \not\cong \etil{i}M$,
then $\ep{i}(K) < \ep{i}(M) -1$.

\item \label{e-etil-relationship}
For irreducible $M \in R(\nu) \Mod$ let $m = \ep{i}(M)$. Then
$\e{i}^m M$ is isomorphic to $(\etil{i})^m M^{\oplus [m]_i!}$.
In particular, if $m =1$ then $\e{i}M = \etil{i}M$.

\item \label{e-and-f-undo-eachother} For irreducible modules $N \in R(\nu) \Mod$ and $M \in R(\nu+\alpha_i) \Mod$ we have $\ftil{i}N \cong M$ if and only if $N \cong \etil{i}M$.

\item Let $M, N \in R(\nu) \Mod$ be irreducible. Then $\ftil{i}M \cong \ftil{i}N$ if and only if $M \cong N$. Assuming $\ep{i}(M), \ep{i}(N) > 0$, $\etil{i}M \cong \etil{i}N$ if and only if $M \cong N$.

\end{enumerate}
\end{proposition}

On the level of characters, $\e{i}$ roughly removes an $i$ from the rightmost entry of a module's character. We can construct analogous functors for removal of $i$ from the left side of a module's character, as well as an analogue to $\ftil{i}$. These are denoted by $\ech{i}$, $\etilch{i}$, $\ftilch{i}$ and we will use them extensively in this paper. We use the involution $\sigma$ introduced below to define them. Let $w_0$ be the longest element of $\Sy{|\nu|}$. Then $\sigma: R(\nu) \rightarrow R(\nu)$ is defined as follows:
\begin{align*}
1_{\und{i}} &\mapsto 1_{w_0(\und{i})}
\\
x_r &\mapsto x_{|\nu|+1-r}
\\
\psi_r 1_{\und{i}} &\mapsto
(-1)^{\delta_{i_r,i_{r+1}}}\psi_{|\nu|-r}1_{w_0(\und{i})}.
\end{align*}
For an $R(\nu)$-module $M$, let $\sigma^*M$ be the $R(\nu)$-module $M$ but with the action of $R(\nu)$ twisted by $\sigma$,
\begin{equation}
r \cdot u = \sigma(r)u.
\end{equation}

Now let $\ech{i}: R(\nu) \Mod \rightarrow R(\nu-\alpha_i) \Mod$ be the restriction functor defined as
\begin{align}
&\ech{i} := \sigma^*\e{i} \sigma = \res^{R(\alpha_i)\otimes R(\nu-\alpha_i)}_{R(\nu - \alpha_i)} \circ \res_{\alpha_i,\nu-\alpha_i}^\nu,
\intertext{
and similarly,}
&\etilch{i}M := \sigma^*(\etil{i}(\sigma^*M)) = \soc \ech{i}M,
\\
&\ftilch{i}M := \sigma^*(\ftil{i}(\sigma^*M)) = \cosoc \ind^{\nu+\alpha_i}_{\alpha_i,\nu} \Lii{i} \boxtimes M,
\\
&\epch{i}(M) := \ep{i}(\sigma^*M) = \max\{m \geq 0 \; | \; (\etilch{i})^m M \neq \zero \}.
\end{align}

It is important to note that by the exactness of restriction, $\e{i},
\ech{i}$ are exact functors, while $\etil{i}$ and $\etilch{i}$ are only
left exact, and $\ftil{i}$ and $\ftilch{i}$ are only right exact.

If $M \in R(\nu)\Mod$, we define $\wt(M)  = -\nu$ and $\wt_i(M) = -\langle h_i, \nu \rangle$. The functors $\etil{i}$, $\ftil{i}$, as well as $\ep{i}$ and $\wt$ will be part of a crystal structure with simple $R$-modules as nodes \cite{LV11}.

\begin{remark} \label{rem-char-epsilon}
There is a convenient character theoretic interpretation of $\ep{i}$ and $\epch{i}$. Let $M$ be a simple $R(\nu)$-module with $|\nu| = m$. Then
\begin{enumerate}[a.)]
\item $\ep{i}(M) = c$ implies that there exists 
\begin{equation}
\und{i} = (i_1, \dots, i_{m-c},\underbrace{i,i, \dots ,i}_{c})
\end{equation}
such that $1_{\und{i}}A \neq \zero$.
In other words $[\und{i}]$ is in the support of $M$; however no $[\und{j}]$ such that 
\begin{equation}
\und{j} = (i_{1}, \dots, i_{m-c-1}, \underbrace{i ,i ,\dots, i}_{c+1}).
\end{equation}
is in $\supp{M}$.

\item $\epch{i}(M) = c$ implies that there exists $[\und{i}]$ in the support of $M$ of the form
\begin{equation}
\und{i} = (\underbrace{i,i,\dots ,i}_{c},i_{c+1}, \dots ,i_m)
\end{equation}
but no $[\und{j}]$ of the form
\begin{equation}
\und{j} = (\underbrace{i,i,\dots ,i}_{c+1},i_{c+2}, \dots ,i_m).
\end{equation}

\end{enumerate}
\end{remark}

\begin{proposition} \label{cosoc-soc-isomorphic} \cite{V02}, \cite{K14}
Suppose that $M$ is a simple $R(\nu)$-module with $\ep{i}(M) > 0$. Then up to grading shift $\cosoc(\e{i}M) \cong \etil{i}M$.
\end{proposition}


The following proposition describes how
the restriction functors $\e{i}$ interact with induction.

\begin{proposition} \label{e-and-short-exact-sequence} \cite[Theorem 5.3]{Vaz1}
Let $M$ be a simple $R(\mu)$-module and $N$ be a simple $R(\nu)$-module. Then for $i \in I$ there is a short exact sequences
\begin{equation}
\zero \rightarrow \ind M \boxtimes \e{i}N \rightarrow \e{i}\Big(\ind M \boxtimes N\Big) \rightarrow \ind \e{i}M \boxtimes N \rightarrow \zero
\end{equation}
and
\begin{equation}
\zero \rightarrow \ind \ech{i}M \boxtimes N \rightarrow \ech{i}\Big(\ind M \boxtimes N \Big) \rightarrow \ind M \boxtimes \ech{i}N \rightarrow \zero.
\end{equation}
\end{proposition}

Note that it follows from this result that if $\e{i}M = \zero$ then
\begin{equation}
\e{i}\Big(\ind M \boxtimes N\Big) \cong \ind M \boxtimes \e{i}N
\end{equation}
and similarly if $\e{i}N = \zero$ then
\begin{equation}
\e{i}\Big(\ind M \boxtimes N\Big) \cong \ind \e{i}M \boxtimes N.
\end{equation}
An analogous remark applies to $\ech{i}$.

\begin{proposition} \label{when-can-apply-etil-directly}
Suppose that $M$ is a simple $R(\mu)$-module, $N$ is a simple $R(\nu)$-module, $A$ is simple $R(\mu + \nu)$-module and there is a surjection
\begin{equation}
\label{eq-indMN}
\ind M \boxtimes N \twoheadrightarrow A.
\end{equation} 
\begin{enumerate}
\item \label{ei-case-for-direct-application-of-etili} If $\ep{i}(A) \neq 0$ then: 
\begin{enumerate}
\item \label{ei-case-for-direct-application-of-etili-a} If $\ep{i}(M) = 0$ then there is a surjection
\begin{equation}
\ind M \boxtimes \etil{i}N \twoheadrightarrow \etil{i}A.
\end{equation}
\item \label{ei-case-for-direct-application-of-etili-b} If $\ep{i}(N) = 0$, then there is a surjection
\begin{equation}
\ind \etil{i}M \boxtimes N \twoheadrightarrow \etil{i}A.
\end{equation}
\end{enumerate}
\item \label{eich-case-for-direct-application-of-etili} If $\epch{i}(A) \neq 0$ then:
\begin{enumerate}
\item If $\epch{i}(M) = 0$ there is a surjection
\begin{equation}
\ind M \boxtimes \etilch{i}N \twoheadrightarrow \etilch{i}A.
\end{equation}
\item If $\epch{i}(N) = 0$ there is a surjection
\begin{equation}
\ind \etilch{i}M \boxtimes N \twoheadrightarrow \etilch{i}A.
\end{equation}
\end{enumerate}
\end{enumerate}
\end{proposition}

\begin{proof}
We prove Part \ref{ei-case-for-direct-application-of-etili-a}, the other parts follow from analogous arguments.
By Frobenius reciprocity it follows that $\ep{i}(N) \leq \ep{i}(A)$. By
the exactness of the functor $\e{i}$ and Proposition
\ref{e-and-short-exact-sequence}, $\ep{i}(N) \geq \ep{i}(A)$ so that
$\ep{i}(N) = \ep{i}(A)$.
Also by Proposition \ref{e-and-short-exact-sequence},
the map of \eqref{eq-indMN} induces a surjection
\begin{equation}
\ind M \boxtimes \e{i}N \twoheadrightarrow \e{i}A.
\end{equation}
Composing this with the surjection from Proposition \ref{cosoc-soc-isomorphic} gives
\begin{equation}
\ind M \boxtimes \e{i}N \twoheadrightarrow \etil{i}A.
\end{equation}
From Remark \ref{factoring-through-convention} there exists a composition factor $K$ of $\e{i}N$ such that there is a surjection,
\begin{equation}
\ind M \boxtimes K \twoheadrightarrow \etil{i}A.
\end{equation}
We must have
\begin{equation}
\ep{i}(K) = \ep{i}(\etil{i}A) = \ep{i}(A) - 1 = \ep{i}(N) -1.
\end{equation}
Proposition \ref{crystal-op-facts}.\ref{etil-is-only-special-compfactor}, then forces $K \cong \etil{i}N$.
\end{proof}


\subsubsection{The Serre relations}


Because the functors $\e{i}$, $i \in I$, are exact, they descend to well-defined linear operators on the Grothendieck group of $R$, $G_0(R)$. By abuse of notation, we will also call these operators $\e{i}$. We define divided powers
\begin{align*}
e^{(r)}_i : G_0(R) &\rightarrow G_0(R)
\\
[M] &\mapsto \frac{1}{ [r]_i !}[e^r_i M].
\end{align*}
It is shown in \cite{KL09,KL11} that these operators satisfy the quantum Serre relations, and that these relations are in fact minimal. We have
\begin{equation}
\sum^{-a_{ij}+1}_{r = 0} (-1)^r \e{i}^{(-a_{ij}+1-r)}\e{j}\e{i}^{(r)}[M] = \zero
\end{equation}
for all $i \neq j \in I$,  and $M \in R\Mod$.
(Recall $a_{ij} = \langle h_i, \alpha_j \rangle$.)
 The minimality of these relations imply that, for $0 < c < -a_{ij}+1$, 
\begin{equation} \label{Serre-relations}
\sum^c_{r=0} (-1)^r\e{i}^{(c-r)}e_je_i^{(r)}
\end{equation}
is never the zero operator on $G_0(R)$ by the quantum Gabber-Kac Theorem \cite{Lus10} and the work of \cite{KL09,KL11}, which essentially computes the kernel of the map from the free algebra on generators $\e{i}$ to $G_0(R)$.


\subsection{Simple {\protect$R(c \alpha_{i} + \alpha_{j})$-}modules}
\label{R(c-alpha)-modules}


Assume that $i \neq j$ and set $a = a_{ij} = \langle h_i,\alpha_j\rangle$. We introduce the notation
\begin{equation}
\Lcal{i^{c-n}ji^n} 
\end{equation}
for the irreducible $R(c\alpha_i + \alpha_j)$-modules when $c \leq -a$. 

\begin{theorem} \cite{LV11}
Let $c \leq -a$ and let $\nu = c\alpha_i + \alpha_j$. For each $n$ with $0 \leq n \leq c$, there exists a unique irreducible $R(\nu)$-module denoted $\Lcal{i^{c-n}ji^n}$ with
\begin{equation}
\ep{i}(\Lcal{i^{c-n}ji^{n}}) = n.
\end{equation}
Furthermore
\begin{equation}
\epch{i}(\Lcal{i^{c-n}ji^n}) = c-n
\end{equation}
and
\begin{equation}
\Char(\Lcal{i^{c-n}ji^n}) = [c-n]_i ! [n]_i ! i^{c-n}ji^n.
\end{equation}
\end{theorem}

Observe the support of $\Lcal{i^{c-n}ji^n}$ is exactly $i^{c-n}ji^n$.

\begin{proposition} \cite{LV11} \label{when-jump-zero-for-Lcal-modules}
For $m > 0$,
\begin{equation}
\ind \Lcal{i^{-a-n}ji^n} \boxtimes \Lii{i^m} \cong \ind \Lii{i^m} \boxtimes \Lcal{i^{-a-n}ji^n}
\end{equation}
is irreducible.
\end{proposition}


\subsection{Jump} \label{jump-section}

When we apply $\ftil{i}$ to irreducible $R(\nu)$-module $M$ for $i \in I$, then Proposition \ref{crystal-op-facts}.\ref{ftil-and-ep} tells us that $\ftil{i}M$ is an irreducible $R(\nu+\alpha_i)$-module with
\begin{equation}
\ep{i}(\ftil{i}M) = \ep{i}(M) + 1.
\end{equation}
We could also ask whether $\epch{i}(\ftil{i}M)$ and $\epch{i}(M)$ differ. Questions like this motivate the introduction of the function $\jump{i}$, which is based on a concept for Hecke algebras in \cite{Gro99}, and was introduced for KLR algebras and studied extensively in \cite{LV11}.

\begin{definition}
Let $M$ be a simple $R(\nu)$-module, and let $i \in I$. Then
\begin{equation}
\jump{i}(M) := \max \{ J \geq 0 \; | \; \epch{i}(M) = \epch{i}(\ftil{i}^J M)\}.
\end{equation}
\end{definition}

\begin{lemma} \cite{LV11} \label{jump-lemma}
Let $M$ be a simple $R(\nu)$-module. The following are equivalent:
\begin{enumerate}
\item \label{jump-lemma-1} $\jump{i}(M) = 0$
\item \label{switch-ftil-ftilch} $\ftil{i} M \cong \ftilch{i}M$
\item \label{jump-0-simple} $\ind M \boxtimes \Lii{i^m}$ is irreducible for all $m \geq 1$
\item \label{jump-lemma-4} $\ind M \boxtimes \Lii{i^m} \cong \ind \Lii{i^m} \boxtimes M$ for all $m \geq 1$
\item \label{jump-formula} $\wti{i}(M) + \ep{i}(M) + \epch{i}(M) = 0$
(recall that $\wti{i}(M) = -\langle h_i, \nu \rangle$),
\item \label{jump-lemma-6} $\ep{i}(\ftilch{i}M) = \ep{i}(M) + 1$
\item $\epch{i}(\ftil{i}M) = \epch{i}(M) + 1$
\end{enumerate}
\end{lemma}

\begin{proof}
See \cite{LV11}.
\end{proof}

\begin{example}
For $c \leq -a_{ij}$ we can calculate $\jump{i}(\Lcal{i^{c-n}ji^n}) = -a_{ij} - c$. Proposition \ref{when-jump-zero-for-Lcal-modules} follows from the equivalence of \ref{jump-lemma-1}, \ref{jump-0-simple}, and \ref{jump-lemma-4} in Lemma \ref{jump-lemma}.
\end{example}

It is shown in \cite{LV11} that 
\begin{equation} \label{ftil-and-jump-ob}
\jump{i}(\ftil{i}M) = \max\{0,\jump{i}(M) -1\} = \jump{i}(\ftilch{i}M).
\end{equation}
This means that if we know $\jump{i}(M)$, we can easily calculate $\jump{i}(\ftil{i}^k M)$ for any $k \geq 0$. 
It is also shown in \cite{LV11} that 
\begin{equation} \label{eqn-jump}
\jump{i}(M) =  
\wti{i}(M) + \ep{i}(M) + \epch{i}(M).
\end{equation}
Using information from $\jump{i}$ we can also determine when the crystal operators commute with their $\sigma$-symmetric versions.

\begin{example}
In type $A^{(1)}_{\ell}$ for $\ell >1$, observe that $\jump{1}(\Lii{0}) = 1$ and 
\begin{equation}
\ftilch{1}\ftil{1} \Lii{0} \cong \ind \Lii{1} \boxtimes \Lcal{01}
\end{equation}
whose character has support $\{[1,0,1],[0,1,1],[0,1,1]\}$. However 
\begin{equation}
\ftil{1}\ftilch{1} \Lii{0} \cong \ind \Lcal{10} \boxtimes \Lii{1}
\end{equation}
whose character has support $\{[1,0,1],[1,1,0],[1,1,0]\}$.
\end{example}

In $A^{(1)}_{1}$, note that $\jump{1}(\Lcal{01}) = 1$ and we can similarly calculate $\ftilch{1} \ftil{1} \Lcal{01} \not\cong \ftil{1}\ftilch{1} \Lcal{01}$ (in fact the former is 8-dimensional while the latter is 4-dimensional).

We shall see below that this phenomenon is special to $\jump{i}(M) = 1$.

\begin{lemma} \label{commuting-functors}
Let $M$ be a simple $R(\nu)$-module. 
\begin{enumerate} 
\item \label{diff-i-j} \cite{LV11}
If $i \neq j$, then
\begin{enumerate} 
\item \label{2-fs} $\ftil{i}\ftilch{j}M \cong \ftilch{j}\ftil{i}M.$ 
\item \label{e-and-f} If $\etilch{j}M \neq \zero$ then $\ftil{i}\etilch{j}M \cong \etilch{j}\ftil{i}M$.
\item \label{f-and-e} If $\etil{j}M \neq \zero$ then $\ftilch{i}\etil{j}M \cong \etil{j}\ftilch{i}M$.
\item \label{2-es} If further $\etil{i}M \neq \zero$ then, $\etil{i}\etilch{j}(M) \cong \etilch{j}\etil{i}(M)$.
\end{enumerate}

\omitt{\item $\jump{i}(\etilch{i}M) = 0$ then $\ftil{i}\etilch{i}(M) \cong \etilch{i}\ftil{i}(M).$
} 

\item \label{commuting-f-same-i} 
\begin{enumerate}
\item \label{2-fs-same} $\jump{i}(M) \neq 1$ if and only if $\ftilch{i} \ftil{i} M \cong \ftil{i}\ftilch{i} M$. 
\item \label{etilchi-ftil-same} If $\etilch{i}M \neq \zero$, then $\jump{i}(\etilch{i}M) \neq 1$ if and only if $\etilch{i}\ftil{i}M \cong \ftil{i}\etilch{i}M$.
\item \label{etil-ftilch-same} If $\etil{i}M \neq \zero$, then $\jump{i}(\etil{i}M) \neq 1$ if and only if $\etil{i}\ftilch{i}M \cong \ftilch{i}\etil{i}M$.
\end{enumerate}
\end{enumerate}
\end{lemma}

\begin{proof} \mbox{}
\begin{enumerate}
\item Consider the short exact sequence,
\begin{equation}
\zero \rightarrow K \rightarrow \ind M \boxtimes \Lii{i} \rightarrow \ftil{i}M \rightarrow \zero
\end{equation}
and recall $\ftil{i}M$ is the unique composition factor of $\ind M \boxtimes \Lii{i}$ such that $\ep{i}(\ftil{i}M) = \ep{i}(M) + 1$, and
that for all composition factors $N$ of $K$, $\ep{i}(N) \leq \ep{i}(M)$. By the exactness of induction there is a second short exact sequence
\begin{equation} \label{induced-exact-sequence}
\zero \rightarrow \ind \Lii{j} \boxtimes K \rightarrow \ind \Lii{j} \boxtimes M \boxtimes \Lii{i} \rightarrow \ind \Lii{j} \boxtimes \ftil{i}M \rightarrow \zero,
\end{equation}
and since $i \neq j$ the Shuffle Lemma tells us that for all
composition factors $N'$ of $\ind \Lii{j} \boxtimes K$, $\ep{i}(N')
\leq \ep{i}(M)$. 
By the Shuffle Lemma and  Frobenius reciprocity
 \begin{equation}
\ep{i}(\ftilch{j}\ftil{i}M) =
\ep{i}(\ftil{i}\ftilch{j}M) =
\ep{i}(M) + 1.
\end{equation}
Hence there can be no nonzero map
\begin{equation}
\ind \Lii{j} \boxtimes K \rightarrow
\ftil{i}
\ftilch{j}
M,
\end{equation} 
so that
the submodule $\ind \Lii{j} \boxtimes K$ is contained in the kernel of
$\beta$, as pictured in \eqref{commuting-ftils-diagram}.

\begin{equation}
 \label{commuting-ftils-diagram}
\begin{tikzpicture}
\node at (0,0) {$\ind \Lii{j} \boxtimes M \boxtimes \Lii{i}$};
\node at (5,.7) {$\ind \Lii{j} \boxtimes \ftil{i}M$};
\node at (5,-.7) {$\ind \ftilch{j}M \boxtimes \Lii{i}$};
\node at (9,.7) {$\ftilch{j}\ftil{i}M$};
\node at (9,-.7) {$\ftil{i}\ftilch{j}M$};
\node at (0,-1.5) {};
\node at (5,2.5) {$\alpha$};
\node at (5,-2.5) {$\beta$};
\draw[thick,->>] (2,.2) -- (3.2,.5);
\draw[thick,->>] (2,0) -- (3.2,-.5);
\draw[thick,->>] (7,.73) -- (8,.73);
\draw[thick,->>] (7,-.73) -- (8,-.73);
\draw[thick,->>] (1,.3) to [out = 50, in = 150] (8.8,1.2);
\draw[thick,->>] (1,-.3) to [out = -50, in = -150] (8.8,-1.2);
\end{tikzpicture} 
\end{equation}
Hence $\beta$ induces 
 a nonzero map (necessarily surjective)
\begin{equation}
\ind \Lii{j} \boxtimes \ftil{i}M \twoheadrightarrow \ftil{i}\ftilch{j}M.
\end{equation}
Because $\ind \Lii{j} \boxtimes \ftil{i} M$ has unique simple quotient $\ftilch{j}\ftil{i}M$, then $\ftilch{j}\ftil{i}M \cong \ftil{i}\ftilch{j}M$. This proves \ref{2-fs}.

The three isomorphisms in \ref{e-and-f}, \ref{f-and-e}, and \ref{2-es} all follow from
\ref{2-fs}. For example, if $\etilch{j}M$ is nonzero, then
\begin{equation}
\ftil{i}M \cong \ftil{i}\ftilch{j}\etilch{j}M \cong \ftilch{j}\ftil{i}\etilch{j}M.
\end{equation}
Applying $\etilch{j}$ to both sides we get \ref{e-and-f}. \ref{f-and-e} and \ref{2-es} follow similarly.

\item
We prove \ref{2-fs-same}. Let $c = \epch{i}(M), m = \ep{i}(M)$.
\begin{itemize}

  \item Suppose $\jump{i}(M) = 0$. Then also $\jump{i} (\ftil{i} M) = \jump{i} (\ftilch{i} M) = 0$ by \eqref{ftil-and-jump-ob}. Thus by Lemma \ref{jump-lemma}
\begin{equation}
\ftilch{i}\ftil{i}M \cong \ftil{i}\ftil{i}M \cong \ftil{i}\ftilch{i}M.
\end{equation}
  
  \item Suppose $\jump{i}(M) = 1$. By Lemma \ref{jump-lemma} and Proposition \ref{crystal-op-facts}, $\ep{i}(\ftil{i}M) = m+1$ but $\epch{i}(\ftil{i}M) = c$. While $\ep{i}(\ftilch{i}M) = m$ but $\epch{i}(\ftilch{i}M) = c+1$. Further by \eqref{ftil-and-jump-ob} $\jump{i}(\ftil{i}M) = \jump{i} (\ftilch{i} M) = 0$. Hence $\ep{i}(\ftilch{i}\ftil{i}M) = m+2$, $\epch{i}(\ftilch{i} \ftil{i} M) = c+1$ whereas $\ep{i}(\ftil{i}\ftilch{i}M) = m+1$, $\epch{i}(\ftil{i} \ftilch{i}M) = c+2$. Thus the two modules cannot be isomorphic.
  
\item Suppose $\jump{i}(M) \geq 2$. Then $\jump{i}(\ftil{i}M) = \jump{i}(\ftilch{i}M) \geq 1$. We calculate 
\begin{align}
\ep{i}(\ftil{i}\ftilch{i} M) = m+1 = \ep{i}(\ftilch{i} \ftil{i} M) \\
\epch{i}(\ftil{i} \ftilch{i} M) = c+1 = \epch{i}(\ftilch{i} \ftil{i} M).
\end{align}
We will show there is no nonzero map 
\begin{equation}
\ind \Lii{i} \boxtimes K \rightarrow \ftil{i} \ftilch{i} M 
\end{equation}
for any proper submodule $K \subseteq \ind M \boxtimes \Lii{i}$. Given we have a surjection 
\begin{equation}
\ind \Lii{i} \boxtimes M \boxtimes \Lii{i} \twoheadrightarrow \ftil{i} \ftilch{i} M
\end{equation}
this means we must have a nonzero map
\begin{equation}
\ind \Lii{i} \boxtimes \ftil{i}M \rightarrow \ftil{i}\ftilch{i}M,
\end{equation}
which will prove the lemma as 
\begin{equation}
\ftilch{i} \ftil{i} M = \cosoc \ind \Lii{i} \boxtimes \ftil{i}M.
\end{equation}
First note there is no nonzero map
\begin{equation}
\ind \Lii{i} \boxtimes \ftilch{i}M \rightarrow \ftil{i} \ftilch{i}M
\end{equation}
as $\cosoc( \ind \Lii{i} \boxtimes \ftilch{i}M) = (\ftilch{i})^2 M$ and $\epch{i}((\ftilch{i})^2 M) = c+ 2 \neq c+1 = \epch{i}(\ftil{i} \ftilch{i}M)$. Let $D$ be any other composition factor of $\ind M \boxtimes \Lii{i}$ apart from $\ftil{i}M$ or $\ftilch{i}M$ (recall the latter occur with multiplicity one as composition factors). Then by Proposition \ref{crystal-op-facts}, $\ep{i}(D) \leq m$,  $\epch{i}(D) \leq c$. If there were a nonzero map $\ind \Lii{i} \boxtimes D \rightarrow \ftil{i} \ftilch{i}M$, it would imply $\ftilch{i} D \cong \ftil{i}\ftilch{i} M$ and so $\epch{i}(\ftilch{i}D) = c+1$ meaning $\epch{i}(D) = c$. Also $m+1 = \ep{i}(\ftilch{i}D) \leq \ep{i}(D) + 1$ by the Shuffle Lemma, forcing $\ep{i}(D) = m$. By Lemma \ref{jump-lemma} this forces $0 = \jump{i}(D)$ and $\ftil{i}D \cong \ftilch{i}D \cong \ftil{i}\ftilch{i} M$ from above, forcing $D \cong \ftilch{i}M$, which we already ruled out. Hence there must be a nonzero map
\begin{equation}
\ind \Lii{i} \boxtimes \ftil{i} M \rightarrow \ftil{i} \ftilch{i} M.
\end{equation}
Now that we have established $\ftilch{i} \ftil{i} M \cong
\ftil{i}\ftilch{i} M$ if and only if $\jump{i}(M) \neq 1$,
statements \ref{etilchi-ftil-same} and \ref{etil-ftilch-same}
follow directly from Proposition
\ref{crystal-op-facts}.\ref{e-and-f-undo-eachother}.

\end{itemize}

\end{enumerate}
\end{proof}

\begin{remark} \label{ep-and-ftil-for-i-neq-j}
Because $\etil{i}$ and $\ftilch{j}$ commute for $i \neq j$, then $\ep{i}(\ftilch{j}M) = \ep{i}(M)$. An equivalent statement holds for $\etilch{i}$, $\ftil{j}$, and $\epch{i}$.
When $\jump{i}(M) \neq 0$, $\ep{i}(\ftilch{i}M) = \ep{i}(M)$.
\end{remark}

\begin{proposition} \label{when-ei-ej-commute}
Let $M$ be a simple $R(\nu)$-module. Let $i,j \in I$ and $a_{ij} = 0$. Then
\begin{enumerate}
\item \label{when-ei-ej-commute1} $\ftil{i}\ftil{j}M \cong \ftil{j}\ftil{i}M$.
\item \label{when-ei-ej-commute2} $\etil{i}\etil{j}M \cong \etil{j}\etil{i}M$.
\end{enumerate}
\end{proposition}

\begin{proof}
The proof of part \ref{when-ei-ej-commute1} is similar to the proof of Proposition \ref{commuting-functors}.\ref{2-fs} once we note that by formula \eqref{eqn-jump} and Lemma \ref{jump-lemma},
\begin{equation}
\ind \Lii{i} \boxtimes \Lii{j} \cong \ind \Lii{j} \boxtimes \Lii{i}.
\end{equation}
Part \ref{when-ei-ej-commute2} follows directly from part \ref{when-ei-ej-commute1} by Proposition \ref{crystal-op-facts}.\ref{e-and-f-undo-eachother}. Observe that Part \ref{when-ei-ej-commute2} is true even in the case $\etil{i}\etil{j}M = \zero$.
\end{proof}


\section{$\repL{}$ and the functor $\pr{}$}
\label{sec-pr}

All of the lemmas and propositions in this section
about the interaction of the functors $\pr{}$ and $\ind$ are
adapted from \cite{V15}.

For $\Lambda = \sum_{i \in I} \lambda_i \Lambda_i \in P^+$ define $\cycloI{\Lambda}_\nu$ to be the two-sided ideal of $R(\nu)$ generated by the elements $x_1^{\lambda_{i_{1}}}1_{\und{i}}$ for all $\und{i} \in \seq(\nu)$. When $\nu$ is clear from the context we write, $\cycloI{\Lambda}_\nu = \cycloI{\Lambda}$. The {\emph{cyclotomic KLR algebra of weight $\Lambda$}} is then defined as
\begin{equation}
R^\Lambda = \bigoplus_{\nu \in Q^+} R^{\Lambda}(\nu) \quad \text{where} \quad R^\Lambda(\nu) := R(\nu)/\cycloI{\Lambda}_\nu.
\end{equation}

 The algebra $R^\Lambda(\nu)$ is finite dimensional, \cite{BK09a, LV11}.
The category of finite dimensional
 $R^{\Lambda}(\nu)$-modules is denoted $R^{\Lambda}(\nu) \Mod$ and the
category of finite dimensional $R^\Lambda$-modules is denoted
$R^{\Lambda} \Mod$.
The category of finite dimensional $R$-modules on which
$\cycloI{\Lambda}$ vanishes is denoted
$$\rep{\Lambda}.$$
While we can identify $R^{\Lambda} \Mod$ with $\rep{\Lambda},$
we choose to work with $\rep{\Lambda}.$
We
construct a right-exact functor,
$\pr{}: R(\nu) \Mod \rightarrow R(\nu)\Mod$,  via
\begin{equation} \pr{}M := M/\cycloI{\Lambda}M
\end{equation}
and extend it to 
$\pr{}: R\Mod \rightarrow R\Mod$.
It is customary in the literature to interpret $\pr{}$ as being a
functor from $R\Mod$ to $R^{\Lambda}\Mod$, but in this
paper it will be more convenient to consider it as a functor
$R\Mod \to R\Mod$ or $R\Mod \to \rep{\Lambda}$.
The reader may keep in mind that the image of $\pr{}$ consists of
$R(\nu)$-modules which descend to $R^{\Lambda}(\nu)$-modules.
Observe
that in the opposite direction there is an exact functor $\infl{}:
R^{\Lambda}(\nu) \Mod \rightarrow R(\nu) \Mod$, where $R(\nu)$ acts on
$R^{\Lambda}(\nu)$-module $M$ through the projection map $R(\nu)
\twoheadrightarrow R^{\Lambda}(\nu)$.

\begin{remark} \label{sujrection-onto-pr-remark}
If $M$ is a $R(\nu)$-module and $A$ is a simple
module in $\rep{\Lambda}$ for $\Lambda \in P^+$, then since $\pr{}A \cong A$, the right exactness of $\pr{}$ implies that any surjection $M \twoheadrightarrow A$ gives a surjection $\pr{}M \twoheadrightarrow A$. Similarly, since there always exists a surjection $M \twoheadrightarrow \pr{}M$, given a surjection $\pr{}M \twoheadrightarrow A$ we immediately get a surjection $M \twoheadrightarrow A$. In such situations there is an equivalence between the two surjections $M \twoheadrightarrow A$ and $\pr{}M \twoheadrightarrow A$ which we will henceforth use freely.
\end{remark}

If $M$ is simple then either $\pr{}M = \zero$ or $\pr{}M = M$. There is a useful criterion for determining the action of $\pr{}$ on simple $R(\nu)$-modules given by the following proposition.

\begin{proposition} \label{cyclotomic-char} \cite{LV11}
Let $\Lambda = \sum_{i \in I} \lambda_i \Lambda_i \in P^+$,
$\nu \in Q^+$, and let $M$ be a simple $R(\nu)$-module. Then
$\cycloI{\Lambda} M = \zero$ if and only if $\pr{}M \cong M$
if and only if
$\pr{}M \neq \zero$
if and only if
$$\epch{i}(M) \leq \lambda_i$$
for all $i \in I$. When these
conditions hold $M \in \rep{\Lambda}$, and we may identify $M$ with
$\pr{}M$ 
(or as an $R^{\Lambda}(\nu)$-module).
 \end{proposition}

In this paper we will primarily consider $\Lambda =
\Lambda_i$ in which case $\cycloI{\Lambda_i}_\nu$ is generated by
$x_11_{ii_2 \dots i_m}$ and $1_{ji_2 \dots i_m}, j \neq i$
 ranging over $\und{i}\in \seq(\nu)$.

Notice that Proposition \ref{cyclotomic-char} immediately tells us that if $p$ is a cyclotomic path of tail weight $(\Lambda_j,\Lambda_i)$ then the modules $\Tii{p}{k}$ belong to $\rep{\Lambda_i}$ for any $k \geq 0$. This is part of the motivation for the definition of cyclotomic path in Section \ref{Crystal-Review}.

For $\Lambda = \sum_{i \in I}\lambda_i \Lambda_i \in P^+$ and $M$ an
irreducible $R(\nu)$-module 
set
\begin{equation} \label{phcyclo-formula}
\phcyc{}{j}(M) = \lambda_j + \ep{j}(M) + \wti{j}(M).
\end{equation}
Notice that when $\Lambda = \Lambda_i$ this gives
\begin{equation} \label{special-phcyclo-formula}
\phcyc{i}{j}(M) = \delta_{ij} + \ep{j}(M) + \wti{j}(M).
\end{equation}

\begin{remark} \label{when-phcyc-jump-the-same}
By formula \eqref{special-phcyclo-formula} if $M$ is a simple
module in $\rep{\Lambda_i}$ it follows that
\begin{equation}
\phcyc{i}{j}(M) = \begin{cases}
\delta_{ij} & \text{if  }M = \UnitModule, \\
\jump{j}(M) & \text{otherwise.} \\
\end{cases}
\end{equation}
\end{remark}

\begin{proposition}\cite{LV11} \label{interp-of-phcyc}
Let $M$ be a simple $R(\nu)$-module with $\pr{}M \neq \zero$. Then
\begin{equation}
\phcyc{}{j}(M) = \max\{k \in \mathbb{Z} \; | \; \pr{}\ftil{j}^k M \neq \zero \}.
\end{equation}
\end{proposition}

\subsubsection{Module-theoretic model of  $B(\Lambda)$}
\label{sec-LV}

Let $M$ be a simple $R(\nu)$-module.  Set
\begin{equation}
\wt(M) =
- \nu \quad \text{ and } \quad \wti{i}(M) = -\langle h_i, \nu
\rangle. 
\end{equation}
Let $\Irr R$ be the set of isomorphism
classes of simple $R$-modules and $\Irr R^\Lambda$ be the set of
isomorphism classes of simple
modules in $\rep{\Lambda}$.
In \cite{LV11}
it was shown that the tuple  $(\Irr R, \ep{i}, \p_i, \etil{i},
\ftil{i},  \wt)$ defines a crystal isomorphic to $B(\infty)$ and
$(\Irr R^\Lambda, \ep{i}, \phcyc{}{i}, \etil{i}, \ftil{i},  \wt)$
defines a crystal isomorphic to the highest weight crystal
$B(\Lambda)$.

\medskip

From property \eqref{ftil-and-jump-ob}
 of $\jump{i}$ it is clear that if we apply $\ftil{i}$ sufficiently many times to any simple module $M \in \rep{\Lambda}$,
then eventually we will reach an $n$ for which  
\begin{equation}
\epch{i}(\ftil{i}^n M) > \lambda_i
\end{equation}
and so
$\pr{}\ftil{i}^n M = \zero$. Proposition \ref{interp-of-phcyc} says 
that $\phcyc{}{i}$ measures this  for simple 
modules in $\rep{\Lambda}$.
In fact it
is true that $\pr{}M \neq \zero$ if and only if $\phcyc{}{i}(M) \geq
0$ for all $i \in I$. Above we saw that the function
$\phcyc{}{i}$ is part of a crystal datum.
With this is mind, we mimic the conventions usually used in the theory of crystals and define
\begin{equation}
\phcycall{}(M) := \sum_{i \in I} \phcyc{}{i}(M) \Lambda_i,
\end{equation}
and
\begin{equation}
\ep{}(M) := \sum_{i \in I} \ep{i}(M)\Lambda_i, \quad\quad \epch{}(M) := \sum_{i \in I} \epch{i}(M)\Lambda_i.
\end{equation}

\begin{lemma} \label{simpe-quotients-and-pr}
Let $M$ be an $R(\nu)$-module such that $\pr{}M \neq \zero$. Let $Q$ be a simple quotient of $\pr{}M$. Then $\pr{}Q = Q$.
\end{lemma}

\begin{proof}
This follows from the right exactness of $\pr{}$.
\end{proof}

\subsubsection{Interaction of $\pr{}$ and $\ind$}

The following is a list of
useful facts about the way that the functor $\pr{}$ interacts with
induction.

\begin{proposition}
\label{pr-facts}
Let $\mu, \nu \in Q^+$,
$\Lambda \in P^+$, $i,j \in I$,
and let $M$ be simple $R(\mu)$-module and $N$ a simple $R(\nu)$-module.
\begin{enumerate} [(a)]
\item \label{part-1-pr-facts} If $\pr{}M = \zero$ then $\pr{} \ind M \boxtimes N = \zero$.
\item \label{vanishing-pr-prop}
If $\pr{}\ind M \boxtimes \Lii{i^c} =
\zero$ and $\epch{i}(N) \ge c$ then $\pr{} \ind M \boxtimes N = \zero$.
\item \label{pr-fact-5} If $c > \phcyc{}{i}(M)$ then $\pr{}\ind M
\boxtimes \Lii{i^c} = \zero$.
\item \label{pr-ind-determine-rep} If $\pr{} \ind M \boxtimes N \neq \zero$ then $N \in \rep{\phcycall{}(M)}$.
\item \label{f-and-pr} Let $\p = \p_i^\Lambda (M)$, then $\pr{}\ind M \boxtimes \Lii{i^\p} \cong \ftil{i}^\p M$.
\item \label{pr-fact-4} If $\pr{}C = M$ then $\pr{}\ind C \boxtimes N \cong \pr{}\ind M \boxtimes N$.
\item \label{two-nodes-to-add} Suppose $a_{ij} < 0$, $\ep{}(M) = \Lambda_i$, and $\pr{}\ind \etil{i}M \boxtimes \Lii{j} = \zero$. Then $\pr{}\ind M \boxtimes \Lii{j} \cong \pr{}\ftil{j}M$. In particular, if $\pr{}M \neq \zero$ then $\ep{}(\ftil{j}M) = \Lambda_j$.
\end{enumerate}
\end{proposition}

\begin{proof} We write
\begin{equation}
\Lambda = \sum_{i \in I} \lambda_i \Lambda_i.
\end{equation}
\begin{enumerate} [(a)]

\item If $\pr{}M = \zero$, then by Proposition \ref{cyclotomic-char} there is some $i \in I$ such that $\epch{i}(M) > \lambda_i$. Suppose that $\pr{}\ind M \boxtimes N \neq \zero$. Then it has some
simple quotient $Q$, and there are surjections
\begin{equation} \label{double-surjection-to-simple}
\ind M \boxtimes N \twoheadrightarrow \pr{}\ind M \boxtimes N \twoheadrightarrow Q.
\end{equation}
By Lemma \ref{simpe-quotients-and-pr} $\pr{}Q = Q$. By Frobenius reciprocity $\res_{\mu,\nu}^{\mu + \nu}Q$ has $M \boxtimes N$ as a $(R(\mu) \otimes R(\nu))$-submodule. But Remark \ref{rem-char-epsilon} then implies $\epch{i}(Q) > \lambda_i$ so that $\pr{}Q = \zero$, a contradiction.
\\
\item If $\epch{i}(N) \ge c$ then there is a surjection,
\begin{equation}
\ind \Lii{i^c} \boxtimes (\etilch{i})^c N \twoheadrightarrow N
\end{equation}
and by the exactness of induction a surjection
\begin{equation}
\ind M \boxtimes \Lii{i^c} \boxtimes (\etilch{i})^cN \twoheadrightarrow \ind M \boxtimes N.
\end{equation}
If $\pr{}\ind M \boxtimes \Lii{i^c} = \zero$, then by part
\eqref{part-1-pr-facts} above and the right exactness of $\pr{}$,
$\pr{} \ind M \boxtimes N = \zero$.  \\
\item
This follows from Proposition \ref{interp-of-phcyc} and the fact that
the induced module has unique simple quotient $\ftil{i}^c M$;
or see \cite{LV11}.

\item We will show the contrapositive.
 Suppose $N \notin \rep{\phcycall{}(M)}$. Then by Proposition \ref{cyclotomic-char}, there is some $i \in I$ such that $\epch{i}(N) = c > \phcyc{}{i}(M)$. By part \eqref{pr-fact-5}, $\pr{}\ind M
\boxtimes \Lii{i^c} = \zero$. Part \eqref{vanishing-pr-prop} then implies that $\pr{}\ind M \boxtimes N = \zero$. 

\item Consider the exact sequence,
\begin{equation}
\zero \rightarrow K \rightarrow \ind M \boxtimes L(i^{\p}) \rightarrow \ftil{i}^{\p}M \rightarrow \zero.
\end{equation}
$\ftil{i}^\p M$ is the unique composition factor of $\ind M \boxtimes
\Lii{i^\p}$ such that $\ep{i}(\ftil{i}^\p M) = \p + \ep{i}(M)$, and 
$\ep{i}(D) < \p + \ep{i}(M)$ for all composition factors $D$ of $K$.
All composition factors $D$ of $K$ have the same weight as $\ftil{i}^\p
M$.
By \eqref{phcyclo-formula} and Proposition \ref{crystal-op-facts},
$\phcyc{}{i}(D) = \lambda_i + \ep{i}(D) + \wti{i}(D) 
	< \lambda_i + \ep{i}(\ftil{i}^\p M) + \wti{i}(\ftil{i}^\p M)
	= \phcyc{}{i}(\ftil{i}^\p M) = 0$.
In particular this shows $\pr{}K=\zero$ so by the right exactness
of $\pr{}$ we get \eqref{f-and-pr}.

\item Consider the diagram in \eqref{pr-diagram-1},
\begin{equation} \label{pr-diagram-1}
\begin{tikzpicture}
\node at (0,0) {$\zero$};
\node at (0,1.5) {$\ind M \boxtimes N$};
\node at (0,3) {$\ind C \boxtimes N$};
\node at (0,4.5) {$\ind \cycloI{\Lambda}C \boxtimes N$};
\node at (0,6) {${\zero}$};
\node at (-6.5,3) {${\zero}$};
\node at (-3.5,3) {$\cycloI{\Lambda}(\ind C \boxtimes N)$};
\node at (3.5,3) {$\pr{}(\ind C \boxtimes N)$};
\node at (6.5,3) {$\zero$};
\node at (-.45,3.75) {$\alpha$};
\node at (1.4,3.3) {$\beta$};
\node at (2.2,4.3) {$\beta \circ \alpha$};
\node at (1.8,1.4) {$g$};
\node at (-.3,2.3) {$\gamma$};
\draw[thick,<-] (0,.3) -- (0,1.1);
\draw[thick,<-] (0,3.3) -- (0,4.1);
\draw[thick,<-] (0,1.8) -- (0,2.6);
\draw[thick,<-] (0,4.8) -- (0,5.6);
\draw[thick,->] (-6.1,3) -- (-4.8,3);
\draw[thick,->] (-2.1,3) -- (-1.1,3);
\draw[thick,->] (1.1,3) -- (2.1,3);
\draw[thick,->] (4.8,3) -- (6.1,3);
\draw[thick,->] (1.2,4.3) -- (2.5,3.4);
\draw[thick,dotted,->] (1.1,1.5) -- (2.5,2.5);
\end{tikzpicture}
\end{equation}
where the horizontal and vertical sequences are exact. Recall that $\cycloI{\Lambda}_\mu$ in $R(\mu)$ is generated by the set $\{x_1^{\lambda_{i_{1}}}1_{\und{i}}\}_{\und{i} \in \seq(\mu)}$ where $\und{i} = i_1i_2 \dots i_m$ and $|\mu| = m$. Under the embedding
\begin{equation}
R(\mu) \hookrightarrow R(\mu) \otimes R(\nu) \hookrightarrow R(\mu + \nu),
\end{equation}
this set maps to the set 
\begin{equation}
\Big\{\sum_{\und{j} \in \seq(\nu)}x_1^{\lambda_{i_{1}}}1_{\und{ij}}\Big\}_{\und{i} \in \seq(\mu)}
\end{equation}
in $R(\mu + \nu)$. This set is contained in the ideal generated by $\{x_1^{\lambda_{i_{1}}}1_{\und{k}}\}_{\und{k} \in \seq(\mu + \nu)}$ which generates $\cycloI{\Lambda}_{\mu + \nu}$. It follows that
\begin{equation}
R(\mu + \nu)\cycloI{\Lambda}_\mu \subseteq \cycloI{\Lambda}_{\mu+\nu},
\end{equation}
and hence
\begin{equation}
\ind \cycloI{\Lambda}_\mu C \boxtimes N \subseteq \cycloI{\Lambda}_{\mu + \nu}(\ind C \boxtimes N).
\end{equation}
This tells us that the composition $\beta \circ \alpha$ from the
diagram in \eqref{pr-diagram-1} is zero, so there exists a
surjective homomorphism $g: \ind M \boxtimes N \rightarrow \pr{}\ind C
\boxtimes N$. Applying $\pr{}$ to the diagram in 
\eqref{pr-diagram-1}, and sending maps $\gamma$, $\beta$, and $g$ to
$\tilde{\gamma}$, $\tilde{\beta}$, and $\tilde{g}$ respectively,
right exactness yields
$\tilde{\gamma}$, $\tilde{\beta}$, and $\tilde{g}$ are surjections
as shown in diagram \eqref{pr-diagram-2}. It follows from considerations
of dimension and that $\pr{}C=M$
 that $\tilde{g}$ must be an isomorphism.
 \begin{equation}
\begin{tikzpicture} \label{pr-diagram-2}
\node at (0,1.5) {$\pr{}(\ind M \boxtimes N)$};
\node at (0,3) {$\pr{}(\ind C \boxtimes N)$};
\node at (3.5,3) {$\pr{}(\ind C \boxtimes N)$};
\node at (6.5,3) {$\zero$};
\node at (1.7,3.5) {$\tilde{\beta}$};
\node at (2.2,1.6) {$\tilde{g}$};
\node at (-.3,2.3) {$\tilde{\gamma}$};
\node at (0,0) {$\zero$};
\draw[thick,<-] (0,.3) -- (0,1.1);
\draw[thick,<-] (0,2) -- (0,2.6);
\draw[thick,->] (1.4,1.7) -- (2.5,2.5);
\draw[thick,->] (5,3) -- (6.1,3);
\draw[thick,->] (1.5,3) -- (2,3);
\end{tikzpicture}
\end{equation}

\item First, $\ep{}(M) = \Lambda_i$ implies $\e{i}M \cong \etil{i}M \neq \zero$. Further, if $\pr{}\etil{i}M = \zero$ then $\pr{}M = \zero$ so the conclusion follows as all modules in question are $\zero$. If $\pr{}\etil{i}M \neq \zero$ the hypotheses imply that $\ind \etil{i}M \boxtimes \Lii{j}$ is irreducible by Lemma \ref{jump-lemma}.

Suppose
\begin{equation}
\zero \rightarrow K \rightarrow \ind M \boxtimes \Lii{j} \rightarrow \pr{}\ind M \boxtimes \Lii{j} \rightarrow \zero
\end{equation}
is exact. By the exactness of $\e{i}$ then 
\begin{equation}
\zero \rightarrow \e{i}K \rightarrow \e{i}\Big( \ind M \boxtimes \Lii{j} \Big) \xrightarrow{\p} \e{i}\Big(\pr{}\ind M \boxtimes \Lii{j} \Big) \rightarrow \zero
\end{equation}
is also exact. By Proposition \ref{e-and-short-exact-sequence}, $\e{i}( \ind M \boxtimes \Lii{j} ) \cong \ind \e{i}M \boxtimes \Lii{j} \cong \ind \etil{i}M \boxtimes \Lii{j}$. But $\pr{}\ind \etil{i}M \boxtimes \Lii{j} = \zero$ by hypothesis so the surjection $\p$ then forces $\e{i}(\pr{}\ind M \boxtimes \Lii{j}) = \zero$. 
In particular $\e{i}(\pr{} \ftil{j} M) = \zero$ and so
$\ep{i}(\ftil{j}M) = 0$.

Note if $D$ is any composition factor of $\ind M \boxtimes \Lii{j}$ other than $\ftil{j}M$ then $\ep{j}(D) = 0$ by Proposition \ref{crystal-op-facts}.\ref{ftil-and-ep} and the fact that $\ep{j}(M) = 0$. Then $\ep{}(M) = \Lambda_i$ and the Shuffle Lemma forces $\e{i}D \neq \zero$. Hence $D$ cannot be a composition factor of $\pr{}\ind M \boxtimes \Lii{j}$. In particular $\pr{} \ind M \boxtimes \Lii{j} \cong \pr{} \ftil{j}M$.

If $\pr{}\ftil{j}M \neq \zero$ then $\pr{}\ftil{j}M = \ftil{j}M$ because $\ftil{j}M$ is simple and as $\ep{j}(\ftil{j}M) = 1$, $\ep{i}(\ftil{j}M) = 0$, and $\ep{}(M) = \Lambda_i$ we must have $\ep{}(\ftil{j}M) = \Lambda_j$.

\end{enumerate}
\end{proof}


\section{The family of modules \protect$\Tii{{p}}{{k}}$}
\label{Description-of-family}

In this section we describe certain modules, motivated by their
crystal-theoretic description in \cite{V07}, see Proposition
\ref{building-T} below.  
For each of the classical affine types listed in Section \ref{Cartan-datum-section} we specify a set
of elements $\{\gammaplus{p}{k}\}$ of $Q^{+}$ corresponding to paths in $B^{1,1}$. For path $p: \{0, 1, \dots, k-1\} \rightarrow I$ of length $k$ define
\begin{equation}
\gammaplus{p}{k} := \alpha_{p(0)} + \alpha_{p(1)} + \dots + \alpha_{p(k-1)}.
\end{equation}
Observe that $|\gammaplus{p}{k}| = k$. We also define $\und{\trivseq{p}{k}} \in \seq(\gammaplus{p}{k})$ to be 
\begin{equation}
\und{\trivseq{p}{k}} = (p(0), p(1), \dots, p(k-1)).
\end{equation}

Let $\adjac{j_1}{j_2}{\und{\trivseq{p}{k}}} \subseteq \{1,2, \dots, k\}$ be the collection of indices for adjacent $j_1, \; j_2 \in I$ in the sequence $\und{\trivseq{p}{k}}$,
\begin{equation}
\adjac{j_1}{j_2}{\und{\trivseq{p}{k}}} := \{\; t \; | \; p(t-1) = j_1, \; p(t) = j_2, \; 1 \leq t \leq k-1 \}.
\end{equation}
Let $|\adjac{j_1}{j_2}{\und{\trivseq{p}{k}}}|$ denote the cardinality of this set. 

\begin{example}
If $p$ is a type $A_6^{(2)}$ path such that
\begin{equation}
\und{\trivseq{p}{18}} = (0,1,2,3,2,1,0,0,1,2,3,2,1,0,0,1,2,3),
\end{equation}
then $\adjac{0}{0}{\und{\trivseq{p}{18}}} = \{7,14\}$.
\end{example}

To each path $p$ of length $k$, associate an $R(\gammaplus{p}{k})$-module denoted by $\Tii{p}{k}$. In all types $\Tii{p}{0} = \UnitModule$, the unique simple $R(0)$-module which we will refer to as the \emph{unit module}. 

\begin{itemize}

\item ($A^{(1)}_{\ell}$): \; $\Tii{p}{k}$ is a 1-dimensional simple module with character 
\begin{equation}
\Char(\Tii{p}{k}) = [\und{\trivseq{p}{k}}] = [p(0), p(1), \dots, p(k-1)].
\end{equation}
If $v$ spans $\Tii{p}{k}$, then the generators of $R(\gammaplus{p}{k})$ act by
\begin{equation}
1_{\und{i}} v = \delta_{\und{\raisebox{1pt}{\scriptsize$\bf{i}$}},\und{\raisebox{3pt}{\scriptsize$\bf{\trivseq{p}{k}}$}}} v, \quad\quad x_r 1_{\und{i}} v = 0, \quad\quad \psi_r 1_{\und{i}} v = 0. 
\end{equation}

\item ($C^{(1)}_{\ell}$): \; In this type $\Tii{p}{k}$ is also 1-dimensional, again with character $\Char(\Tii{p}{k}) = [\und{\trivseq{p}{k}}]$. The action of the generators of $R(\gammaplus{p}{k})$ is the same as in type $A^{(1)}_\ell$.

\item ($A^{(2)}_{2\ell}$): \; Let $d = |\adjac{0}{0}{\und{\trivseq{p}{k}}}|$. Then each $\Tii{p}{k}$ has ungraded dimension $2^d$ and character
\begin{equation}
\Char(\Tii{p}{k}) = q_0^{d}[2]_0^d[\und{\trivseq{p}{k}}].
\end{equation}
If 
\begin{equation}
\adjac{0}{0}{\und{\trivseq{p}{k}}} = \{t_1, \dots, t_d\}
\end{equation}
then $\Tii{p}{k}$ has homogeneous basis 
\begin{equation}
\Big\{v_{\beta} \; \Big| \; \beta = (\beta_{1},\dots,\beta_{d}) \in \{0,1\}^d \Big\}. 
\end{equation}
The basis element $v_{\beta}$ belongs to the degree $2(\beta_1 + \dots + \beta_{d})$ component of $\Tii{p}{k}$. The generators of $R(\gammaplus{p}{k})$ act on $v_{\beta}$ by
\begin{equation}
1_{\und{i}} v_\beta = \delta_{\und{\raisebox{1pt}{\scriptsize$\bf{i}$}},\und{\raisebox{3pt}{\scriptsize$\bf{\trivseq{p}{k}}$}}} v_\beta,
\end{equation} 
\begin{equation}
x_rv_\beta = x_r1_{\und{\trivseq{p}{k}}} v_\beta = \begin{cases}
v_{\beta + \delta_r} & \text{if} \; r \in \adjac{0}{0}{\und{\trivseq{p}{k}}}, \; \beta_r = 0 \\
-v_{\beta + \delta_{r}} & \text{if} \; r-1 \in \adjac{0}{0}{\und{\trivseq{p}{k}}}, \; \beta_r = 0 \\
0 & \text{otherwise}
\end{cases}
\end{equation}
(here $\delta_{r}$ is the element of $\{0,1\}^d$ with only $r$th coordinate non-zero). Finally,
\begin{equation}
\psi_r v_\beta = \psi_r 1_{\und{\trivseq{p}{k}}} v_\beta = \begin{cases}
v_{\beta - \delta_r} & \text{if} \; r \in \adjac{0}{0}{\und{\trivseq{p}{k}}}, \; \beta_r = 1 \\
0 & \text{otherwise}.
\end{cases}
\end{equation}

\item ($A^{(2)\dagger}_{2\ell}$): \; Let $d = |\adjac{\ell}{\ell}{\und{\trivseq{p}{k}}}|$. Then $\Tii{p}{k}$ has ungraded dimension $2^d$ and character
\begin{equation}
\Char(\Tii{p}{k}) = q_\ell^{d}[2]_\ell^d[\und{\trivseq{p}{k}}].
\end{equation}
There is a basis for $\Tii{p}{k}$ whose description is similar to the description of the basis $\{v_\beta\}$ for type $A^{(2)}_{2\ell}$, the only difference being that we replace $\adjac{0}{0}{\und{\trivseq{p}{k}}}$ with $\adjac{\ell}{\ell}{\und{\trivseq{p}{k}}}$ everywhere. 

\item ($D^{(2)}_{\ell+1}$): \; Let $d = |\adjac{0}{0}{\und{\trivseq{p}{k}}}| + |\adjac{\ell}{\ell}{\und{\trivseq{p}{k}}}|$. Here each $\Tii{p}{k}$ has ungraded dimension $2^d$ and character
\begin{equation} \label{type-D-char-formula}
\Char(\Tii{p}{k}) = q_0^{d}[2]_0^d[\und{\trivseq{p}{k}}]
\end{equation}
(note that in this type $[2]_0 = [2]_\ell$ since $(\alpha_0,\alpha_0) = (\alpha_\ell,\alpha_\ell) = 2$, so we may write formula \eqref{type-D-char-formula} in terms of $[2]_0$ instead of both $[2]_0$ and $[2]_\ell$). It has a homogeneous basis analogous to that for types $A^{(2)}_{2\ell}$ and $A^{(2)\dagger}_{2\ell}$ which we again call $\{v_\beta\}$. This time however we replace $\adjac{0}{0}{\und{\trivseq{p}{k}}}$ from the description for type $A^{(2)}_{2\ell}$ with $\adjac{0}{0}{\und{\trivseq{p}{k}}} \cup \adjac{\ell}{\ell}{\und{\trivseq{p}{k}}}$, and $v_\beta$ belongs to the degree $2(\beta_1 + \dots + \beta_{d})$ component of $\Tii{p}{k}$.
\item ($A^{(2)}_{2\ell-1}$)  
Let $d = |\adjac{0}{1}{\und{\trivseq{p}{k}}} \cup \adjac{1}{0}{\und{\trivseq{p}{k}}}|$. Here $\Tii{p}{k}$ has graded dimension $2^d$. If
\begin{equation}
\adjac{0}{1}{\und{\trivseq{p}{k}}} \cup \adjac{1}{0}{\und{\trivseq{p}{k}}} = \{t_1, \dots, t_d \}
\end{equation}
then $\Tii{p}{k}$ is concentrated in a single degree and has character
\begin{equation}
\Char(\Tii{p}{k}) = \sum_{(\beta_1, \dots, \beta_d) \in \{0,1\}^d} \Big[s_{t_1}^{\beta_1}s_{t_2}^{\beta_2} \dots s_{t_d}^{\beta_d}\und{\trivseq{p}{k}}\Big].
\end{equation}
Here $s_i$ is the adjacent transposition from $\Sy{k}$ that interchanges $i$ and $i+1$. $\Sy{k}$ acts on $\seq(\gammaplus{p}{k})$ as defined in \eqref{sym-grp-action-seq}. Let $v$ be a nonzero vector from $1_{\und{\trivseq{p}{k}}}\Tii{p}{k}$. Then a homogeneous weight basis for $\Tii{p}{k}$ is given by 
\begin{equation}
\{\psi_{t_1}^{\beta_1}\psi_{t_2}^{\beta_2} \dots \psi_{t_d}^{\beta_d}
v \mid (\beta_1, \dots, \beta_d) \in \{0,1\}^d\}.  \end{equation}
If $\beta = (\beta_1,\dots, \beta_d) \in \{0,1\}^d$ then write $v_\beta = \psi_{t_1}^{\beta_1}\psi_{t_2}^{\beta_2} \dots \psi_{t_d}^{\beta_d} v$. Also write 
\begin{equation}
\und{i_\beta} = s_{t_1}^{\beta_1}s_{t_2}^{\beta_2} \dots s_{t_d}^{\beta_d}\und{\trivseq{p}{k}}.
\end{equation}
Notice that the $\und{i_\beta}$ are distinct and
$\dim(1_{\und{i_\beta}}T) = 1$. Then the generators of $R(\gammaplus{p}{k})$ act by,
\begin{equation}
1_{\und{i}} v_\beta = \delta_{\und{\raisebox{1pt}{\scriptsize$\bf{i}$}},\und{\smash{\raisebox{1pt}{\scriptsize$\bf{i_\beta}$}}}} v_\beta. 
\end{equation}
\begin{equation}
x_r v_\beta = x_{r}1_{\und{i_\beta}}v_\beta = 0, \quad 1 \leq r \leq k,
\end{equation}
\begin{equation}
\psi_r v_\beta = \psi_r 1_{\und{i_\beta}}v_\beta = 
\begin{cases}
v_{\beta + \delta_r \text{mod}\; 2 } & \text{if $r \in \adjac{0}{1}{\und{\trivseq{p}{k}}}\cup \adjac{1}{0}{\und{\trivseq{p}{k}}}$},  \\
0, & \text{otherwise}.
\end{cases}
\end{equation}

\item ($B^{(1)}_{\ell}$) 
For this type write $d_1 = |\adjac{0}{1}{\und{\trivseq{p}{k}}}\cup \adjac{1}{0}{\und{\trivseq{p}{k}}}|$ and $d_2 = |\adjac{\ell}{\ell}{\und{\trivseq{p}{k}}}|$. Here $\Tii{p}{k}$ has ungraded dimension $2^{d_1 + d_2}$. If
\begin{equation}
\adjac{0}{1}{\und{\trivseq{p}{k}}}\cup \adjac{1}{0}{\und{\trivseq{p}{k}}} = \{t_1, \dots, t_{d_1} \}
\end{equation}
\begin{equation}
\adjac{\ell}{\ell}{\und{\trivseq{p}{k}}} = \{t_1', \dots, t_{d_2}' \}.
\end{equation}
Then $\Tii{p}{k}$ has character
\begin{equation}
\Char(\Tii{p}{k}) = \sum_{(\beta_1, \dots, \beta_{d_1}) \in
\{0,1\}^{d_1}}
q_{\ell}^{d_2}[2]_\ell^{d_2}\Big[s_{t_1}^{\beta_1}s_{t_2}^{\beta_2}
\dots s_{t_{d_1}}^{\beta_{d_1}}\und{\trivseq{p}{k}}\Big].
\end{equation}
We can describe a homogeneous weight basis $\{v_\beta\}$ for $\Tii{p}{k}$ and action by the generators of $R(\gammaplus{p}{k})$ in this type by mixing our descriptions for type $A^{(2)\dagger}_{2\ell}$ with type $A^{(2)}_{2\ell-1}$.

\item ($D^{(1)}_{\ell}$): \; Let $d = |\adjac{0}{1}{\und{\trivseq{p}{k}}}\cup \adjac{1}{0}{\und{\trivseq{p}{k}}} \cup \adjac{\ell-1}{\ell}{\und{\trivseq{p}{k}}}\cup \adjac{\ell}{\ell-1}{\und{\trivseq{p}{k}}}|$. Here $\Tii{p}{k}$ has graded dimension $2^d$. If
\begin{equation}
\adjac{0}{1}{\und{\trivseq{p}{k}}}\cup \adjac{1}{0}{\und{\trivseq{p}{k}}} \cup \adjac{\ell-1}{\ell}{\und{\trivseq{p}{k}}}\cup \adjac{\ell}{\ell-1}{\und{\trivseq{p}{k}}} = \{t_1, \dots, t_d \}
\end{equation}
then $\Tii{p}{k}$ is concentrated in a single degree which we can take to be zero and has character
\begin{equation}
\Char(\Tii{p}{k}) = \sum_{(\beta_1, \dots, \beta_d) \in \{0,1\}^d} \Big[s_{t_1}^{\beta_1}s_{t_2}^{\beta_2} \dots s_{t_d}^{\beta_d}\und{\trivseq{p}{k}}\Big].
\end{equation}
We can describe a homogeneous weight basis $\{v_\beta\}$ for $\Tii{p}{k}$ and action on it by the generators for $R(\gammaplus{p}{k})$ in this type in the same manner as we did for type $A^{(2)}_{2\ell-1}$ by replacing every instance of $\adjac{0}{1}{\und{\trivseq{p}{k}}}\cup \adjac{1}{0}{\und{\trivseq{p}{k}}}$ with $\adjac{0}{1}{\und{\trivseq{p}{k}}}\cup \adjac{1}{0}{\und{\trivseq{p}{k}}} \cup \adjac{\ell-1}{\ell}{\und{\trivseq{p}{k}}}\cup \adjac{\ell}{\ell-1}{\und{\trivseq{p}{k}}}$.

\end{itemize}

\begin{remark}
\label{rem-q}
In the above description of characters, we could have removed
global factors such as $q_i^d$, as we only care about the simples
up to overall grading shift. Further, in those cases above
we in fact have $q_i = q$.
\end{remark}

Showing that all these descriptions actually define modules involves checking that the actions agree with the KLR relations. For the $\Tii{p}{k}$ above which are 1-dimensional, it is obvious that $\Tii{p}{k}$ is simple. Showing that $\Tii{p}{k}$ is simple in each of the other types is an easy exercise.

By \cite{LV11}, a simple $R$-module corresponds to a node in a
highest weight crystal and the nodes corresponding to the
$\Tii{p}{k}$ were precisely those studied in \cite{V07}.
See Proposition \ref{building-T} below.

In type $A^{(1)}_\ell$ we could have also chosen to work with $B^{\ell,1}$. As noted in Remark \ref{B-ell-1-remark} in type $A^{(1)}_\ell$, $B^{\ell,1}$ is similar to $B^{1,1}$ but with all arrows reversed. If $p$ is a path in $B^{\ell,1}$ then as above define
\begin{equation}
\gammaminus{p}{k} := \alpha_{p(0)} + \dots + \alpha_{p(k-1)} \quad\quad \und{\signseq{p}{k}} = (p(0), \dots, p(k-1)),
\end{equation} 
and let $\Sii{p}{k}$ denote the simple 1-dimensional $R(\gammaminus{p}{k})$-module with character
\begin{equation}
\Char(\Sii{p}{k}) = [\und{\signseq{p}{k}}].
\end{equation}
In the Hecke algebra or symmetric group case, $\Tii{p}{k}$ are the trivial modules and $\Sii{p}{k}$ are the sign modules (see \cite{V15}).

\begin{remark}
In fact, given any 1-dimensional $R(\nu)$-module $M$ in type $A^{(1)}_\ell$, there is some path $p$ in either $B^{1,1}$ or $B^{\ell,1}$ such that $\nu = \gammaplus{p}{k}$ and $M \cong \Tii{p}{k}$.
\end{remark}

Returning to our family of modules $\Tii{p}{k}$ from Section \ref{basic-definitions} we note that the path $p$ that parametrizes $\Tii{p}{k}$ also determines the sequence of $\ftil{i}$ that we have to apply to $\UnitModule$ to construct to $\Tii{p}{k}$.

\begin{proposition}
\label{building-T}
For all types $X_\ell$ and for a fixed path $p$ of length $k$ in $B^{1,1}$
\begin{equation} \label{builiding-Tpk}
\ftil{p(k-1)} \ftil{p(k-2)} \dots \ftil{p(0)} \UnitModule \cong \Tii{p}{k}.
\end{equation}
\end{proposition}

\begin{proof}
Recall that
\begin{equation}
[p(0), p(1), \dots, p(k-1)]
\end{equation}
is in the support of $\Tii{p}{k}$
For $k \geq 1$, let $i = p(k-1)$. Inspection of $[\und{i}] \in
\supp{\Tii{p}{k}}$ implies that $\ep{} := \ep{i}(\Tii{p}{k}) 
\in \{1,2\}$ and in particular is nonzero.
Also by inspection $\Char
(\etil{i}^{\ep{i}} \Tii{p}{k}) = \Char (\e{i}^{(\ep{})}\Tii{p}{k})$
agrees with $\Char(\Tii{p}{k-\ep{}})$ (up to overall grading shift).
Hence by Remark \ref{character-determines-simple},
$\etil{i}^{\ep{}}\Tii{p}{k} \cong \Tii{p}{k-\ep{}}$. Proposition
\ref{crystal-op-facts}.\ref{e-and-f-undo-eachother} then implies
$\Tii{p}{k} \cong \ftil{i}^{\ep{}}\Tii{p}{k-\ep{}}$. The identity
\eqref{builiding-Tpk} then follows by induction.

\end{proof}

Note that this proposition also implies that if $j \in \pextplus{p}$ then $\ftil{j} \Tii{p}{k} \cong \Tii{p \extheadby{j}}{k+1}$.


We will frequently need to compute $\jump{j}(\Tii{p}{k})$ and $\phcyc{p(0)}{j}(\Tii{p}{k})$ for a path $p$ in $B^{1,1}$.
In our main theorems of Section \ref{sec-main}, we will only
consider paths $p$ such that $\Tii{p}{k}\in \rep{\Lambda_{p(0)}}$. 
In this case, Remarks \ref{unique-ext-and-phi} and
\ref{when-phcyc-jump-the-same}  apply, so we can in fact just
consider $\jump{}$. However, for completeness, we give
data for other $p$.
Observe the right/left symmetry of $\jump{j}$ as in Table
\ref{jump-table} that is broken for $\phcyc{p(0)}{j}(\Tii{p}{k})$ in general.

We introduce special notation,
\begin{equation}
\phcyctriv{j}{p}{k} := \phcyc{p(0)}{j}(\Tii{p}{k}), \quad\quad \phcyctriv{}{p}{k} = \phcyc{p(0)}{}(\Tii{p}{k}).
\end{equation}
When $k = 0$, the unit module $\Tii{p}{0} = \UnitModule$ has
$$\jump{j}(\UnitModule) = 0 \quad \text{and} \quad
	\phcyc{i}{j}(\UnitModule) = \delta_{ij}.$$
When $k=1$,
\begin{gather*}
\jump{j}(\Lii{j}) = 0= \phcyc{j}{j}(\Lii{j})
 	\quad \text{and} \\
 \jump{j}(\Lii{i}) = -\langle h_j,\alpha_i \rangle
= - a_{ji} = \phcyc{i}{j}(\Lii{i}).
\end{gather*}

For $k \geq 2$, we return to the
classification of elements of $I$ summarized in Table
\ref{node-classification}. Tables
\ref{jump-table}
and
\ref{phcyc-table}
give the values of
$\jump{j}(\Tii{p}{k})$
and
$\phcyctriv{j}{p}{k}$
for $k \ge 2$,
according to the class that $j$ belongs to.
In these tables we use the notation that for $j \in I$
\begin{equation}
\delta_{j,\pextminus{p}} = \begin{cases}
1 & \text{if  $j \in \pextminus{p}$} \\
0 & \text{otherwise} \\
\end{cases}
\quad \quad \quad \delta_{j,\pextplus{p}} = \begin{cases}
1 & \text{if  $j \in \pextplus{p}$} \\
0 & \text{otherwise} \\
\end{cases}
\end{equation}

\begin{table}
\begin{center}
{\tabulinesep=1.2mm
 \begin{tabu}{| l | c | c |}
 \hline
 Class & $\jump{j}(\Tii{p}{k})$  \\
 \hline \hline
 class $\ClassA$ & $\delta_{j,\pextminus{p}} \;+\; \delta_{j,\pextplus{p}}$ \\
 class $\ClassB$ &
$\delta_{j,\pextminus{p}}(2  - \delta_{j,p(0)}) \; + \; \delta_{j,\pextplus{p}}(2 - \delta_{j,p(k-1)})$ \\
 class $\ClassD$ & $ \delta_{j,\pextminus{p}} \; + \; \delta_{j,\pextplus{p}} $     \\                                   
 \hline
\end{tabu}}
\caption{
Values of $\jump{j}(\Tii{p}{k})$, $k>1$, by class of $j \in I$.}
\label{jump-table}
\end{center}
\end{table}


\begin{table}
\begin{center}
{\tabulinesep=1.2mm
 \begin{tabu}{| l | c | }
 \hline
 Class & $\phcyctriv{j}{p}{k}$  \\
 \hline \hline
 class $\ClassA$ &  $\delta_{j,\pextminus{p}} \;+\; \delta_{j,\pextplus{p}}$ \\
 class $\ClassB$  & $2\delta_{j,\pextminus{p}} \;-\; \delta_{j,p(0)} \;+\; \delta_{j,\pextplus{p}}(2 - \delta_{j,p(k-1)})$ \\
 class $\ClassD$   & $ \delta_{j,\pextminus{p}} 
+ (\delta_{j,\pextminus{p}}-1) \delta_{j',p(0)}
\;+\; \delta_{j,\pextplus{p}} $     \\ 
 \hline
\end{tabu}}
\caption{Values of $\phcyctriv{j}{p}{k}$, $k>1$, by class of $j \in I$.
When $j$ is Class $\ClassD$, $(j,j')$ is a Class $\ClassD$ pair.
}
\label{phcyc-table}
\end{center}
\end{table}

For $k > 1$, we define $\phcyctrivminus{j}{p}{k}$ and
$ \phcyctrivplus{j}{p}{k}$
as in Table \ref{phcycplus-table}, so that
\begin{equation}
\phcyctriv{j}{p}{k} = \phcyctrivminus{j}{p}{k} + \phcyctrivplus{j}{p}{k}. 
\end{equation}
Note
$\phcyctrivminus{j}{p}{k}$ is ``controlled" by elements at the left end of the sequence $\und{\trivseq{p}{k}}$ and $\phcyctrivplus{j}{p}{k}$ is controlled by elements at the right end of $\und{\trivseq{p}{k}}$.
 We also define, 
\begin{equation}
\Nexttermminus{p}{k} := \sum_{i \in I}  \phcyctrivminus{i}{p}{k}\Lambda_i
\end{equation}
\begin{equation}
\Nexttermplus{p}{k} := \sum_{i \in I} \phcyctrivplus{i}{p}{k}\Lambda_i
\end{equation}
so that
\begin{equation}
\Nextterm{p}{k} = \Nexttermminus{p}{k} + \Nexttermplus{p}{k}.
\end{equation}

\begin{table}
\begin{center}
{\tabulinesep=1.2mm
 \begin{tabu}{| l | c | c | c |}
 \hline
 Class & $\phcyctrivminus{j}{p}{k}$ & $\phcyctrivplus{j}{p}{k}$  \\
 \hline \hline
 class $\ClassA$ &  $\delta_{j,\pextminus{p}}$ & $ \delta_{j,\pextplus{p}}$ \\
 class $\ClassB$ & $2\delta_{j,\pextminus{p}} \; - \; \delta_{j,p(0)}$ & $ \delta_{j,\pextplus{p}}(2 - \delta_{j,p(k-1)})$ \\
 class $\ClassD$ & $\delta_{j,\pextminus{p}} +
(\delta_{j,\pextminus{p}}-1) \delta_{j',p(0)} $
 & $\delta_{j,\pextplus{p}}$     \\ 
 \hline
\end{tabu}}
\caption{Values of $\phcyctrivminus{j}{p}{k}$ and $\phcyctrivplus{j}{p}{k}$, $k>1$, by class of $j \in I$.
When $j$ is Class $\ClassD$, $(j,j')$ is a Class $\ClassD$ pair.}
\label{phcycplus-table}
\end{center}
\end{table}

\begin{remark} \label{path-interp-phcyc}
There is a more visual interpretation of $\phcyctriv{j}{p}{k}$, $\phcyctrivminus{j}{p}{k}$, and $\phcyctrivplus{j}{p}{k}$ 
in terms of the path $p$ in $B^{1,1}$ 
when $k>1$. 
\begin{itemize}
\item  The value $\phcyctriv{j}{p}{k}$ is the maximum number of $j$-arrows that extend path $p$ from both its tail and head.  
\item 
When $\epch{ }(\Tii{p}{k}) = \Lambda_{p(0)}$, 
the value $\phcyctrivminus{j}{p}{k}$ is the maximum number of $j$-arrows that extend the tail of path $p$,
\begin{equation}
\phcyctrivminus{j}{p}{k} = \max\{k \in \MB{Z}_{\geq 0} \; | \; (\pi(j))^k \star p \neq \zero \; \}.
\end{equation}
\item The value $\phcyctrivplus{j}{p}{k}$ is the maximum number of $j$-arrows that extend the head of path $p$,
\begin{equation}
\phcyctrivplus{j}{p}{k} = \max\{k \in \MB{Z}_{\geq 0} \; | \; p \star (\pi(j))^k \neq \zero \; \}.
\end{equation}
\end{itemize}
Observe that under the above hypotheses
$\phcyctrivminus{j}{p}{k}$ and $\phcyctrivplus{j}{p}{k}$ take values in $\{0,1,2\}$. This agrees with the fact that for all types there is no $j \in I$ such that more than 2 $j$-arrows appear consecutively in $B^{1,1}$.
(More generally, we can also have $\phcyctrivplus{j}{p}{k} = -1$
when $\Tii{p}{k} \notin \rep{\Lambda_{p(0)}}$).
\end{remark}

When $\Tii{p}{k} \notin \rep{\Lambda_{p(0)}}$ 
then $\Tii{p}{k} \in \rep{\Lambda_{p(0)} + \Lambda_{p(1)}}$.
Since $\langle h_{p(1)}, \Lambda_{p(0)} \rangle = \epch{p(1)}(\Tii{p}{k}) -1$ in this case, by 
\eqref{special-phcyclo-formula}
and
\eqref{eqn-jump}
$\p_{p(1)}^{\Lambda_{p(0)}}(\Tii{p}{k}) = \jump{p(1)}(\Tii{p}{k})-1$.
In particular, when $j$ is of Class $\ClassB$ and
 $\epch{j}(\Tii{p}{k}) = 2$, then $\phcyctriv{j}{p}{k} = 
\jump{j}(\Tii{p}{k}) -1 =
\jump{j}(\Tii{p}{k}) - 
\delta_{j,p(0)}(1- \delta_{j,\pextminus{p}})$.
 When $j$ is of Class $\ClassD$, $(j,j')$ is a Class $\ClassD$ pair,
and $p(0)=j', p(1)=j$ then $\phcyctrivminus{j}{p}{k} = -1.$
 Note, for these modules
Remark \ref{when-phcyc-jump-the-same} does not apply. 
Further, such $p$ will not ever
arise in our main theorems. 
See Remark \ref{unique-ext-and-phi} below as well as Remark 
\ref{technical-aspects-multiple-paths-remark}.

\medskip 

Computing the values in Table \ref{jump-table} is an easy exercise using equation \eqref{eqn-jump}. Calculating $\jump{j}$ is a key ingredient to calculating $\phcyc{}{j}$ and from this one obtains Table \ref{phcyc-table}. Then it is easy to check Remark \ref{path-interp-phcyc} holds
for $k>1$.

\begin{remark} \label{unique-ext-and-phi}
Recall that a cyclotomic path $p$ of tail weight $(\Lambda_{i_1},\Lambda_{i_2})$ has a unique extension to its tail by an $i_1$-arrow, i.e. $\pextminus{p} = \{i_1\}$. By Remark \ref{path-interp-phcyc}, this implies $\Nexttermminus{p}{k} = \Lambda_{i_1}$. The condition on cyclotomic paths requires that $p(0) \neq p(1)$. When $k > 1$ this implies $\phcyctriv{j}{p}{k} = \jump{j}(\Tii{p}{k})$ for all $j \in I$ and we shall use these two expressions interchangeably. 
\end{remark}

We now apply the lemmas and propositions from Section \ref{sec-pr} 
to $\Tii{p}{k}$.

\begin{proposition} \label{triv-pr-lemma}
Let $p$ be a cyclotomic path of tail weight
$(\Lambda_{p(-1)},\Lambda_{p(0)})$ so that
$\Nexttermminus{p}{k} = \Lambda_{p(-1)}$.
When $k \geq 1$ then the following hold:
\begin{enumerate}

\item
$\Tii{p}{k} \in \rep{\Lambda_{p(0)}} $.

\item If $j \neq p(-1)$ and $\phcyctrivplus{j}{p}{k} = 0$, then
\begin{equation}
\pr{p(0)}\ind \Tii{p}{k} \boxtimes \Lii{j} = \zero.
\end{equation}

\item \label{triv-pr-lemma-p(-1)neqp(k)} If $j \neq p(-1)$ but $\phcyctrivplus{j}{p}{k} \geq 1$ then
\begin{equation}
\pr{p(0)} \ind \Tii{p}{k} \boxtimes \Lii{j} \cong \Tii{p \extheadby{j}}{k+1}.
\end{equation}

\item \label{triv-pr-when-j=p(-1)} When $j = p(-1)$ and $\phcyctrivplus{j}{p}{k} \geq 1$ then there is a surjection
\begin{equation}
\pr{p(0)} \ind \Tii{p}{k} \boxtimes \Lii{j} \twoheadrightarrow \Tii{p \extheadby{j}}{k+1}
\end{equation}
and for all composition factors $K$ of $\ind \Tii{p}{k} \boxtimes \Lii{j}$ such that $K \not\cong \Tii{p \extheadby{j}}{k+1}$, 
\begin{equation}
\phcyc{p(0)}{j}(K) \leq \phcyctriv{j}{p}{k} - 2.
\end{equation}

\end{enumerate}
\end{proposition}

\begin{proof} \mbox{}
\begin{enumerate}

\item This follows from  the definition of cyclotomic
path and can also be verified by inspecting 
$\Char(\Tii{p}{k})$
from Section \ref{Description-of-family}.

\item This follows from Proposition \ref{pr-facts}.\ref{pr-fact-5}.

\item If $p(-1) \neq j$ but $\phcyctrivplus{j}{p}{k} \geq 1$, then either $\phcyctriv{j}{p}{k} = 1$ or $2$. When $\phcyctriv{j}{p}{k} = 1$, the result is implied by Proposition \ref{pr-facts}.\ref{f-and-pr}, Remark \ref{path-interp-phcyc} and Proposition \ref{building-T}. When $\phcyctriv{j}{p}{k} = 2$ then $j$ must be class $\ClassB$ and the head of $p$ can be extended twice by $j$-arrows. An examination of paths in $B^{1,1}$ whose head can be extended twice by class $\ClassB$ arrows shows that $\ep{}(\Tii{p}{k}) = \Lambda_{p(k-1)}$, $a_{p(k-1),j} < 0$, and $\phcyctriv{j}{p}{k-1} = 0$. Hence $\pr{p(0)} \ind \Tii{p}{k-1} \boxtimes \Lii{j} = \zero$. The result then follows from Proposition \ref{pr-facts}.\ref{two-nodes-to-add} and Proposition \ref{building-T}.

\item By Proposition \ref{building-T} there is a surjection
\begin{equation} \label{trivial-surjection-to-path-module}
\ind \Tii{p}{k} \boxtimes \Lii{j} \twoheadrightarrow \Tii{p \extheadby{j}}{k+1} 
\end{equation}
since $\Tii{p \extheadby{j}}{k+1} \cong \ftil{j}\Tii{p}{k}$. Because $\pr{p(0)} \Tii{p \extheadby{j}}{k+1} = \Tii{p \extheadby{j}}{k+1}$, applying $\pr{p(0)}$ to \eqref{trivial-surjection-to-path-module} gives us
\begin{equation}
\pr{p(0)}\ind \Tii{p}{k} \boxtimes \Lii{j} \twoheadrightarrow \Tii{p \extheadby{j}}{k+1}.
\end{equation}
By Proposition \ref{crystal-op-facts}.\ref{all-comp-have-less-i} if $K$ is a composition factor of $\ind \Tii{p}{k} \boxtimes \Lii{j}$ not isomorphic to $\Tii{p \extheadby{j}}{k+1}$ then $\ep{j}(K) \leq \ep{j}(\Tii{p}{k})$ but $\wti{j}(K) = \wti{j}(\Tii{p}{k}) - 2$. By formula \eqref{phcyclo-formula} it follows that
\begin{equation} \label{phi-composition-factor-relation}
\phcyc{p(0)}{j}(K) \leq \phcyctriv{j}{p}{k} -2.
\end{equation}
It is possible $\pr{p(0)}K = \zero$; but \eqref{phi-composition-factor-relation} still holds even in this case.

\end{enumerate}
\end{proof}

%
\section{Main theorems}
\label{sec-main}
%


We introduce some shortcut notation. For a simple $R(\nu)$-module $A$, define
\begin{equation}
\etilch{i_k i_{k-1} \dots i_2 i_1} A := \etilch{i_k} \etilch{i_{k-1}} \dots \etilch{i_2} \etilch{i_1} A
\end{equation}
and analogously,
\begin{equation}
\ftilch{i_k i_{k-1} \dots i_2 i_1} A := \ftilch{i_k} \ftilch{i_{k-1}} \dots \ftilch{i_2} \ftilch{i_1} A.
\end{equation}
If $p$ is a path in $B^{1,1}$ such that $p(0) = i_1$, $p(1) = i_2,$ $\dots,$ $p(k-1) = i_k$, we will also write 
\begin{equation}
\etilch{p;k} := \etilch{i_k i_{k-1} \dots i_2 i_1} \quad\quad \text{and} \quad\quad \ftilch{p;k} := \ftilch{i_k i_{k-1} \dots i_2 i_1}.
\end{equation}
Similarly we will occasionally write
\begin{equation}
\Tiip{i_1i_2\dots i_k} := \Tii{p}{k},
\end{equation}
particularly for small $k$.

\subsection{Existence theorem}
\label{sec-exist}
This is the first of our main theorems, which we term
the ``existence" theorem, and also sometimes refer to as a construction
or algorithm.

\begin{theorem} \label{exist-theorem}
Let $A \in \repL{i}$ be a simple $R(\nu)$-module with $|\nu| \geq 1$ and $i$ not a forbidden element of $I$. 
\begin{enumerate}

\item \label{exist-theorem-1}
There exists a cyclotomic path $p: \{0,1,\dots,k-1\} \rightarrow I$
with $p(0) = i$ of tail weight $(\Lambda_{p(-1)},\Lambda_i)$ and length $k$ such that $\etilch{p;k} A$ is a
simple $R (\nu - \gammaplus{p}{k})$-module in $\repL {{p(-1)}}$.

\item Let $r(A) = k$ be the minimal $k$ such that statement \eqref{exist-theorem-1} holds and let $\Rcal{A} = \etilch{p;k} A$. Then there exists a surjection
\begin{equation} \label{exist-theorem-2-surj}
\pr{i}\ind \Tii{p}{k} \boxtimes \Rcal{A} \twoheadrightarrow A.
\end{equation}
\end{enumerate}
\end{theorem}

Note by Remark \ref{sujrection-onto-pr-remark} this is equivalent to proving $\ind \Tii{p}{k} \boxtimes \Rcal{A} \twoheadrightarrow A$. We further conjecture the following.

\begin{conjecture}
Under the same hypotheses as Theorem \ref{exist-theorem}
\begin{equation}
A \cong \cosoc \pr{i} \ind \Tii{p}{k} \boxtimes \Rcal{A}.
\end{equation}
\end{conjecture}

Note that by Section \ref{sec-LV}, the crystal-theoretic consequence
of Theorem \ref{exist-theorem} is a map
$B(\Lambda_i) \to
\begin{displaystyle}
\bigoplus_{j}
\end{displaystyle}
\Bkr \otimes B(\Lambda_{j})$
given by 
$$[A] \mapsto 
\begin{tikzpicture}[baseline=-2pt]
        \node at (0,0) {\;\;{\tiny{$p(k\!-\!1)$}}\;\;};
        \draw (0,0) ellipse (.49cm and .26cm);
\end{tikzpicture}
\otimes [\Rcal{A}],$$
and where $j$ runs over all possibilities for $p(-1)$.
By abuse of notation, we let
\begin{tikzpicture}[baseline=-2pt]
        \node at (0,0) {\;\;{\tiny{$p(k\!-\!1)$}}\;\;};
        \draw (0,0) ellipse (.49cm and .26cm);
\end{tikzpicture}
be the node in $\Bkr$ that the path $p$ ends at.
In this way, each node of $\Bkr$ corresponds to the infinite
collection of paths $p$ that end at that node, and in turn
to the collection of modules $\Tii{p}{k}$.
(As remarked in the introduction, this is not a categorification,
but it is a useful correspondence.)
If we choose to specify $B(\Lambda_i)$ as above, that further specifies
that of that collection, we care about the $p$ with $p(0)=i$.

To recover the crystal isomorphism 
\eqref{eq-perfect} or \eqref{eq-nonperfect}
from Theorem \ref{exist-theorem}
one must actually fix $p(-1)$ 
and let $i$ vary
(whereas the theorem fixes $i$).
In many types $X_\ell$, $i \in I$,
 specifying $p(0)=i$ determines $p(-1)$  
(in particular when $\Lambda_i$ is of level $1$ and $\Bkr$ is perfect).
In type $A$ the relationship between \eqref{eq-perfect} and
\eqref{exist-theorem-2-surj} is transparent.
\cite{V15} discusses the crystal isomorphism in type $A$ in more depth.
Examples \ref{ex-typeC}, \ref{ex-modC} and Remark \ref{rem-typeC}
discuss a case when there are two choices for $p(-1)$,
even when $p$ is a cyclotomic path.
Note that $p(0)$ and $p(-1)$ must be ``adjacent" arrows in $\Bkr$.
At the moment, the above theorem  just gives us a map of nodes.
Our second main theorem
in Section \ref{sec-main2} below will show that it is a morphism
of crystals. 

\medskip

We will construct the path $p$ in the proof of Theorem \ref{exist-theorem}. We shall see $p$ is not always unique. For our choice of path $p$ we will also construct intermediate modules $\Rcalj{t}{A} = \etilch{p;t}A$ and refer to this process as our \emph{algorithm}. We first prove a series of lemmas.

\begin{lemma} \label{where-is-D}
Suppose that $A \in \repL{i}$ is a simple $R(\nu)$-module, $D$ is a simple $R(\nu - \gammaplus{p}{k})$-module, $p$ is a cyclotomic path of length $k$ and tail weight $(\Lambda_{p(-1)},\Lambda_i)$, and there is a surjection
\begin{equation}
\ind \Tii{p}{k} \boxtimes D \twoheadrightarrow A.
\end{equation}
Then $D \in \rep{\Lambda_{p(-1)} + \Nexttermplus{p}{k}}$. 
\end{lemma}

\begin{proof}
By Remark \ref{unique-ext-and-phi}, $\Nexttermminus{p}{k} = \Lambda_{p(-1)}$ so in particular $\Lambda_{p(-1)} + \Nexttermplus{p}{k} = \Nextterm{p}{k}$. Proposition \ref{pr-facts}.\ref{pr-ind-determine-rep} then implies the result.
\end{proof}

\begin{lemma} \label{existence-of-surj}
Let $p$ be a cyclotomic path of length $k$ and tail weight
$(\Lambda_{p(-1)},\Lambda_i)$. Suppose that $A \in \repL{i}$ is a
simple $R(\nu)$-module, $D$ is a simple
$R(\nu - \gammaplus{p}{k})$-module and there is a surjection
\begin{equation}
\ind \Tii{p}{k} \boxtimes D \twoheadrightarrow A.
\end{equation}
If $j \in \pextplus{p}$ such that $\epch{j}(D) \neq 0$ and either $j \neq p(-1)$ or $\epch{j}(D) = \phcyctriv{j}{p}{k}$, then there is a surjection
\begin{equation}
\label{eq-lemma-exist-surj}
\ind \Tii{p \extheadby{j}}{k+1} \boxtimes \etilch{j}D \twoheadrightarrow A.
\end{equation}
\end{lemma}

\begin{proof}
Since $\epch{j}(D) \neq 0$, there is a surjection
\begin{equation} 
\ind \Lii{j} \boxtimes \etilch{j}D \twoheadrightarrow D.
\end{equation}
Then by the exactness of induction we get
\begin{equation} \label{surjection-not-K}
\ind \Tii{p}{k} \boxtimes \Lii{j} \boxtimes \etilch{j}D \twoheadrightarrow A.
\end{equation}
First suppose that $p(-1) \neq j$. Then by Proposition \ref{triv-pr-lemma}.\ref{triv-pr-lemma-p(-1)neqp(k)}
\begin{equation} 
\pr{i} \ind \Tii{p}{k} \boxtimes \Lii{j} \cong \Tii{p \extheadby{j}}{k+1}
\end{equation}
and by Proposition \ref{pr-facts}.\ref{pr-fact-4} applying $\pr{i}$ to \eqref{surjection-not-K} gives
\begin{equation} \label{surjection-through-Tp}
\ind \Tii{p\extheadby{j}}{k+1} \boxtimes \etilch{j}D \twoheadrightarrow \pr{i} \ind \Tii{p \extheadby{j}}{k+1} \boxtimes \etilch{j}D \twoheadrightarrow \pr{i}A \cong A.
\end{equation}
Second, if $j = p(-1)$ then both $\Nexttermminus{p}{k}$ and $\Nexttermplus{p}{k}$ have $\Lambda_j$ in their support so that $\epch{j}(D) = \phcyctriv{j}{p}{k} \geq 2$. By Proposition \ref{building-T}, $\pr{i} \ind \Tii{p}{k} \boxtimes \Lii{j}$ has cosocle isomorphic to $\Tii{p \extheadby{j}}{k+1}$. So if \eqref{surjection-through-Tp} doesn't hold then by Remark \ref{factoring-through-convention} there exists some other composition factor $K$ of $\ind \Tii{p}{k} \boxtimes \Lii{j}$ such that 
\begin{equation}
\ind K \boxtimes \etilch{j}D \twoheadrightarrow A.
\end{equation}
By Proposition \ref{triv-pr-lemma}.\ref{triv-pr-when-j=p(-1)}, $K$ satisfies $\phcyc{i}{j}(K) = \phcyc{i}{j}(\Tii{p}{k}) - 2 =: d$. Observe that 
\begin{equation}
\epch{j}(\etilch{j}D) = \epch{j}(D) - 1 = \phcyctriv{j}{p}{k} -1 = d+1 > \phcyc{i}{j}(K),
\end{equation}
which implies that
\begin{equation}
\pr{i} \ind K \boxtimes \Lii{j^{d+1}} \boxtimes (\etilch{j})^{d+2} D \twoheadrightarrow A.
\end{equation}
But $\pr{i} \ind K \boxtimes \Lii{j^{d+1}} \boxtimes (\etilch{j})^{d+2}
D = \zero$ by Proposition \ref{pr-facts}.\ref{pr-fact-5}. Hence the
surjection \eqref{surjection-not-K} must factor through $\ind \Tii{p
\extheadby{j}}{k+1} \boxtimes \etilch{j} D$, i.e.
\eqref{eq-lemma-exist-surj}
holds. 
\end{proof}

\begin{lemma} \label{existence-of-difficult-surj-lemma}
Suppose $i$ is of class $\ClassB$ and that locally in $B^{1,1}$
\begin{center}
\begin{tikzpicture}
\draw (0,0) ellipse (.3cm and .2cm);
\draw (4,0) ellipse (.3cm and .2cm);
\draw (2,0) ellipse (.3cm and .2cm);
\draw[->,thick] (.35,0) -- (1.65,0);
\draw[->,thick] (2.35,0) -- (3.65,0);
\draw[->,thick] (-1.7,0) -- (-.35,0);
\draw[->,thick] (4.35,0) -- (5.85,0);
\node at (1.1,.4) {$i$};
\node at (2.9,.4) {$i$};
\node at (-1,.4) {$h$};
\node at (4.95,.4) {$h$};
\end{tikzpicture}
\end{center}
Set $\Lambda = \Lambda_i + \Lambda_h$ and suppose that $D \in \repL{}$
is  simple,
$\etilch{h}D \in \rep{2\Lambda_i}$, and
$\etilch{h}D \notin \rep{\Lambda_i}$. Then \begin{enumerate}
\item \label{difficult-surjection-a} $\epch{}(\etilch{ih}D) = \Lambda_i$,
\item There exists a surjection
\begin{equation} \label{first-lemma-surjection}
\ind \Tiip{hi} \boxtimes \etilch{ih}D \twoheadrightarrow D
\end{equation}
\end{enumerate}

\end{lemma}

\begin{proof} 
\begin{enumerate}
\item As $\epch{}(\etilch{h}D) = 2\Lambda_i$ there is a surjection
\begin{equation}
\ind \Lii{i} \boxtimes \etilch{ih}D \twoheadrightarrow \etilch{h}D.
\end{equation}
Because $i$ is a class $\ClassB$ node, $a_{ij} = 0$ unless $j = h$ or $j = i$ in which case
$a_{hi} = -1$ or $a_{ii} = 2$ respectively. Using equation \eqref{phcyclo-formula}, $\p^{2\Lambda_i}(\Lii{i}) = \Lambda_i + \Lambda_h$, and it follows from Proposition \ref{pr-facts}.\ref{pr-fact-5} that $\etilch{ih}D \in \rep{\Lambda}$. We therefore only need to show that $\epch{h}(\etilch{ih}D) = 0$.

There exist $m_i, m_h \in \MB{Z}_{> 0}$, a simple
$R(m_i \alpha_i + m_h \alpha_h)$-module $K \in \repL{}$, and a simple
$R(\nu - (m_i \alpha_i + m_h \alpha_h))$-module $N$ such that
$\epch{i}N = \epch{h}N = 0$, and a surjection
\begin{equation}
\ind K \boxtimes N \twoheadrightarrow D.
\end{equation}
Then $K \in \rep{\Lambda}$, $\etilch{h}K \in \rep{2\Lambda_i}$, and $\etilch{h}K \notin \rep{\Lambda_i}$. An analysis of the classification of simple
modules in $\rep{\Lambda_i + \Lambda_h}$
 in the Appendix shows that $K$ must be isomorphic to one of the following:
\begin{enumerate}
\item \label{case1} $\Lcal{hii} \cong \ftilch{h}\ftilch{i}\Lii{i}$, 
\item \label{case2} $\ftil{h}\Lcal{hii} \cong \ftilch{h}\ftilch{i}\Lcal{ih}$, which has character $(1+q_i^2)[hiih]$ (up to overall grading shift),
\item \label{case3} $\ind \Lcal{hi} \boxtimes \Lcal{ihi} \cong \ftilch{h}\ftilch{i}\Lcal{ihi}$.
\end{enumerate}  
By Proposition \ref{when-can-apply-etil-directly}, it follows that there is a surjection,
\begin{equation}
\ind \etilch{ih}K \boxtimes N \twoheadrightarrow \etilch{ih}D.
\end{equation}
From the three cases above, $\etilch{ih}K$ is isomorphic to either $\Lii{i}$, $\Lcal{ih}$, or $\Lcal{ihi}$. In all three cases $\epch{h}(\etilch{ih}K) = 0$. Since $\epch{h}(N) = 0$, it follows from the Shuffle Lemma that $\epch{h}(\etilch{hi}D) = 0$.

\item There is a nonzero map 
\begin{equation} \label{Lij-or-Lji-factorization-thru-surj}
\ind \Lii{h} \boxtimes \Lii{i} \boxtimes \etilch{ih}D \twoheadrightarrow D.
\end{equation}
This surjection factors through either $\ind \Lcal{ih} \boxtimes \etilch{ih}D$ or $\ind \Lcal{hi} \boxtimes \etilch{ih}D$. Since $\epch{h}(D) = 1$ but $\epch{h}(\Lcal{ih}) = 0$ and $\epch{h}(\etilch{ih}D) = 0$ by part \eqref{difficult-surjection-a}, then $\ech{h}(D) \neq \zero$ but $\ech{h}(\ind \Lcal{ih} \boxtimes \etilch{hi}D) = \zero$ a contradiction to the exactness of $\ech{h}$. It follows that \eqref{Lij-or-Lji-factorization-thru-surj} factors through
\begin{equation}
\ind \Lcal{hi} \boxtimes \etilch{ih}D \twoheadrightarrow D.
\end{equation}

\end{enumerate}

\end{proof}

\begin{lemma} \label{application-of-difficult-surjection}
Let $i$ belong to class $\ClassB$ such that locally in $B^{1,1}$
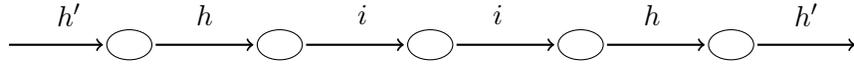
\begin{center}
\begin{tikzpicture}
\draw (0,0) ellipse (.3cm and .2cm);
\draw (4,0) ellipse (.3cm and .2cm);
\draw (2,0) ellipse (.3cm and .2cm);
\draw (-2,0) ellipse (.3cm and .2cm);
\draw (6,0) ellipse (.3cm and .2cm);
\draw[->,thick] (-3.6,0) -- (-2.35,0);
\draw[->,thick] (.35,0) -- (1.65,0);
\draw[->,thick] (6.35,0) -- (7.65,0);
\draw[->,thick] (2.35,0) -- (3.65,0);
\draw[->,thick] (-1.65,0) -- (-.35,0);
\draw[->,thick] (4.35,0) -- (5.65,0);
\node at (1.1,.4) {$i$};
\node at (2.9,.4) {$i$};
\node at (-1,.4) {$h$};
\node at (4.95,.4) {$h$};
\node at (7,.4) {$h'$};
\node at (-2.8,.4) {$h'$};
\end{tikzpicture}
\end{center}
Suppose that $p$ is a cyclotomic path in $B^{1,1}$ of tail weight $(\Lambda_i,\Lambda_i)$ and length $k+2$ such that $p(k-2) = h'$, $p(k-1) = h$, $p(k) = p(k+1) = i$. 
Suppose further that $A$ is a simple $R(\nu)$-module
in $\repL{i}$,
$D$ is a simple $R(\nu-\gammaplus{p}{k})$-module
such that $D \in \rep{2\Lambda_i}$, $D \notin \rep{\Lambda_i}$,
and $\ftilch{h}D \in \rep{\Lambda_i + \Lambda_h}$,
and there are surjections
\begin{equation}
\ind \Tii{p}{k} \boxtimes D \twoheadrightarrow A,
\end{equation}
\begin{equation} \label{double-ii-ext-second-surj}
\ind \Tii{p}{k-1} \boxtimes \ftilch{h}D \twoheadrightarrow A.
\end{equation}
Then $\etilch{i}D \in \rep{\Lambda_i}$ and there is a surjection
\begin{equation} \label{ext-of-surjection}
\ind \Tii{p}{k+1} \boxtimes \etilch{i}D \twoheadrightarrow A
\end{equation}
\end{lemma}

\begin{proof}
The claim that $\etilch{i}D \in \rep{\Lambda_i}$ follows immediately from Lemma \ref{existence-of-difficult-surj-lemma} and the fact that $\ftilch{h}D \in \rep{\Lambda_i + \Lambda_h}$, $D \in \rep{2\Lambda_i}$, $D \notin \rep{\Lambda_i}$.

To construct the surjection \eqref{ext-of-surjection} we first 
prove that
\begin{equation} \label{pr-joins-3-step-triv}
\pr{i} \ind \Tii{p}{k-1} \boxtimes \Tiip{hi} \cong \Tii{p}{k+1}.
\end{equation}
Recall the notation $\Tiip{hi} = \Lcal{hi}$.
Examining Section \ref{Description-of-family}, we see
$\ep{h}(\Tii{p}{k+1}) = 0$, so recalling Proposition 
\ref{e-and-short-exact-sequence}
the map 
$$ \ind \Tii{p}{k-1} \boxtimes \Lii{h} \boxtimes \Lii{i}
\twoheadrightarrow \Tii{p}{k+1}$$
cannot factor through $\ind \Tii{p}{k-1} \boxtimes\Lcal{ih}$
yielding surjections
\begin{equation}
\ind \Tii{p}{k-1} \boxtimes \Lii{h} \boxtimes \Lii{i} \twoheadrightarrow \ind \Tii{p}{k-1} \boxtimes \Tiip{hi} \twoheadrightarrow \Tii{p}{k+1}.
\end{equation}
Since $i = p(0) = p(-1) \neq h$, Proposition
\ref{pr-facts}.\ref{pr-fact-4}, Proposition
\ref{triv-pr-lemma}.\ref{triv-pr-lemma-p(-1)neqp(k)},
the fact
$\pr{i}\Tii{p}{k+1} \cong \Tii{p}{k+1}$, and
the right exactness of $\pr{}$  imply we have 
\begin{equation} \label{surjection-for-ii-extension}
\begin{tikzpicture}
\node at (0,0) {
$\pr{i} \ind \Tii{p}{k-1} \boxtimes \Lii{h} \boxtimes \Lii{i} \twoheadrightarrow \pr{i} \ind \Tii{p}{k-1} \boxtimes \Tiip{hi} \twoheadrightarrow \Tii{p}{k+1}$.
};
\node[rotate=90] at (-1.8,-.5) {$\cong$};
\node at (-2.3,-1) {$\pr{i} \ind \Tii{p}{k} \boxtimes \Lii{i} $};
\end{tikzpicture}
\end{equation}
Keeping in mind the lower bounds on $\ell$ and  by inspection
of the types $X_\ell$ that have class $\ClassB$ nodes, we see
that $(h,h')$ do {\em not} form a class $\ClassD$ pair.
Hence
all composition factors $K \not\cong \cosoc(\ind \Tii{p}{k} \boxtimes \Lii{i}) \cong \Tii{p}{k+1}$ of $\pr{i} \ind \Tii{p}{k} \boxtimes \Lii{i}$ have $\ep{}(K) = \Lambda_{h}$.
As $\e{h}(\ind \Tii{p}{k-1} \boxtimes \Tiip{hi}) = \zero$ it
follows that \begin{equation} \label{gluing-together-trivs}
\pr{i} \ind \Tii{p}{k-1} \boxtimes \Tiip{hi} \cong \Tii{p}{k+1}. 
\end{equation}

By Lemma \ref{existence-of-difficult-surj-lemma} there is a surjection
\begin{equation} \label{assumption-surjection-2}
\ind \Tiip{h i} \boxtimes \etilch{i}D \twoheadrightarrow \ftilch{h}D.
\end{equation}
The exactness of induction and \eqref{double-ii-ext-second-surj}
then give
\begin{equation}
\ind \Tii{p}{k-1} \boxtimes \Tiip{h i} \boxtimes \etilch{i}D \twoheadrightarrow A.
\end{equation}
By \eqref{gluing-together-trivs},
Proposition \ref{pr-facts}.\ref{pr-fact-4} implies that 
\begin{equation}
\pr{i} \ind \Tii{p}{k+1} \boxtimes \etilch{i}D \twoheadrightarrow A.
\end{equation}

\end{proof}

We now prove our first main theorem,  Theorem \ref{exist-theorem}.

\begin{proof}
Since $i$ is not a forbidden element of $I$ then Lemma \ref{cyclotomic-path-lemma} states that we can find at least one cyclotomic path of length 1 and tail weight $(\Lambda_{p(-1)},\Lambda_i)$ in $B^{1,1}$ for some $p(-1) \in I$. Choose one of these cyclotomic paths and call it $p_1$. We will construct the path $p$ in the statement of the theorem by repeatedly extending the head of $p_1$. In particular $p_t$ will be the length $t$ path that results from extending the head of $p_{t-1}$ by some $j \in \pextplus{p_{t-1}}$ so that $p_t := p_{t-1} \extheadby{j}$. Since the tail is fixed for $p_1, p_2, \dots, p_t, \dots$, all $p_t$ will be cyclotomic paths of tail weight $(\Lambda_{p(-1)},\Lambda_i)$, and therefore for all $t \geq 1$, $\Nexttermminus{p_t}{t} = \Lambda_{p(-1)}$. When $i$ is part of a class $\ClassD$ pair $(i,j)$ then $p(-1) = j$, and we add the requirement to $p_t$ that we \emph{favor} extension by $i$-arrows over extension by $j$-arrows. In other words when $\pextplus{p_t} = \{i,j\}$ set 
\begin{equation}
p_{t+1} = p_t \extheadby{i} \quad \text{and} \quad p_{t+2} = p_{t+1} \extheadby{j}
\end{equation}
unless the algorithm terminates before $t+2$. 

Set $\Rcalj{0}{A} := A$, and 
\begin{equation}
\Rcalj{t}{A} := \etilch{p_t(t-1) p_t(t-2) \dots p_t(1) p_t(0)}A = \etilch{p_t;t}A.
\end{equation}
We will show inductively that there exist surjections
\begin{equation}
\pr{i} \ind \Tii{p_t}{t} \boxtimes \Rcalj{t}{A} \twoheadrightarrow A.
\end{equation}
For each of these surjections Lemma \ref{where-is-D} implies that $\Rcalj{t}{A} \in \rep{\Lambda_{p(-1)} + \Nexttermplus{p_t}{t}}$. The induction will end at the smallest $k$ such that $\Rcalj{k}{A} \in \rep{\Lambda_{p(-1)}}$. Then we set $p = p_k$, $r(A) = k$, and $\Rcal{A} = \Rcalj{k}{A}$.

In the base case $t = 0$, $\Rcalj{0}{A} = A$. If $|\nu| =1$, then $A = \Lii{i}$ and $\etilch{p_1(0)}\Lii{i} \cong \UnitModule \in \rep{\Lambda_{p(-1)}}$, so $r(A) = 1$. The existence of the surjection in this case is vacuous. 

Assume that $|\nu| > 1$. Then $\Rcalj{1}{A} = \etilch{p_1(0)}A \neq \zero, \UnitModule$. By Proposition \ref{crystal-op-facts}.\ref{ftil-and-ep} there is a surjection
\begin{equation} 
\ind \Tiip{p_1(0)} \boxtimes \Rcalj{1}{A} \cong \ind \Lii{p_1(0)} \boxtimes \etilch{p_1(0)}A \twoheadrightarrow A.
\end{equation}
As noted above $\Rcalj{1}{A} \in \rep{\Nexttermplus{p_1}{1} + \Lambda_{p(-1)}}$.

Now suppose that we have shown that for $1 \leq t \leq r$ there exist surjections
\begin{equation} \label{inductive-hypthesis}
\pr{i} \ind \Tii{p_t}{t} \boxtimes \Rcalj{t}{A} \twoheadrightarrow A
\end{equation}
with $\Rcalj{t}{A} \in \rep{\Lambda_{p(-1)} + \Nexttermplus{p_t}{t}}$ but for $t<r$,
$\Rcalj{t}{A} \notin \rep{\Lambda_{p(-1)}}$. If in fact $\Rcalj{r}{A} \in \rep{\Lambda_{p(-1)}}$ then we are done and we set $p_r = p$, $r(A) = r$, and $\Rcal{A} = \Rcalj{r}{A}$. If not, then there is some $j \in I$ with $\Lambda_j$ in the support of $\Nexttermplus{p_{r}}{r}$ such that either $j \neq p(-1)$ and $\epch{j}(\Rcalj{r}{A}) \geq 1$, or $j = p(-1)$ and $\epch{p(-1)}(\Rcalj{r}{A}) \ge 2$.
By Remark \ref{path-interp-phcyc} in either case we can extend the head of $p_{r}$ by a $j$-arrow  and set $p_{r+1} := p_{r} \extheadby{j}$. $p_{r+1}$ is a cyclotomic path of tail weight $(\Lambda_{p(-1)},\Lambda_i)$ and length $r+1$. Recall that if $(p(-1),i)$ is a class $\ClassD$ pair, then we always favor extension by $i$-arrows. Set $\Rcalj{r+1}{A} := \etilch{p_{r+1}(r)}\Rcalj{r}{A} = \etilch{j}\Rcalj{r}{A}$. When $j \neq p(-1)$, Lemma \ref{existence-of-surj} implies that there is a surjection
\begin{equation} \label{inductive-achieved-surjection}
\pr{i} \ind \Tii{p_{r+1}}{r+1} \boxtimes \Rcalj{r+1}{A} \twoheadrightarrow A,
\end{equation}
and then Lemma \ref{where-is-D} and \eqref{inductive-achieved-surjection} imply that $\Rcalj{r+1}{A} \in \rep{\Lambda_{p(-1)} + \Nexttermplus{p_{r+1}}{r+1}}$. 

If $j = p(-1)$ and $\epch{p(-1)}(\Rcalj{r}{A}) > 1$, then from the formulas for $\Nexttermplus{p}{r}$ in Table \ref{phcycplus-table}, $\epch{p(-1)}(\Rcalj{r}{A}) \leq \phcyctrivplus{p(-1)}{p}{r} + 1 \leq 3$. If $\epch{p(-1)}(\Rcalj{r}{A}) = 3$ then $\epch{p(-1)}(\Rcalj{r}{A}) = \phcyctriv{p(-1)}{p_{r}}{r}$. In this case Lemma \ref{existence-of-surj} implies that surjection \eqref{inductive-achieved-surjection} exists and as before $\Rcalj{r+1}{A} \in \rep{\Lambda_{p(-1)} + \Nexttermplus{p_{r+1}}{r+1}}$. 

The only remaining case is when $\epch{p(-1)}(\Rcalj{r}{A}) = 2$,
and
 either $\phcyctriv{p(-1)}{p_{r}}{r} = 2$ or $\phcyctriv{p(-1)}{p_{r}}{r} = 3$. When the former holds again Lemma \ref{existence-of-surj} gives surjection  \eqref{inductive-achieved-surjection} and
hence
 $\Rcalj{r+1}{A} \in \rep{\Lambda_{p(-1)} + \Nexttermplus{p_{r+1}}{r+1}}$. If $\phcyctriv{p(-1)}{p_{r}}{r} = 3$, then because $p_{r+1}$ is a cyclotomic path, $j = p(-1) = p(0) = i$ is a class $\ClassB$ node, $p_{r}(r-1) \neq p(-1)$ and $r \geq 4$. Because $r \geq 4$ the inductive hypothesis implies that there is a surjection
\begin{equation}
\ind \Tii{p_{r}}{r-1} \boxtimes \Rcalj{r-1}{A} \twoheadrightarrow A
\end{equation}
(here we make the identification $\Tii{p_{r-1}}{r-1} \cong \Tii{p_{r}}{r-1}$ as $p_r(t) = p_{r-1}(t)$ for all $0 \leq t < r-1$). Then Lemma \ref{application-of-difficult-surjection} implies that there is a surjection
\begin{equation}
\pr{i}\ind \Tii{p_{r+1}}{r+1} \boxtimes \Rcalj{r+1}{A} \twoheadrightarrow A
\end{equation}
and $\Rcalj{r+1}{A} \in \rep{\Lambda_{p(-1)}}$ so $\Rcal{A} := \Rcalj{r+1}{A}$ and $r(A) = r+1$ and the induction terminates. This proves the inductive step. 

We continue the induction until we reach the first $k$ such that $\Rcalj{k}{A} \in \rep{\Lambda_{p(-1)}}$. At this point we set $p = p_k$, $r(A) = k$, $\Rcal{A} = \Rcalj{k}{A}$. $\Rcalj{k}{A}$ is a $R(\nu - \gammaplus{p_k}{k})$-module. As $|\nu - \gammaplus{p}{k}| = |\nu| - k$, the algorithm necessarily terminates at or before $k = |\nu|$. If $k = |\nu|$ then $\Rcalj{|\nu|}{A} = \UnitModule \in \rep{\Lambda_{p(-1)}}$. In this case 
\begin{equation}
\Tii{p_{|\nu|}}{|\nu|} \cong \pr{i} \ind \Tii{p_{|\nu|}}{|\nu|} \boxtimes \UnitModule \twoheadrightarrow A
\end{equation}
so that in fact $A \cong \Tii{p_{|\nu|}}{|\nu|}$.

 Note that for all $r \leq r(A)$ and $t < r$, $p(t) = p_r(t)$ so all intermediate paths constructed above are compatible.

\end{proof}

In Theorem \ref{exist-theorem} we chose $p = p_k$ such that whenever $(i,j)$ is a class $\ClassD$ pair and $\{p(t-1),p(t) \} = \{i,j\}$ we had $p(t-1) = i,$ $p(t) = j$. However if $p'$ agree with $p$ except that for $t \in \adjac{j}{i}{\und{\trivseq{p}{k}}}$ we allow $\{p'(t-1),p'(t)\} = \{i,j\}$ in either order then in fact $\Tii{p}{k} \cong \Tii{p'}{k}$ even though $p \neq p'$.
This also applies to class $\ClassD$ pairs $(j, j')$ when $i \notin 
\{j,j'\}$. 
Here $p$ and $p'$ take different routes around the bifurcation in $B^{1,1}$ generated by the class $\ClassD$ pair $(i,j)$. (Of course if $p(k-2) \notin \{i,j\}$ and $p(k-1) = i$ we must also have $p'(k-1) = i$ since $p$ ends before it traverses the entire bifurcation). This phenomenon only occurs in types $D^{(1)}_\ell$, $B^{(1)}_\ell$, and $A^{(2)}_{2\ell-1}$. (See Sections \ref{perfect-crystal-section} and \ref{Description-of-family}.)

The following proposition shows that choosing route $p'$ over $p$ does not affect the output of the algorithm from Theorem \ref{exist-theorem}. In other words, after the initial choice of length 1 cyclotomic path $p_1$, $\Rcal{A}$ and $r(A)$ are well-defined. Part \ref{taking-different-paths-1} of Proposition \ref{taking-different-paths} shows that $\etilch{p;k}A = \etilch{p';k}A$ while part \ref{taking-different-paths-2} in conjunction with part \ref{taking-different-paths-1} shows that $r(A)$ is well-defined. It follows from this that in fact $\Rcal{A}$ is well-defined.

\begin{proposition} \label{taking-different-paths}
Suppose that locally in $B^{1,1}$
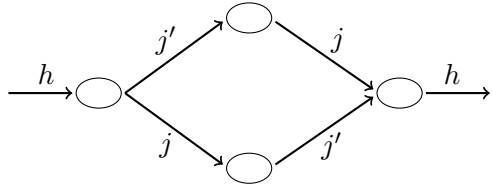
\begin{center}
\begin{tikzpicture}
\draw (0,0) ellipse (.3cm and .2cm);
\draw (4,0) ellipse (.3cm and .2cm);
\draw (2,1) ellipse (.3cm and .2cm);
\draw (2,-1) ellipse (.3cm and .2cm);
\draw[->,thick] (.35,0) -- (1.65,.95);
\draw[->,thick] (.35,0) -- (1.65,-.95);
\draw[->,thick] (2.35,.95) -- (3.65,.05);
\draw[->,thick] (2.35,-.95) -- (3.65,-.05);
\draw[->,thick] (-1.2,0) -- (-.35,0);
\draw[->,thick] (4.35,0) -- (5.2,0);
\node at (.9,.7) {$j'$};
\node at (.9,-.7) {$j$};
\node at (3.18,.7) {$j$};
\node at (3.1,-.7) {$j'$};
\node at (-.7,.25) {$h$};
\node at (4.7,.25) {$h$};

\end{tikzpicture}
\end{center}
so that $(j,j')$ is a class $\ClassD$ pair. Let $A \in \repL{i}$
be a simple $R(\nu)$-module, and suppose $r$ steps of the algorithm from Theorem \ref{exist-theorem} have been executed so 
we have constructed 
a simple $R(\nu - \gammaplus{p}{r})$-module $\Rcalj{r}{A}$, a cyclotomic path $p_r$ with length $r$ and tail weight $(\Lambda_{p(-1)},\Lambda_i)$, and a surjection
\begin{equation}
\ind \Tii{p_r}{r} \boxtimes \Rcalj{r}{A} \twoheadrightarrow A
\end{equation}
Finally, suppose that $p_r(r-1) = h$, $\Nexttermplus{p_r}{r} = \Lambda_{j} + \Lambda_{j'}$. Then
\begin{enumerate}
\item \label{taking-different-paths-1} $\etilch{j'}\etilch{j}\Rcalj{r}{A} \cong \etilch{j}\etilch{j'}\Rcalj{r}{A}$.
\item \label{taking-different-paths-2}  If $\etilch{j}\Rcalj{r}{A} \in \rep{\Lambda_{p(-1)}}$ then $\Rcalj{r}{A} \in \rep{\Lambda_{p(-1)} + \Lambda_j}$.
\end{enumerate}
\end{proposition}

\begin{proof}
Part \ref{taking-different-paths-1} follows directly from Proposition \ref{when-ei-ej-commute} and the fact that for class $\ClassD$ pairs $(j,j')$, $a_{jj'} = 0$. Part \ref{taking-different-paths-2} is a direct application of part \ref{taking-different-paths-1} after noting that by Lemma \ref{where-is-D}, $\Rcalj{r}{A} \in \rep{\Lambda_{p(-1)} + \Lambda_j + \Lambda_{j'}}$.
\end{proof}

\begin{remark} \label{technical-aspects-multiple-paths-remark}
Let $(j,j')$ be a class $\ClassD$ pair.
The attentive reader will notice that the case where $\Rcalj{r}{A} \cong \ind \Lii{j} \boxtimes \Lii{j'}$ with $p(-1) = j$ and $i = j'$ could potentially give two different values of $r(A)$ depending on whether one chooses to extend $p_r$ by $j'$ or $j$. In the former case $\Rcal{A} = \Rcalj{r+1}{A} = \Lii{j} \in \rep{\Lambda_{p(-1)}}$ and $r(A) = r+1$. In the latter case $\Rcalj{r+1}{A} = \Lii{j'} \notin \rep{\Lambda_{p(-1)}}$ and the algorithm tells us to continue to get $\Rcal{A} = \Rcalj{r+2}{A} = \UnitModule$, yielding $r(A) = r+2$. Recall that when $(i,p(-1))$ is a class $\ClassD$ pair, then the algorithm requires that when building
our path $p$, we always favor extending by $i$-arrows so that actually, Theorem \ref{exist-theorem} will always give $\Rcal{A} \cong \Lii{j}$ and $r(A) = r+1$ in this case. We require this favoring precisely to avoid the above situation.
\end{remark}


\subsection{The action of the crystal operators}
\label{sec-main2}


Next we study the action of the crystal operators $\etil{j}$ and $\ftil{j}$ on \eqref{exist-theorem-2-surj} to show that the map in 
part 2 of Theorem \ref{exist-theorem} categorifies our crystal isomorphism $\crystalmap$. 

\begin{remark}
\label{rem-typeC}
In some types, having chosen $p(0) = i$, the choice of $p(1)$ and 
consequently $p_t$, $t < r(A)$ is forced upon us.
In other types, such as $C_\ell^{(1)}$, there can be 2 choices
for $p(1)$ (and hence $p(-1)$).
This choice is mirrored by the combinatorial
structure of $B^{1,1} \otimes B(\Lambda_i)$. See Example
\ref{ex-typeC} and Example \ref{ex-modC}.
\end{remark}

Compare the theorems below with the crystal-theoretic statements
\eqref{eq_ei_tensor} 
and
\eqref{eq_fi_tensor}.
As in \cite{LV11} simple 
modules in $\repL{i}$ correspond to nodes in
the highest weight crystal $B(\Lambda_i)$. Each node $b$ of the KR crystal
$B^{1,1}$ (respectively $B^{\ell,1}$ for type $A^{(1)}_\ell$)
corresponds to an infinite family of $R(\gammaplus{p}{k})$-modules 
$\Tii{p}{k}, k \in \MB{Z}_{\geq 0}$ that satisfy $\ep{}(\Tii{p}{k}) = \ep{}(b)$.
It is in this manner that the main theorems of this paper
give a categorification of the crystal isomorphism $\crystalmap$.

\begin{theorem} \label{thm-action-etil}
Let $A \in \repL{i}$ be a simple $R(\nu)$-module and
$j \in I$ be such that $\etil{j}A \neq \zero$. When $i$ is class $\ClassB$ we furthermore require that $|\nu| > 1$. Let $p$ be a cyclotomic path of tail weight $(\Lambda_{p(-1)},\Lambda_i)$ and length $k = r(A)$, and $\Rcal{A} = \etilch{p;k}A$, as constructed 
by
the algorithm in Theorem \ref{exist-theorem}. Then there exists a surjection 
  \begin{align} 
    \ind \; \etil{j} \Tii{p}{k} \; \boxtimes \; \Rcal{A} \twoheadrightarrow \etil{j}A  \quad \text{if   } \;\;\;\; \ep{j}(\Tii{p}{k} )> \phcyc{p(-1)}{j}(\Rcal{A}), \label{E1}
\\
    \ind \; \Tii{p}{k} \; \boxtimes \; \etil{j}\Rcal{A} \twoheadrightarrow \etil{j}A  \quad \text{if  }\;\;\;\; \ep{j}(\Tii{p}{k}) \leq \phcyc{p(-1)}{j}(\Rcal{A}). \label{E2}
  \end{align} 
\end{theorem}

\begin{theorem} \label{thm-action-ftil}
Let $i$, $p$, $k$, $A$, and $\Rcal{A}$ be as in Theorem
\ref{exist-theorem}. Let $j \in I$ be such that $\pr{i} \ftil{j} A \neq \zero$.  Then there exists a surjection
  \begin{align}
    \ind \; \ftil{j} \Tii{p}{k} \; \boxtimes \; \Rcal{A}  \twoheadrightarrow \ftil{j}A  \quad \text{if} \;\;\;\; \ep{j}(\Tii{p}{k}) \geq \phcyc{p(-1)}{j}(\Rcal{A}), \label{F1}
\\
    \ind \; \Tii{p}{k} \; \boxtimes \; \ftil{j}\Rcal{A} \twoheadrightarrow \ftil{j}A  \quad \text{if} \;\;\;\; \ep{j}(\Tii{p}{k}) < \phcyc{p(-1)}{j}(\Rcal{A}). \label{F2}
  \end{align} 
\end{theorem}

Theorem \ref{thm-action-ftil} follows directly from Theorem  \ref{thm-action-etil}, therefore we will only prove \ref{thm-action-etil}.

\begin{remark}
In the case $\ep{j}(\Tii{p}{k}) \geq \phcyc{p(-1)}{j}(\Rcal{A})$,
$\pr{i}\ftil{j}A \neq 0$ it will hold that $\ftil{j}\Tii{p}{k} \cong \Tii{p \extheadby{j}}{k+1}$ and $r(\ftil{j}A) = r(A) + 1$. 
(This relies in part on the properties of cyclotomic paths.)
\end{remark}

We state a few lemmas below that follow directly from propositions in Section \ref{jump-section}.

\begin{lemma} \label{cutting-ej-looking-at-epchi}
Let $M$ be a simple $R(\nu)$-module and $h \neq j$. If $\etil{j}M \neq \zero$ then $\epch{h}(\etil{j}M) = \epch{h}(M)$.
\end{lemma}

\begin{proof}
This follows from Proposition \ref{commuting-functors}.\ref{2-es}, and the definition of $\epch{h}(M)$.
\end{proof}

\begin{lemma} \label{cutting-ej-looking-at-epchi-2}
Suppose that $D$ is a simple $R(\nu)$-module. If $\etil{j}D \neq \zero$ and $\jump{j}(D) \geq 1$, then $\epch{j}(D) = \epch{j}(\etil{j}D)$. 
\end{lemma}

\begin{proof}
Set $N = \etil{j}D$. By \eqref{ftil-and-jump-ob} $\jump{j}(N) = \jump{j}(\ftil{j}N) + 1 \geq 2$. Proposition \ref{jump-lemma} then implies, $\epch{j}(N) = \epch{j}(\ftil{j}N)$, which is equivalent to $\epch{j}(\etil{j}D) = \epch{j}(D)$.
\end{proof}

In the following lemma we allow $i =j$.

\begin{lemma} \label{ei-app-does-not-alter-algo}
Let $i,j \in I$. Suppose that $D \in \repL{i}$
is a simple $R(\nu)$-module and $ \etil{j}D \neq \zero$. 

\begin{enumerate}

\item \label{ech-commutes-with-f} If $\jump{j}(D) \geq 1$, then 
\begin{enumerate} 
\item $\epch{h}(\etil{j}D) = \epch{h}(D)$ for all $h \in I$,
\item $\etil{j}D \in \rep{\Lambda_{i}}$,
\item $\epch{j}(\etil{j}\ftilch{j}D) = \epch{j}(\ftilch{j}D)$
\end{enumerate}

\item If $\jump{j}(D) \geq 2$, then furthermore $\epch{j}(\etil{j}\ftilch{j}\ftilch{j}D) = \epch{j}(\ftilch{j}\ftilch{j}D)$.

\end{enumerate}
\end{lemma}

\begin{proof}
\begin{enumerate}

\item When $h \neq j$ the equality $\epch{h}(\etil{j}D) = \epch{h}(D)$ holds by Lemma \ref{cutting-ej-looking-at-epchi}, and when $h = j$ it holds by Lemma \ref{cutting-ej-looking-at-epchi-2}. Since $\epch{h}(\etil{j}D) = \epch{h}(D)$ for all $h \in I$ then by Proposition \ref{cyclotomic-char}, $\etil{j}D \in \rep{\Lambda_{i}}$. 

Since $\jump{j}(\etil{j}D) \geq 2$ then by Proposition \ref{commuting-functors}.\ref{etil-ftilch-same} $\etil{j}\ftilch{j}D \cong \ftilch{j}\etil{j}D$. With the equality $\epch{j}(\etil{j}D) = \epch{j}(D)$ this implies that 
 \begin{equation}
 \epch{j}(\etil{j}\ftilch{j}D) = \epch{j}(\ftilch{j}\etil{j}D) = \epch{j}(\etil{j}D) +1 = \epch{j}(D) + 1 = \epch{j}(\ftilch{j}D).
\end{equation} 

\item In the second case we have $\jump{j}(\etil{j}D) \geq 3$, $\jump{j}(\etil{j}\ftilch{j}D) \geq 2$, so that by Proposition \ref{commuting-functors}.\ref{etil-ftilch-same}, $\etil{j}\ftilch{j}\ftilch{j}D \cong \ftilch{j}\etil{j}\ftilch{j}D \cong \ftilch{j}\ftilch{j}\etil{j}D$. Then
\begin{equation}
\epch{j}(\etil{j}\ftilch{j}\ftilch{j}D) = \epch{j}(\ftilch{j}\ftilch{j}\etil{j}D) = \epch{j}(\etil{j}D) + 2 = \epch{j}(D) + 2 = \epch{j}(\ftilch{j}\ftilch{j}D).
\end{equation}

\end{enumerate}
\end{proof}

We are now ready to prove Theorem \ref{thm-action-etil}.

\begin{proof}
By Theorem \ref{exist-theorem} we have a surjection
\begin{equation}
\ind \Tii{p}{k} \boxtimes \Rcal{A} \twoheadrightarrow A
\end{equation}
where $k = r(A)$. 
We will 
first treat the case where $\Rcal{A} \cong \UnitModule$. This implies $A \cong \Tii{p}{k}$, so $\zero \neq \etil{j}A \cong \etil{j}\Tii{p}{k}$. Note $\phcyc{p(-1)}{j}(\Rcal{A}) = \delta_{j,p(-1)}$. However, we can never have $\etil{p(-1)}A = \etil{p(-1)}\Tii{p}{k} \neq \zero$ as then $\ind \Tii{p}{k-1} \boxtimes \Lii{p(-1)} \twoheadrightarrow A$ has $\Lii{p(-1)} \in \rep{\Lambda_{p(-1)}}$ and so $r(A) < k$,
violating its minimality.
(Recall $|\nu| = k > 1$ in this case.)
Therefore $j \neq p(-1)$ and hence $\phcyc{p(-1)}{j}(\Rcal{A}) = 0$. Thus as $\ep{j}(A) = \ep{j}(\Tii{p}{k}) > 0 = \phcyc{p(-1)}{j}(\UnitModule)$ and we always have trivial surjections
\begin{equation}
\ind \Tii{p}{k} \boxtimes \UnitModule \twoheadrightarrow A \quad \text{and} \quad \ind \etil{j}\Tii{p}{k} \boxtimes \UnitModule \twoheadrightarrow \etil{j}A,
\end{equation}
\eqref{E1} always holds and the theorem holds in this case.

For the rest of the proof
assume $\Rcal{A} \neq \UnitModule$. We divide the proof into the three possible values of $\ep{j}(\Tii{p}{k})$: $0$, $1$, or $2$.

\begin{itemize}

\item \emph{Case 1}: $\ep{j}(\Tii{p}{k}) = 0$ 

Since $\ep{j}(A) \neq 0$, by Proposition \ref{when-can-apply-etil-directly} there is a surjection
\begin{equation}
\ind \Tii{p}{k} \boxtimes \etil{j}\Rcal{A} \twoheadrightarrow \etil{j}A.
\end{equation} 
Note $0 \le \phcyc{p(-1)}{j}(\Rcal{A})$ always.

\item \emph{Case 2}: $\ep{j}(\Tii{p}{k}) = 1$

\begin{itemize}

\item \label{Case-2-Crystal-op-proof} \emph{Case 2a}: $\phcyc{p(-1)}{j}(\Rcal{A}) = 0 < 1 = \ep{j}(\Tii{p}{k})$

Since $\Rcal{A} \neq \UnitModule$, Remark \ref{when-phcyc-jump-the-same} implies $\jump{j}(\Rcal{A}) = 0$. Because $\ep{j}(\Tii{p}{k}) = 1$, then $k \geq 1$ and therefore step $(k-1)$ of the algorithm in Theorem \ref{exist-theorem} provides a surjection
\begin{equation} \label{case-2-surjection}
\ind \Tii{p}{k-1} \boxtimes \ftilch{j}\Rcal{A} \twoheadrightarrow A.
\end{equation}
If $j$ is part of a class $\ClassD$ pair $(j,j')$ and $\ep{}(\Tii{p}{k}) = \Lambda_{j} + \Lambda_{j'}$ then by Proposition \ref{taking-different-paths} we can assume that $p(k-1) = j$, $p(k-2) = j'$ so that $\etil{j}\Tii{p}{k} = \Tii{p}{k-1}$ and $\etil{j}\Tii{p}{k-1} = \etil{j}^2\Tii{p}{k} = \zero$.
By Proposition \ref{when-can-apply-etil-directly}.\ref{ei-case-for-direct-application-of-etili-b}, as $\ep{j}(\Tii{p}{k-1}) = 0$,
 we have
\begin{equation}
\ind \Tii{p}{k-1} \boxtimes \etil{j}\ftilch{j} \Rcal{A} \twoheadrightarrow \etil{j}A.
\end{equation}
Because $\jump{j}(\Rcal{A}) = 0$, by Proposition \ref{jump-lemma} $\ftilch{j}\Rcal{A} \cong \ftil{j}\Rcal{A}$, and 
so $\etil{j}\ftilch{j}\Rcal{A} \cong \Rcal{A}$. Since $\Tii{p}{k-1} \cong \etil{j}\Tii{p}{k}$ this gives \eqref{E1}. 

\item \emph{Case 2b}: $\phcyc{p(-1)}{j}(\Rcal{A}) \geq 1 = \ep{j}(\Tii{p}{k})$

Again by Remark \ref{when-phcyc-jump-the-same} this implies that $\jump{j}(\Rcal{A}) \geq 1$ and $\jump{j}(\etil{j}\Rcal{A}) \geq 2$. Note $\etil{j}\Rcal{A} \neq \zero$ because $\jump{j}(\Rcal{A}) \geq 1$ and \eqref{case-2-surjection} still holds as well as $\ep{j}(\Tii{p}{k-1}) = 0$.

If $k = 1$, then $\Tii{p}{k} = \Lii{i}$ which forces $j = i$ as $\ep{j}(\Lii{i}) = \delta_{ij}$. Further $\Rcal{A} = \etilch{j}A$. We have
\begin{equation}
\ind \Lii{j} \boxtimes \etilch{j}A \twoheadrightarrow A.
\end{equation}
Because $\jump{j}(\etilch{j}A) = \jump{j}(\Rcal{A}) \geq 1$ then $\jump{j}(\etil{j}\etilch{j}A) \geq 2$. Proposition \ref{commuting-functors}.\ref{etil-ftilch-same} then implies that $\etil{j}A \cong \etil{j}\ftilch{i}\etilch{i}A \cong \etil{j}\ftilch{j}\etilch{j}A \cong \ftilch{j}\etil{j}\etilch{j}A$. So there is a surjection
\begin{center}
\begin{tikzpicture}
\node at (0,0) {$\ind \Lii{j} \boxtimes \etil{j}\etilch{j}A  \twoheadrightarrow \etil{j}A$};
\node[rotate=90] at (-.6,-.5) {$\cong$};
\node at (-.6,-1) {$\ind \Tii{p}{1} \boxtimes \etil{j}\Rcal{A}$};
\end{tikzpicture}
\end{center}
and \eqref{E2} holds.

Now assume that $k > 1$ so that by the algorithm in Theorem \ref{exist-theorem}, there is a surjection
\begin{equation} \label{case-2-k-bigger-than-1}
\ind \Tii{p}{k-2} \boxtimes \ftilch{p(k-2)}\ftilch{j}\Rcal{A} \twoheadrightarrow A. 
\end{equation}
Note in Case 2 
that $p(k-2) \neq j$ and $\ep{j}(\Tii{p}{k-2}) = 0$ and therefore by Proposition \ref{when-can-apply-etil-directly}
\begin{equation}
\ind \Tii{p}{k-2} \boxtimes \etil{j}\ftilch{p(k-2)}\ftilch{j}\Rcal{A} \twoheadrightarrow \etil{j}A.
\end{equation}
Furthermore since $\jump{j}(\etil{j}\Rcal{A}) \geq 2$, then by Proposition \ref{commuting-functors}.\ref{f-and-e} and \ref{commuting-functors}.\ref{etil-ftilch-same}
\begin{equation} \label{applying-ej-to-fp(k-2)fjR(a)}
\etil{j}\ftilch{p(k-2)}\ftilch{j}\Rcal{A} \cong \ftilch{p(k-2)}\etil{j}\ftilch{j}\Rcal{A} \cong \ftilch{p(k-2)}\ftilch{j}\etil{j}\Rcal{A}.
\end{equation}
This gives us
\begin{equation} \label{pulled-off-two-sur}
\ind \Tii{p}{k-2} \boxtimes \ftilch{p(k-2)}\ftilch{j}\etil{j}\Rcal{A} \twoheadrightarrow \etil{j}A.
\end{equation}
By Lemma \ref{ei-app-does-not-alter-algo},  
\begin{equation}
\epch{p(-1)}(\ftilch{j}\etil{j}\Rcal{A}) = \epch{p(-1)}(\etil{j}\ftilch{j}\Rcal{A}) = \epch{p(-1)}(\ftilch{j}\Rcal{A}),
\end{equation}
and $\etil{j}\Rcal{A} \in \rep{\Lambda_i}$. By the minimality of $r(A)$, $\ftilch{j}\Rcal{A}, \ftilch{p(k-2)}\ftilch{j}\Rcal{A} \notin \rep{\Lambda_{p(-1)}}$, hence also
considering $\jump{j}$,
 $ \etil{j}(\ftilch{j}\Rcal{A}) = \ftilch{j}\etil{j}\Rcal{A} \notin \rep{\Lambda_{p(-1)}}$ and $\ftilch{p(k-2)}\ftilch{j}\etil{j}\Rcal{A} \notin \rep{\Lambda_{p(-1)}}$. When $j,p(k-2) \neq p(-1)$ then $\epch{p(k-2)}(\ftilch{p(k-2)}\ftilch{j}\etil{j}\Rcal{A}) = 1 = \epch{j}(\ftilch{j}\etil{j}\Rcal{A})$ so that applying Lemma \ref{existence-of-surj} twice gives 
\begin{equation} \label{last-surjection-in-crystal-proof} 
\ind \Tii{p}{k} \boxtimes \etil{j}\Rcal{A} \twoheadrightarrow \etil{j}A,
\end{equation}
agreeing with \eqref{E2} of Theorem \ref{thm-action-etil},
which we are in the process of proving.

 When $p(k-2) = p(-1)$, then $j \neq p(-1)$. By the minimality of $r(A)$ and Lemma \ref{where-is-D}, it follows in this case that $\epch{p(-1)}(\ftilch{p(-1)}\ftilch{j}\Rcal{A}) = 2 = \phcyctriv{p(-1)}{p}{k-2}$. By Lemma \ref{cutting-ej-looking-at-epchi} and \eqref{applying-ej-to-fp(k-2)fjR(a)} then $\epch{p(-1)}(\ftilch{p(-1)}\ftilch{j}\etil{j}\Rcal{A}) = 2$ so that again by Lemma \ref{existence-of-surj} applied to \eqref{pulled-off-two-sur} we get 
\begin{equation} \label{last-surjection-in-crystal-proof-2}
\ind \Tii{p}{k-1} \boxtimes \ftilch{j}\etil{j}\Rcal{A} \twoheadrightarrow \etil{j}A
\end{equation}
and another application of Lemma \ref{existence-of-surj} gives \eqref{E2}. 

The last case to consider of Case 2b  is when $j = p(-1)$. Then $p(k-2) \neq p(-1)$ since $\ep{j}(\Tii{p}{k}) = 1$. Using Lemma \ref{existence-of-surj} gives \eqref{last-surjection-in-crystal-proof-2}
as above. Then we either have $\epch{p(-1)}(\ftilch{p(-1)}\etil{p(-1)}\Rcal{A}) = 2 = \phcyctriv{p(-1)}{p}{k-1}$ if $p(-1)$ is class $\ClassA$ or class $\ClassD$, or $\epch{p(-1)}(\ftilch{p(-1)}\etil{p(-1)}\Rcal{A}) = 2$ and $\phcyctriv{p(-1)}{p}{k-1} = 3$ when $p(-1)$ is class $\ClassB$. In the former case we can apply Lemma \ref{existence-of-surj} to \eqref{last-surjection-in-crystal-proof-2}, and in the latter case apply Lemma \ref{application-of-difficult-surjection} to \eqref{last-surjection-in-crystal-proof-2} and \eqref{pulled-off-two-sur}, to get the desired surjection \eqref{E2}.

\end{itemize}

\item \emph{Case 3}: $\ep{j}(\Tii{p}{k}) = 2$

Note that in this case $j$ must be class $\ClassB$ and $p(k-3) \neq j$.

Note $k > 2$ because $p$ is a cyclotomic path (if $k = 2$ we would have $p(0) = p(1)$). From Theorem \ref{exist-theorem} we have a surjection
\begin{equation} \label{surjection-echi-2}
\ind \Tii{p}{k-3} \boxtimes \ftilch{p(k-3)}\ftilch{j}\ftilch{j}\Rcal{A} \twoheadrightarrow A.
\end{equation}
By Proposition \ref{when-can-apply-etil-directly}.\ref{eich-case-for-direct-application-of-etili} we can apply $\etil{j}$ to \eqref{surjection-echi-2} to get
\begin{equation} \label{surjection-echi-2-2}
\ind \Tii{p}{k-3} \boxtimes \etil{j}\ftilch{p(k-3)}\ftilch{j}\ftilch{j}\Rcal{A}  \twoheadrightarrow \etil{j}A.
\end{equation}

\begin{itemize}

\item \emph{Case} 3a: $\phcyc{p(-1)}{j}(\Rcal{A}) < 2 = \ep{j}(\Tii{p}{k})$

Then $\jump{j}(\Rcal{A}) \leq 1$ and by formula \eqref{ftil-and-jump-ob}, $\jump{j}(\ftil{j}\Rcal{A}) = 0$. Hence by Proposition \ref{jump-lemma}.\ref{switch-ftil-ftilch}, $\ftil{j}\ftil{j}\Rcal{A} \cong \ftilch{j}\ftil{j}\Rcal{A}$. Furthermore because $p(k-3) \neq j$, by Proposition \ref{commuting-functors}.\ref{f-and-e}, $\etil{j}$ and $\ftilch{p(k-3)}$ commute. Thus we have
\begin{equation}
\etil{j}\ftilch{p(k-3)}\ftilch{j}\ftilch{j}\Rcal{A} \cong \ftilch{p(k-3)}\etil{j}\ftilch{j}\ftilch{j}\Rcal{A} \cong \ftilch{p(k-3)}\etil{j}\ftil{j}\ftilch{j}\Rcal{A} \cong \ftilch{p(k-3)}\ftilch{j}\Rcal{A}.
\end{equation}
The map in \eqref{surjection-echi-2-2} then becomes
\begin{equation} \label{pull-off-two-remove-one}
\ind \Tii{p}{k-3} \boxtimes \ftilch{p(k-3)}\ftilch{j}\Rcal{A} \twoheadrightarrow \etil{j} A.
\end{equation}
Note that since $j$ is class $\ClassB$ then $p(k-3)$ must be class $\ClassA$. If $p(k-3), j \neq p(-1)$, then $\epch{p(k-3)}(\ftilch{p(k-3)}\ftilch{j}\Rcal{A}) = 1$ and $\epch{j}(\ftilch{j}\Rcal{A}) = 1$ so that two applications of Lemma \ref{existence-of-surj} give
\begin{equation}
 \label{second-goal-surjection}
\begin{tikzpicture}
\node at (0,0) {$\ind \Tii{p}{k-1} \boxtimes \Rcal{A} \twoheadrightarrow \etil{j}A$};
\node[rotate=90] at (-.9,-.5) {$\cong$};
\node at (-.9,-1) {$\etil{j}\Tii{p}{k}$};
\end{tikzpicture}
\end{equation}
yielding \eqref{E1} of Theorem \ref{thm-action-etil}. 

If $p(k-3) = p(-1)$ then by Lemma \ref{where-is-D} and the Serre relations, $\epch{p(k-3)}(\ftilch{p(k-3)}\ftilch{j}\Rcal{A}) = 2 =\phcyctriv{p(k-3)}{p}{k-3}$. Then we may again apply Lemma \ref{existence-of-surj} twice to \eqref{pull-off-two-remove-one} to get \eqref{second-goal-surjection}. 

Finally, in the case where $j = p(-1)$ then because $p(k-3) \neq j = p(-1)$ we can again apply Lemma \ref{existence-of-surj} to \eqref{pull-off-two-remove-one} to get
\begin{equation}
\ind \Tii{p}{k-2} \boxtimes \ftilch{j}\Rcal{A} \twoheadrightarrow \etil{j}A.
\end{equation}
Then $\epch{p(-1)}(\ftilch{j}\Rcal{A}) = \epch{p(-1)}(\ftilch{p(-1)}\Rcal{A}) = 2$ and $\phcyctriv{j}{p}{k-2} = 3$ so that by Lemma \ref{application-of-difficult-surjection} a surjection like \eqref{second-goal-surjection} exists. 
(Recall the definition of cyclotomic path does not allow
$\phcyctriv{j}{p}{k-2} = 4$.)

\item \emph{Case} 3b: $\phcyc{p(-1)}{j}(\Rcal{A}) \geq 2 = \ep{j}(\Tii{p}{k})$ 

Then $\jump{j}(\Rcal{A}) \geq 2$. This implies $\etil{j}\Rcal{A} \neq \zero$ by \eqref{surjection-echi-2}. It also follows that $\jump{j}(\ftilch{j}\Rcal{A}) \geq 1$, $\jump{j}(\etil{j}\ftilch{j}\Rcal{A}) \geq 2$ so that by Proposition \ref{commuting-functors}.\ref{etil-ftilch-same}, $\etil{j}\ftilch{j}\ftilch{j}\Rcal{A} \cong \ftilch{j}\etil{j}\ftilch{j}\Rcal{A}$. Also, $\jump{j}(\etil{j}\Rcal{A}) \geq 3$ so again by Proposition \ref{commuting-functors}.\ref{etil-ftilch-same} $\etil{j}\ftilch{j}\Rcal{A} \cong \ftilch{j}\etil{j}\Rcal{A}$. Again because $p(k-3) \neq j$, by Proposition \ref{commuting-functors}.\ref{f-and-e}, $\ftilch{p(k-3)}$ and $\etil{j}$ commute. Together these imply
\begin{equation}
\etil{j}\ftilch{p(k-3)}\ftilch{j}\ftilch{j}\Rcal{A} \cong \ftilch{p(k-3)}\etil{j}\ftilch{j}\ftilch{j}\Rcal{A} \cong \ftilch{p(k-3)}\ftilch{j}\ftilch{j}\etil{j}\Rcal{A}.
\end{equation}
Considering cases similar to those in Case 3a, applying Lemma
\ref{existence-of-surj} and Lemma
\ref{application-of-difficult-surjection} to
\eqref{surjection-echi-2-2},
we get
\begin{equation}
\ind \Tii{p}{k} \boxtimes \etil{j}\Rcal{A} \twoheadrightarrow \etil{j}A
\end{equation}
agreeing with \eqref{E2}.

This concludes the proof of Theorem \ref{thm-action-etil}
apart from observing that 
 in the case $p(k-3) = p(-1) \neq j$ it is useful to note 
\begin{equation}
\epch{p(-1)}(\ftilch{p(k-1)}\ftilch{j}\ftilch{j}\etil{j}\Rcal{A}) = \epch{p(-1)}(\ftilch{p(k-1)}\ftilch{j}\ftilch{j}\Rcal{A})
\end{equation}
by Lemma \ref{where-is-D} and Lemma \ref{ei-app-does-not-alter-algo}.

	\end{itemize}
\end{itemize}
\end{proof}

\begin{example}
\label{ex-modC}
Recall that in type $C_\ell^{(1)}$,
$B^{1,1} \otimes B(\Lambda_1) \simeq B(\Lambda_0) \oplus B(\Lambda_2)$.
We illustrate the module-theoretic phenomena that mirrors this fact.
It holds that
\begin{gather}
\Lii{1} \in \repL{1}
\qquad
\ftilch{0} \Lii{1} \simeq \ftil{1} \Lii{0} \simeq  \Lcal{01} \in \repL{0}
\qquad
\ftilch{2} \Lii{1} \simeq \Lcal{21} \in \repL{2}.
\end{gather}
In other words, modules $A$ both from $\repL{0}$ and $\repL{2}$
can lead to modules  $\Rcal{A} \in \repL{1}$ when one performs
the algorithm of Theorem \ref{exist-theorem}.
Regarding
Example \ref{ex-typeC}, Remark \ref{rem-typeC},
and Theorem \ref{thm-action-etil},
this corresponds to the fact that
\begin{gather}
\ftil{1} \ftil{0} \bi{0} = \crystalmap(\ftil{1} \ftil{0}
(\begin{tikzpicture}[baseline=-2pt]
        \node at (0,0) {\;\;{\scriptsize{$\overline 1$}}\;\;};
        \draw (0,0) ellipse (.38cm and .2cm);
\end{tikzpicture} \; \otimes\; \bi{1} )) =
\crystalmap(\begin{tikzpicture}[baseline=-2pt]
        \node at (0,0) {\;\;{\scriptsize{$1$}}\;\;};
        \draw (0,0) ellipse (.38cm and .2cm);
\end{tikzpicture}  \; \otimes\; \bi{1} )
\end{gather}
whereas
\begin{gather}
\ftil{1} \ftil{2} \bi{2} = \crystalmap(\ftil{1} \ftil{2}
(\begin{tikzpicture}[baseline=-2pt]
        \node at (0,0) {\;\;{\scriptsize{$1$}}\;\;};
        \draw (0,0) ellipse (.38cm and .2cm);
\end{tikzpicture} \; \otimes\; \bi{1} )) =
\crystalmap(\begin{tikzpicture}[baseline=-2pt]
        \node at (0,0) {\;\;{\scriptsize{$2$}}\;\;};
        \draw (0,0) ellipse (.38cm and .2cm);
\end{tikzpicture}  \; \otimes\; \ftil{1} \bi{1} ).
\end{gather}
A more trivial example of the same phenomenon is
given by
$\UnitModule \in \repL{1}$,
$\ftilch{0} \UnitModule  \simeq  \Lii{0}  \in \repL{0}$,
$\ftilch{2} \UnitModule  \simeq  \Lii{2}  \in \repL{2}$.
This corresponds to
$ \ftil{0} \bi{0} = \crystalmap( \ftil{0}
(\begin{tikzpicture}[baseline=-2pt]
        \node at (0,0) {\;\;{\scriptsize{$\overline 1$}}\;\;};
        \draw (0,0) ellipse (.38cm and .2cm);
\end{tikzpicture} \; \otimes\; \bi{1} )) =
\crystalmap(\begin{tikzpicture}[baseline=-2pt]
        \node at (0,0) {\;\;{\scriptsize{$0$}}\;\;};
        \draw (0,0) ellipse (.38cm and .2cm);
\end{tikzpicture}  \; \otimes\; \bi{1} )$
and 
$\ftil{2} \bi{2} = \crystalmap( \ftil{2}
(\begin{tikzpicture}[baseline=-2pt]
        \node at (0,0) {\;\;{\scriptsize{$1$}}\;\;};
        \draw (0,0) ellipse (.38cm and .2cm);
\end{tikzpicture} \; \otimes\; \bi{1} )) =
\crystalmap(\begin{tikzpicture}[baseline=-2pt]
        \node at (0,0) {\;\;{\scriptsize{$2$}}\;\;};
        \draw (0,0) ellipse (.38cm and .2cm);
\end{tikzpicture}  \; \otimes\; \bi{1} )$.

\end{example}

\section{Appendix} \label{Appendix-Section}
\subsection{Simple modules in
{\protect$\rep{\Lambda_{h} + \Lambda_{i}}$}   in rank $2$}

In Figure \ref{Simple-R(nalphai+malphaj)-modues-without-labels}
we draw the highest weight crystal $B(\Lambda_h + \Lambda_i)$ in type $B_2/C_2$. That is, the associated Cartan matrix is
\begin{equation}
\begin{pmatrix}
2 & -2 \\
-1 & 2 \\
\end{pmatrix}
\end{equation}
with $a_{hi} = -1$ and $a_{ih} = -2$.

In Figure \ref{Simple-R(nalphai+malphaj)-modues} we redraw this crystal
with nodes the corresponding simple
$R(m_h\alpha_h + m_i\alpha_i)$-modules in $\rep{\Lambda_h + \Lambda_i}$,
$m_h, m_i \geq 0$. The
cyclotomic condition $\cycloI{\Lambda_h + \Lambda_i}M = \zero$ forces
$m_h + m_i \leq 7$. Indeed we can see this from Figure
\ref{Simple-R(nalphai+malphaj)-modues-without-labels} or Figure
\ref{Simple-R(nalphai+malphaj)-modues}. 

In Figure \ref{Simple-R(nalphai+malphaj)-modues}
we indicate $\jump{i}(M), \jump{h}(M)$, in part so one
can see when $\ftil{*}M \in \rep{\Lambda_i  +\Lambda_h}$.
Furthermore
when $\jump{h}(M) = 0$ we draw $M
\begin{tikzpicture}
\draw[ultra thick,->,blue] (0,0)--(.6,0);
\node at (.3,.25) {\scriptsize{$\ftil{h}$}};
\node at (0,-.01) {};
\end{tikzpicture}
\ind M \boxtimes \Lii{h}$ (similarly for $i$) to mark the fact $\ftil{h}M \cong \ftilch{h}M$ is irreducibly induced. In other cases when modules can be realized as induced from other modules (without needing to take cosocle) we also express the module in this form. The only exception to this is the 2-dimensional module $\ftil{h}\Lcal{hii}$, whose character has support $[hiih]$.

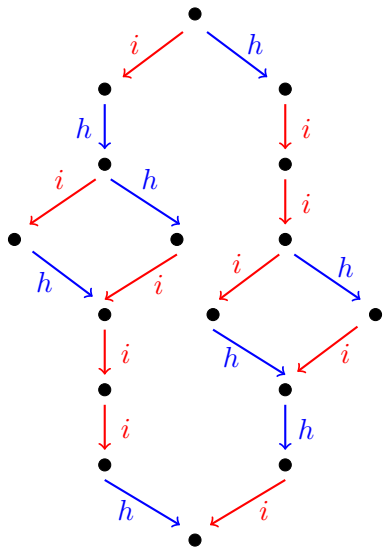
\begin{figure}
\begin{center}
\begin{tikzpicture} [scale=.4]
\draw[fill=black] (0,12) circle (.2cm);
\draw[fill=black] (-3,9.5) circle (.2cm);
\draw[fill=black] (3,9.5) circle (.2cm);
\draw[fill=black] (-3,7) circle (.2cm);
\draw[fill=black] (3,7) circle (.2cm);
\draw[fill=black] (-6,4.5) circle (.2cm);
\draw[fill=black] (-.6,4.5) circle (.2cm);
\draw[fill=black] (3,4.5) circle (.2cm);
\draw[fill=black] (-3,2) circle (.2cm);
\draw[fill=black] (6,2) circle (.2cm);
\draw[fill=black] (.6,2) circle (.2cm);
\draw[fill=black] (-3,-.5) circle (.2cm);
\draw[fill=black] (3,-.5) circle (.2cm);
\draw[fill=black] (-3,-3) circle (.2cm);
\draw[fill=black] (3,-3) circle (.2cm);
\draw[fill=black] (0,-5.5) circle (.2cm);
\draw[thick,->,red] (-.4,11.4)--(-2.4,9.9); 
\node at (-2,11) {$\redj$};
\draw[thick,->,blue] (.4,11.4)--(2.4,9.9); 
\node at (2,11) {$\blueh$};
\draw[thick,->,red] (3,9)--(3,7.5); 
\node at (-3.7,8.2) {$\blueh$};
\draw[thick,->,blue] (-3,9)--(-3,7.5); 
\node at (3.7,8.2) {$\redj$};
\draw[thick,->,red] (3,6.5)--(3,5); 
\node at (3.7,5.7) {$\redj$};
\draw[thick,->,blue] (-2.8,6.5)--(-.5,5); 
\node at (-1.5,6.5) {$\blueh$};
\draw[thick,->,red] (-3.3,6.5)--(-5.5,5); 
\node at (-4.5,6.5) {$\redj$};
\draw[thick,->,blue] (-5.4,4.1)--(-3.4,2.6);
\node at (-5,3.1) {$\blueh$};
\draw[thick,->,blue] (3.3,4)--(5.5,2.5); 
\node at (5,3.6) {$\blueh$};
\draw[thick,->,red] (5.4,1.6)--(3.4,.1); 
\node at (5,.7) {$\redj$};
\draw[thick,->,red] (2.8,4)--(.8,2.5); 
\node at (1.4,3.7) {$\redj$};
\node at (-1.2,3) {$\redj$};
\draw[thick,->,red] (-3,1.5)--(-3,0); 
\draw[thick,->,blue] (.6,1.5)--(3,0); 
\node at (1.2,.6) {$\blueh$};
\node at (-2.3,.7) {$\redj$};
\draw[thick,->,red] (-.6,4)--(-3,2.5); 
\draw[thick,->,blue] (3,-1)--(3,-2.5); 
\draw[thick,->,red] (-3,-1)--(-3,-2.5); 
\node at (-2.3,-1.8) {$\redj$};
\node at (3.7,-1.8) {$\blueh$};
\draw[thick,->,red] (3,-3.5)--(.5,-5); 
\draw[thick,->,blue] (-3,-3.5)--(-.5,-5); 
\node at (-2.3,-4.5) {$\blueh$};
\node at (2.3,-4.5) {$\redj$};
\end{tikzpicture}
\end{center}
\caption{\label{Simple-R(nalphai+malphaj)-modues-without-labels}
Crystal graph $B(\Lambda_{\blueh} + \Lambda_{\redj})$ in type $B_2$.}
\end{figure}

\begin{figure}
\begin{center}
\begin{tikzpicture}
\node at (0,12) {$\UnitModule$};
\node at (-3,9.5) {$\Lii{i}$};
\node at (-1.8,9.2) {\textcolor{red}{\scriptsize{$\jump{i} = 0$}}};
\node at (-1.8,8.8) {\textcolor{blue}{\scriptsize{$\jump{h} = 1$}}};
\node at (3,9.5) {$\Lii{h}$};
\node at (4.2,9.2) {\textcolor{red}{\scriptsize{$\jump{i} = 2$}}};
\node at (4.2,8.8) {\textcolor{blue}{\scriptsize{$\jump{h} = 0$}}};
\node at (-3,7) {$\Lcal{ih}$};
\node at (-1.6,7) {\textcolor{red}{\scriptsize{$\jump{i} = 1$}}};
\node at (-1.6,6.6) {\textcolor{blue}{\scriptsize{$\jump{h} = 0$}}};
\node at (3,7) {$\Lcal{hi}$};
\node at (4.2,6.7) {\textcolor{red}{\scriptsize{$\jump{i} = 1$}}};
\node at (4.2,6.3) {\textcolor{blue}{\scriptsize{$\jump{h} = 0$}}};
\node at (-6,4.5) {$\Lcal{ihi}$};
\node at (-6.9,4.1) {\textcolor{red}{\scriptsize{$\jump{i} = 0$}}};
\node at (-6.9,3.7) {\textcolor{blue}{\scriptsize{$\jump{h} = 0$}}};
\node at (-.6,4.5) {$\ind \Lcal{ih} \boxtimes \Lii{h}$};
\node at (.3,4.1) {\textcolor{red}{\scriptsize{$\jump{i} = 3$}}};
\node at (.3,3.7) {\textcolor{blue}{\scriptsize{$\jump{h} = 0$}}};
\node at (3,4.5) {$\Lcal{hii}$};
\node at (4.6,4.7) {\textcolor{red}{\scriptsize{$\jump{i} = 0$}}};
\node at (4.6,4.3) {\textcolor{blue}{\scriptsize{$\jump{h} = 1$}}};
\node at (-3,2) {$\ind \Lcal{ihi}\boxtimes \Lii{h}$};
\node at (-4.4,1.6) {\textcolor{red}{\scriptsize{$\jump{i} = 2$}}};
\node at (-4.4,1.2) {\textcolor{blue}{\scriptsize{$\jump{h} = 0$}}};
\node at (6,2) {$\ftil{h}\Lcal{hii}$};
\node at (-.3,1.6) {\textcolor{red}{\scriptsize{$\jump{i} = 0$}}};
\node at (-.3,1.2) {\textcolor{blue}{\scriptsize{$\jump{h} = 2$}}};
\node at (.6,2) {$\ind \Lcal{hii} \boxtimes \Lii{i}$};
\node at (6.8,1.6) {\textcolor{red}{\scriptsize{$\jump{i} = 0$}}};
\node at (6.8,1.2) {\textcolor{blue}{\scriptsize{$\jump{h} = 0$}}};
\node at (-3,-.5) {$\ind \Lcal{ihi} \boxtimes \Lcal{hi}$};
\node at (-4.4,-.9) {\textcolor{red}{\scriptsize{$\jump{i} = 1$}}};
\node at (-4.4,-1.2) {\textcolor{blue}{\scriptsize{$\jump{h} = 0$}}};
\node at (3,-.5) {$\ind \ftil{h}\Lcal{hii} \boxtimes \Lii{i}$};
\node at (1.6,-1) {\textcolor{red}{\scriptsize{$\jump{i} = 0$}}};
\node at (1.6,-1.3) {\textcolor{blue}{\scriptsize{$\jump{h} = 1$}}};
\node at (-3,-3) {$\ind \Lcal{ihi} \boxtimes \Lcal{hii}$};
\node at (4.4,-3.4) {\textcolor{red}{\scriptsize{$\jump{i} = 1$}}};
\node at (4.4,-3.7) {\textcolor{blue}{\scriptsize{$\jump{h} = 0$}}};
\node at (3,-3) {$\ind \ftil{h}\Lcal{hii} \boxtimes \Lcal{ih}$};
\node at (-4.4,-3.4) {\textcolor{red}{\scriptsize{$\jump{i} = 0$}}};
\node at (-4.4,-3.7) {\textcolor{blue}{\scriptsize{$\jump{h} = 1$}}};
\node at (0,-5.5) {$\ind \ftil{h}\Lcal{hii} \boxtimes \Lcal{ihi}$};
\node at (2.8,-5.3) {\textcolor{red}{\scriptsize{$\jump{i} = 0$}}};
\node at (2.8,-5.7) {\textcolor{blue}{\scriptsize{$\jump{h} = 0$}}};
\node at (-6,11) {\scriptsize{$\jump{\blueh}$ (source) = 0}};
\node at (-6,10.2) {\scriptsize{$\jump{\redj}$ (source) = 0}};
\draw[ultra thick,->,red] (-8,10.6)--(-4,10.6);  
\draw[ultra thick,->,blue] (-8,11.4)--(-4,11.4);
\draw[ultra thick,->,red] (-.4,11.4)--(-2.4,9.9); 
\node at (-2,11) {$\ftil{\redj}$};
\draw[ultra thick,->,blue] (.4,11.4)--(2.4,9.9); 
\node at (2,11) {$\ftil{\blueh}$};
\draw[thick,->,red] (3,9)--(3,7.5); 
\node at (-3.5,8.2) {$\ftil{\blueh}$};
\draw[thick,->,blue] (-3,9)--(-3,7.5); 
\node at (3.5,8.2) {$\ftil{\redj}$};
\draw[thick,->,red] (3,6.5)--(3,5); 
\node at (3.5,5.7) {$\ftil{\redj}$};
\draw[ultra thick,->,blue] (-2.8,6.5)--(-.5,5); 
\node at (-2.5,5.7) {$\ftil{\blueh}$};
\draw[thick,->,red] (-3.3,6.5)--(-5.5,5); 
\node at (-5,6) {$\ftil{\redj}$};
\draw[ultra thick,->,blue] (-5.4,4.1)--(-3.4,2.6);
\node at (-5,3.1) {$\ftil{\blueh}$};
\draw[thick,->,blue] (3.3,4)--(5.5,2.5); 
\node at (5,3.4) {$\ftil{\blueh}$};
\draw[ultra thick,->,red] (5.4,1.6)--(3.4,.1); 
\node at (5,.7) {$\ftil{\redj}$};
\draw[ultra thick,->,red] (2.8,4)--(.8,2.5); 
\node at (2.5,3.2) {$\ftil{\redj}$};
\node at (-2.5,3.2) {$\ftil{\redj}$};
\draw[thick,->,red] (-3,1.5)--(-3,0); 
\draw[thick,->,blue] (.6,1.5)--(3,0); 
\node at (2.5,.7) {$\ftil{\blueh}$};
\node at (-2.5,.7) {$\ftil{\redj}$};
\draw[thick,->,red] (-.6,4)--(-3,2.5); 
\draw[thick,->,blue] (3,-1)--(3,-2.5); 
\draw[thick,->,red] (-3,-1)--(-3,-2.5); 
\node at (-2.5,-1.8) {$\ftil{\redj}$};
\node at (3.5,-1.8) {$\ftil{\blueh}$};
\draw[thick,->,red] (3,-3.5)--(.5,-5); 
\draw[thick,->,blue] (-3,-3.5)--(-.5,-5); 
\node at (-3,-4.2) {$\ftil{\blueh}$};
\node at (3,-4.2) {$\ftil{\redj}$};
\end{tikzpicture}
\end{center}
\caption{\label{Simple-R(nalphai+malphaj)-modues}
The crystal graph
$B(\Lambda_{\blueh} + \Lambda_{\redj})$ in type $B_2$ with nodes
labeled by corresponding simple modules in $\rep{\Lambda_h + \Lambda_i}$
and jump indicated. When $\jump{\redj}(M) = 0$
($\jump{\blueh}(M) = 0$) we indicate with a thick ${\redj}$-arrow
(respectively ${\blueh}$-arrow).}
\end{figure}
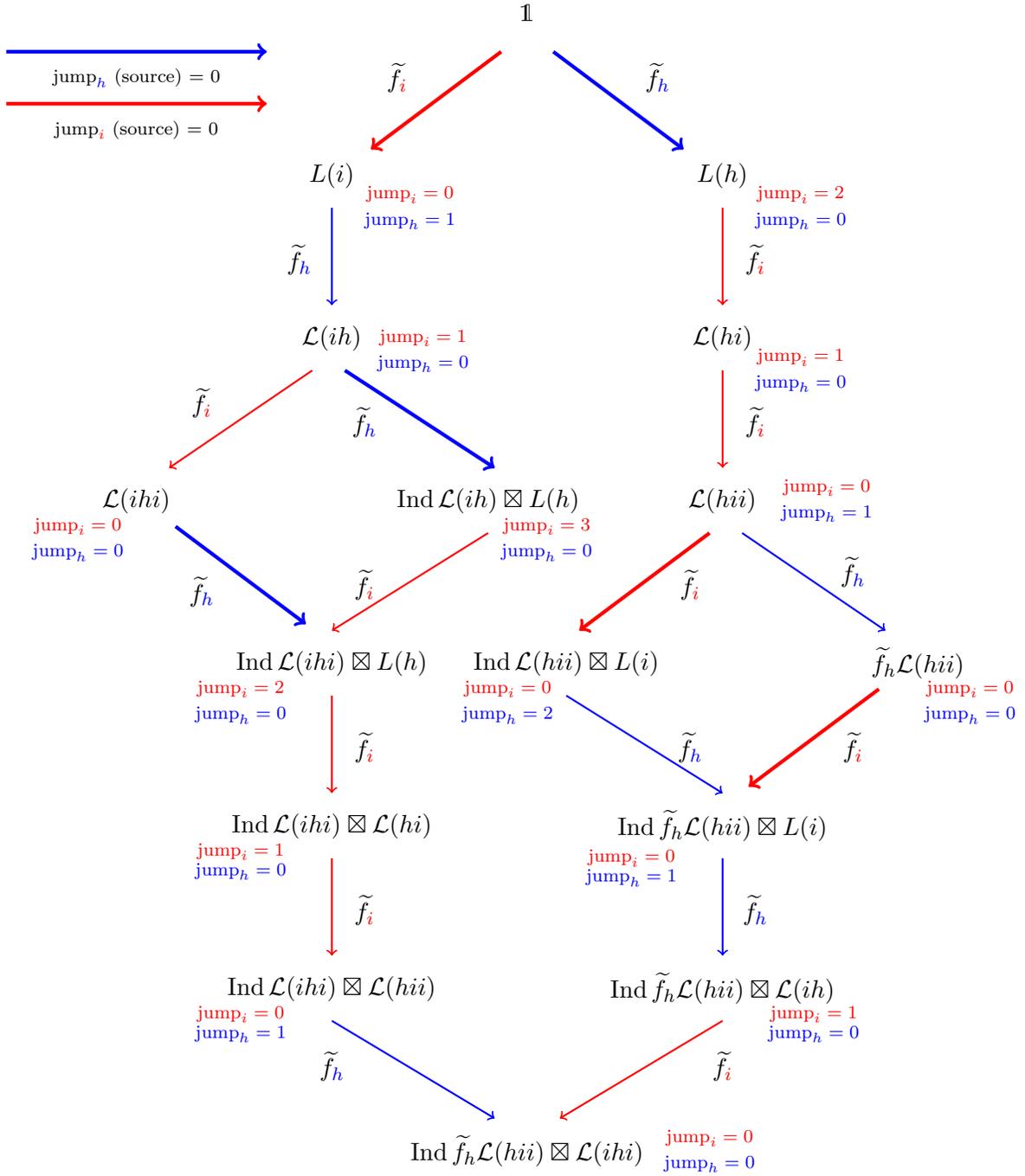

\bibliographystyle{amsplain}

\def\cprime{$'$} \def\cprime{$'$}
\providecommand{\bysame}{\leavevmode\hbox to3em{\hrulefill}\thinspace}
\providecommand{\MR}{\relax\ifhmode\unskip\space\fi MR }
\providecommand{\MRhref}[2]{%
  \href{http://www.ams.org/mathscinet-getitem?mr=#1}{#2}
}
\providecommand{\href}[2]{#2}

\end{document}